\documentclass[11pt]{article}
\usepackage{enumitem}
\setlist[itemize]{leftmargin=*}
\usepackage{a4wide}
\usepackage[utf8]{inputenc}
\usepackage{amstext}
\usepackage{amsmath} 
\usepackage{amssymb}
\usepackage{graphicx} 
\usepackage{array}
\usepackage{epsfig}
\usepackage{psfrag}
\usepackage{todonotes}
\usepackage{stmaryrd}
\usepackage[update,prepend]{epstopdf}

\numberwithin{equation}{section}

\newcommand{\bfI}{\boldsymbol I}

\newcommand{\bfV}{\boldsymbol V}

\newcommand{\mcA}{\mathcal{A}}

\newcommand{\mcF}{\mathcal{F}}

\newcommand{\mcK}{\mathcal{K}}

\newcommand{\mcN}{\mathcal{N}}

\newcommand{\mcT}{\mathcal{T}}

\newcommand{\mcX}{\mathcal{X}}

\newcommand{\mcu}{{H}}

\newcommand{\lp}{\left(}
\newcommand{\rp}{\right)}

\newcommand{\IR}{\mathbb{R}}

\newcommand{\Gammah}{{\Gamma_h}}
\newcommand{\mcKh}{\mcK_h}        

\newcommand{\mcTh}{\mcT_h}

\newcommand{\nablash}{\nabla_{\Gamma_h}}

\newcommand{\Ps}{{P}_\Gamma}
\newcommand{\Psh}{{P}_\Gammah}

\newtheorem{remark}{Remark}[section]

\newtheorem{lem}{Lemma}[section]

\newtheorem{thm}{Theorem}[section]
\newtheorem{rem}{Remark}[section]

\newenvironment{proof}{\noindent \newline {\bf Proof.}}
{\hfill \mbox{\fbox{} } \newline}

\begin{document}
\title{\bf Stabilization of High Order Cut Finite Element Methods on Surfaces\thanks{This research was supported in part by the Swedish Foundation for Strategic Research Grant No.\ AM13-0029, the Swedish Research Council Grants Nos. 2013-4708 and 2014-4804, and the Swedish Research Programme Essence}}
\author{Mats G. Larson
\footnote{Department of Mathematics and Mathematical Statistics, Ume{\aa} University, SE-90187 Ume{\aa}, Sweden} 
 \mbox{ }  
Sara Zahedi
\footnote{Department of Mathematics, KTH Royal Institute of Technology,
SE-100 44 Stockholm, Sweden}
}
\date{}
\numberwithin{equation}{section} 
\maketitle
%%%%%%%%%%%%%%%%%%%%%%%%%%%%%%%%%%%%%%%%%%%%%%%%%%%%%%%%%%%%%%%%%%%%%%%
\begin{abstract}
We develop and analyze a stabilization term for cut finite element
approximations of an elliptic second order partial differential
equation on a surface embedded in $\IR^d$. The new stabilization term
combines properly scaled normal derivatives at the surface together
with control of the jump in the normal derivatives across faces and 
provides control of the variation of the finite element solution on the 
active three dimensional elements that intersect the surface.
We show that the condition number of the stiffness matrix 
is $O(h^{-2})$, where $h$ is the mesh parameter. The
stabilization term works for linear as well as for higher-order elements and the derivation of its stabilizing properties is quite straightforward, which we illustrate 
by discussing the extension of the analysis to general $n$-dimensional smooth manifolds embedded in $\IR^d$, with codimension $d-n$.  We also formulate properties of a general stabilization term that are sufficient to prove optimal scaling of the condition number and optimal error estimates in energy- and $L^2$-norm.  We finally present numerical studies confirming our theoretical results.
\end{abstract} 

\section{Introduction} 
 \paragraph{CutFEM for Surface Partial Differential Equations.}
CutFEM (Cut Finite Element Method) provides a new high order finite element 
method for solution of partial 
differential equations on surfaces. The main approach is to embed the surface into 
a three dimensional mesh and to use the restriction of the finite element functions 
to the surface as trial and test functions. More precisely, let $\Gamma$ be a smooth 
closed surface embedded in $\IR^3$ with exterior unit normal $n$ and consider the 
Laplace-Beltrami problem: find $u\in H^1(\Gamma)$ such that
\begin{equation}
a(u,v) = l(v) \qquad \forall v \in H^1(\Gamma).
\end{equation}
The forms are defined by
\begin{equation}
a(w,v) = \int_\Gamma \nabla_\Gamma w\cdot  \nabla_\Gamma v,\qquad l(v) = \int_\Gamma fv, 
\end{equation}
where $\nabla_\Gamma$ is the surface gradient and $f\in L^2(\Gamma)$ a given function 
with average zero. Let now $\Gamma$ be embedded into a polygonal domain $\Omega$ 
equipped with a quasi-uniform mesh $\mcT_{h,0}$ 
and let $\Gamma_h$ be a family of discrete surfaces that converges to $\Gamma$ in an 
appropriate manner. Let the active mesh 
$\mcT_h$ consist of all elements that intersect $\Gamma_h$. Let $V_h$ be the space 
of continuous piecewise polynomial functions on $\mcT_h$ with average zero on $\Gammah$. 
The basic CutFEM takes the form: find $u_h \in V_h$ such that
\begin{equation}\label{eq:basic-cutfem}
a_h(u_h,v) = l_h(v) \qquad \forall v \in V_h,
\end{equation}
where 
\begin{equation}
a_h(w,v) = \int_{\Gamma_h} \nabla_\Gamma w\cdot  \nabla_\Gamma v, 
\qquad l_h(v) = \int_{\Gamma_h} f_h v, 
\end{equation}
and $f_h$ is a suitable representation of $f$ on $\Gamma_h$.
In the context of surface partial differential equations this approach, also called trace finite 
elements, was first introduced in \cite{OlReGr09} and has then been developed 
in different directions including high order approximations in \cite{Reu15}, discontinuous 
Galerkin methods \cite{BurHanLarMas17}, transport problems 
\cite{OlsReuXu14-b}, embedded membranes \cite{CenHanLar16}, coupled bulk-surface 
problems \cite{BurHanLarZah16} and \cite{GroOlsReu15}, minimal surface problems 
\cite{CenHanLar15}, and time 
dependent problems on evolving surfaces \cite{HanLarZah16}, \cite{OlsReu14}, \cite{OlsReuXu14}, and \cite{Zah17}.
We also refer to the overview article \cite{BurClaHanLarMas15} and the references therein.

\paragraph{Previous Work on Preconditioning and Stabilization.}
The basic CutFEM method (\ref{eq:basic-cutfem})  manufactures a potentially ill conditioned 
linear system of equations  and stabilization or preconditioning is therefore needed. In \cite{OlsReu10} 
it was shown that diagonal preconditioning works for piecewise linear approximation and in 
\cite{Reu15} it was shown that also quadratic elements can be preconditioned if the full gradient 
is used instead of the tangential gradient in the bilinear form $a_h$, see also \cite{DecEllRan14} 
where this formulation was proposed.

In \cite{BurHanLar15} a stabilized version of (\ref{eq:basic-cutfem}) was introduced. The method 
takes the form
\begin{equation}\label{eq:stabilized-cutfem}
a_h(u_h,v) + s_h(u_h,v) = l(v) \qquad \forall v \in V_h,
\end{equation}
where $s_h$ is a stabilization form which is added in order to ensure stability of the method. 
More precisely we show that there is a constant such that for all $v \in V_h$,
\begin{equation}\label{eq:stab-est}
\| v \|^2_{L^2(\mcT_h)} \lesssim  h ( a_h(v,v) + s_h(v,v) ). 
\end{equation}
Using (\ref{eq:stab-est}) and the properties of $a_h$ we can show that the condition 
number $\kappa$ of the stiffness matrix satisfies the optimal bound 
\begin{equation}
\kappa \lesssim h^{-2}.
\end{equation}

In \cite{BurHanLar15} only piecewise linear finite elements were considered and 
the stabilization term was defined by
\begin{align} \label{eq:stabface}
s_{h}(w,v)&= c_{F}  ([D_{n_F}^1 w], [D_{n_F}^1 v])_{\mcF_h}, 
\end{align}
where $[D_{n_F}^1 v]$ is the jump in the normal derivative of $v$ at the face $F$, 
$\mcF_{h}$ is the set of internal faces (i.e. faces with two neighbors) in the active 
mesh $\mcT_h$, and $c_{F}>0$ is a stabilization parameter. Optimal order bounds on 
the error and condition number were established. Furthermore, for linear elements 
the so called full gradient stabilization 
\begin{equation} \label{eq:fullgrad}
s_h(v,w) = ch(\nabla v,\nabla w)_{\mcT_h}
\end{equation}
was proposed and analyzed in \cite{BurHanLarMasZah16} as a simpler alternative, and for 
higher order elements the normal gradient stabilization 
\begin{equation}\label{eq:normalgradel}
s_h(w,v) = c_{\mcT}h^{\alpha} (n_h \cdot \nabla w,n_h \cdot \nabla v)_{\mcT_h}, \qquad \alpha \in [-1,1],
\end{equation}
was proposed and analysed independently in \cite{BurHanLarMas16} and \cite{GraLehReu16}. 
The restrictions on $\alpha$ essentially comes from a lower bound which implies that optimal 
order a priori error estimates hold and an upper bound which implies that the desired stability 
estimate (\ref{eq:stab-est}) holds. Note that stronger control than (\ref{eq:stab-est}) is sometimes 
demanded, for instance in cut discontinuous Galerkin methods \cite{BurHanLarMas17}, which 
may lead to stronger upper restrictions on $\alpha$.

The last two stabilization terms, \eqref{eq:fullgrad} and \eqref{eq:normalgradel}, can however not be used in the discretization of bulk problems without destroying the convergence order while the ghost penalty stabilization term~\cite{Bu10}, which as \eqref{eq:stabface} acts on the element faces, can be used. Thus, in many applications where both bulk and surface problems occur face stabilization is still needed when discretizing the problem with CutFEM. It should be noted that the stabilization term~\eqref{eq:stabface} has been used also for other reasons than controlling the condition number. It is used for example in~\cite{BuHaLaZa15} to stabilize for convection and in~\cite{HanLarZah15} to improve on the accuracy of computing the mean curvature vector of a surface. We study the problem of computing the mean curvature vector of a surface based on the Laplace-Beltrami operator in Section~\ref{sec:meancurv}.

\paragraph{New Contributions.}
In this work we consider stabilization for linear as well as higher-order elements using a combination of 
face stabilization and normal derivative stabilization at the actual surface. More precisely 
the stabilization term takes the form 
\begin{equation}
s_h (v,w) = s_{h,F}(v,w)+s_{h,\Gamma}(v,w)
\end{equation}
with
\begin{align} 
s_{h,F}(v,w)&= \sum_{j=1}^p c_{F,j} h^{2(j-1 + \gamma)} ([D_{n_F}^j v], [D_{n_F}^j w])_{\mcF_h}, 
\\
s_{h,\Gamma}(v,w)&=\sum_{j=1}^p c_{\Gamma,j} h^{2(j-1+ \gamma)}  (D_{n_h}^j v, D_{n_h}^j w)_{\Gamma_h}, 
\end{align}
where $\gamma \in [0,1]$ is a parameter, $[D_{n_F}^j v]$ is the jump in the normal derivative 
of order $j$ across the face $F$, and $c_{F,j}>0$, $c_{\Gamma,j}>0$,  are stabilization constants. This stabilization term with $\gamma=1$ has also been used in a high order space-time CutFEM on evolving surfaces \cite{Zah17}.  
We note that also for this stabilization term there is a certain range of scaling with the mesh 
parameter $h$, and in fact the following stronger version of (\ref{eq:stab-est}) holds for $v \in V_h$
\begin{equation}
\| v \|^2_{L^2(\mcT_h)} + h^{2\gamma} \| \nabla v \|^2_{L^2(\mcT_h)} \lesssim h( a_h(v,v) + s_h(v,v) ),
\end{equation} 
where we get stronger control of the gradient for smaller $\gamma \in [0,1]$, which 
corresponds to stronger stabilization.

Using the combination of the face and surface stabilization terms the proof of (\ref{eq:stab-est}) 
consists of two steps: 
\begin{itemize}
\item Using the face
terms we can estimate the $L^2$ norm of a finite element function at an element in terms 
of the $L^2$ norm at a neighboring element which share a face and the stabilization term 
on the shared face. Repeating this procedure we can pass from any element to an element
which has a large intersection, $|T \cap \Gamma_h|\gtrsim h^2$, with $\Gamma_h$. 
\item  Using the stabilization of normal derivatives at the surface we can control a finite element 
function on elements  which have large intersection with $\Gamma_h$. 
\end{itemize}
Using shape regularity we can show that one can pass from any element in $\mcT_h$ to an 
element with a large intersection in a uniformly bounded number of steps, see \cite{BurHanLar15} 
and \cite{DemOls12}. The analysis is quite straight forward and may in fact also be extended to the 
case on an $n$-dimensional smooth manifold embedded in $\IR^d$ for general codimension 
$cd = d - n$. An important special case is a curve embedded in $\IR^3$, which may for instance 
be an intersection of two surfaces, see also \cite{BurHanLarLar18} for generalizations to so called 
mixed dimensional problems. For clarity, we consider a two dimensional surface embedded 
in $\IR^3$ throughout the paper and discuss the minor modifications necessary in a separate section. 
We also refer to  \cite{BurHanLarMas16} for general background on the analysis of problems with 
higher codimension embeddings.

\paragraph{Outline.} In Section 2 we introduce the model problem, in Section 3 we formulate 
the finite element method, in Section 4 we collect some basic preliminary results, in Section 5 
we formulate the properties of the stabilization term, in Section 6 we establish an optimal order 
condition number estimate,  in Section 7 we discuss the extension to problems on $n$-dimensional 
smooth manifolds embedded in $\IR^d$, in Section 8 we prove optimal order a priori 
error estimates, and in Section 9 we present numerical results that support our theoretical findings. In the last section, Section 10, we conclude and discuss the advantages with the proposed stabilization.

\section{The Laplace--Beltrami Problem on a Surface}

\paragraph{The Surface.}
Let $\Gamma$ be a smooth, closed, simply connected surface in $\IR^3$ with 
exterior unit normal $n$ embedded in a polygonal domain $\Omega \subset \IR^3$. 
Let $U_{\delta}(\Gamma)$ be the open tubular neighborhood of $\Gamma$ with 
thickness $\delta>0$, 
\begin{equation}\label{Udelta} 
U_{\delta}(\Gamma) = \{ x \in \IR^3 : |d(x)| < \delta \},
\end{equation}
where $d(x)$ is the signed distance function of $\Gamma$ with $d<0$ in the 
interior of $\Gamma$ and $d>0$ in the exterior. Note that $n=\nabla d$ is the 
outward-pointing unit normal on $\Gamma$. Then there is $\delta_0>0$ 
such that  for each $x \in U_{\delta_0}(\Gamma)$ the closest point projection:
\begin{equation}\label{eq:closestpoint}
p(x)=x-d(x)n(p(x))
\end{equation}
maps $x$ to a unique closest point on $\Gamma$. In particular, we require 
\begin{equation}
\delta_0 \max_{i=1,2} \|\kappa_i \|_{L^\infty(\Gamma)} < 1,
\end{equation}
where for $x \in \Gamma$, $\kappa_i$ are the principal curvatures.
We also define the extension of a function $u$ on $\Gamma$ to 
$U_{\delta_0}(\Gamma)$ by
\begin{equation} \label{eq:extendu}
u^e(x)=u(p(x)), \quad x\in U_{\delta_0}(\Gamma).
\end{equation} 

\paragraph{The Problem.}
We consider the problem: find 
$u \in H^1_0(\Gamma) = \{v \in H^1(\Gamma) \ | \ \int_\Gamma v \ ds =0 \}$ 
such that   
\begin{equation}
\label{eq:LB}
a(u, v) = l (v)  \qquad \forall v \in H^1_0(\Gamma),
\end{equation}
where 
\begin{equation}
a(u,v)=\int_\Gamma \nabla_\Gamma u \cdot \nabla_\Gamma v \ ds,  \qquad
l(v)=\int_\Gamma f v \ ds.
\end{equation}
Here $\nabla_\Gamma$ is the tangential gradient on $\Gamma$ defined 
by 
\begin{equation}
\nabla_\Gamma u=P_\Gamma \nabla u^e, 
\end{equation}
where
\begin{equation}\label{eq:Proj}
P_\Gamma=I-n \otimes n.
\end{equation} 
We used the extension of $u$ which is constant in the normal direction to $\Gamma$ 
to formally define $\nabla_\Gamma u$. However, the tangential derivative depends only 
on the values of $u$ on $\Gamma$ and does not depend on the particular choice of 
extension.

Using Lax-Milgram's Lemma we conclude that there is a unique solution to 
(\ref{eq:LB}), and we also have the elliptic regularity estimate
\begin{equation}\label{eq:elliptic-reg}
\| u \|_{H^{s+2}(\Gamma)} \lesssim \| f \|_{H^s(\Gamma)},\qquad s\in \mathbb{Z}, \ s \geq -1.
\end{equation}

\section{The Method}

\subsection{The Mesh and Finite Element Space}
\begin{itemize}
\item \emph{The background mesh and space:} Let $\mcT_{h,0}$ be a
  quasi-uniform partition of $\Omega$ with mesh parameter $h \in (0,h_0]$ 
  into shape regular tetrahedra. On $\mcT_{h,0}$ we define $V_{h,0}^p$ to 
  be the space  of continuous piecewise polynomials of degree $p \geq 1$.

\item \emph{The discrete surfaces:} Let $\Gamma_h$ be a sequence of 
surfaces that converges to $\Gamma$. We specify the precise assumptions 
on the convergence below.

\item \emph{The active mesh and the induced surface mesh:} We denote by $\mcT_h$ 
the set consisting of elements in the background mesh $\mcT_{h,0}$ that are cut by 
the discrete geometry $\Gamma_h$. These elements form the so called active mesh. 
See Fig. \ref{fig:illust} for an illustration. The restriction of the active mesh to the 
discrete surface $\Gamma_h$ manufactures an induced cut surface mesh that we 
denote by
\begin{equation}
\mcK_h = \{ K = T\cap \Gamma_h  :  T \in \mcT_h \}.
\end{equation}

\item \emph{The finite element space:} The restriction of the background space 
to the active mesh together with the condition $\int_{\Gamma_h} v \ ds_h =0$ 
defines our finite element space $V_{h}^p$,
\begin{equation}
V_h^p=\left \{v_h \in V^p_{h,0}|_{\mcT_h}  :  \int_{\Gamma_h} v \,ds_h =0 \right\}. 
\end{equation}
\end{itemize}
\begin{figure}
\centering
\includegraphics[width=0.5\textwidth]{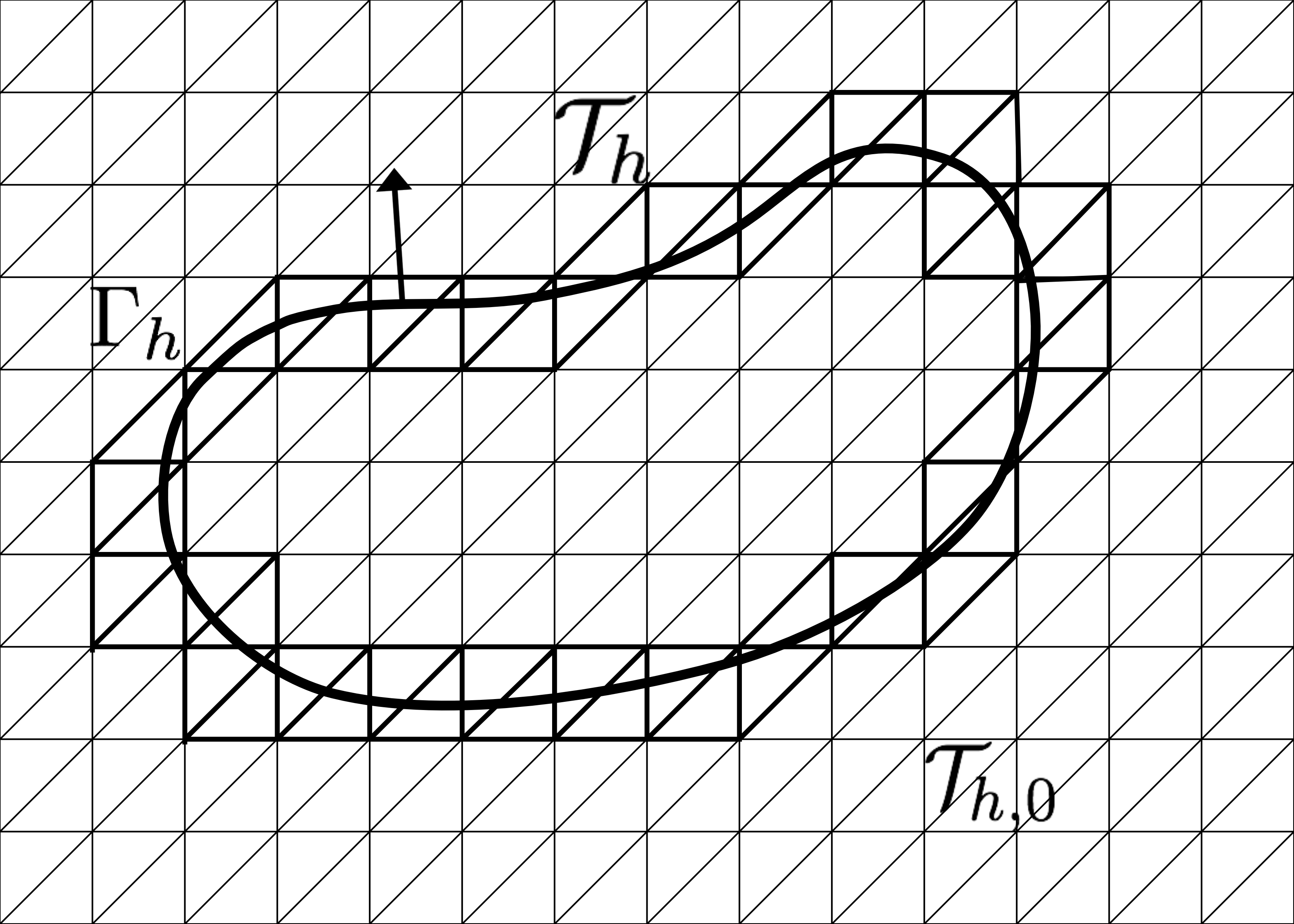}
\caption{The discrete surface $\Gamma_h$, the background mesh $\mcT_{h,0}$, and the active mesh $\mcT_h$.\label{fig:illust}}
\end{figure}

\subsection{The Finite Element Method}
\paragraph{Method.} Find $u_h \in V_h^p$ such that 
\begin{equation}\label{eq:Ah}
A_h(u_h,v) := a_h(u_h,v) + s_h(u_h,v) = l_h(v) \qquad \forall v_h \in V_h^p.
\end{equation}
\paragraph{Forms.} 
\begin{itemize}
\item The forms $a_h$ and $l_h$ are defined by
\begin{equation}
a_h(w,v) = ( \nabla_{\Gamma_h} w , \nabla_{\Gamma_h} v )_\Gammah, 
\qquad 
l_h(v) = (f_h,v)_\Gammah,
\end{equation}
where the tangential gradient $\nabla_{\Gamma_h}$ is defined by 
\begin{equation}
\nabla_{\Gamma_h} u_h=P_{\Gamma_h} \nabla u_h 
\end{equation}
with
\begin{equation}\label{eq:Proj-h}
P_{\Gamma_h}=I-n_h \otimes n_h.
\end{equation}

\item The stabilization form $s_h$ is defined by
\begin{equation}\label{eq:sh}
s_h (w,v) = s_{h,F}(w,v)+s_{h,\Gamma}(w,v),
\end{equation}
where
\begin{align}
s_{h,F}(w,v)&= \sum_{j=1}^p c_{F,j} h^{2(j-1+\gamma) } ([D_{n_F}^j w], [D_{n_F}^j v])_{\mcF_h}, 
\label{eq:sh-face} \\ 
s_{h,\Gamma}(w,v)&=\sum_{j=1}^p c_{\Gamma,j} h^{2(j-1+ \gamma)}
(D_{n_h}^j w, D_{n_h}^j v)_{\mcK_h}. \label{eq:sh-gamma}
\end{align}
Here $(D_a^j v)|_x$ is the $j$-th directional derivative of $v$ along the 
line defined by the unit vector $a(x)$, $[w]\vert_F$ denotes the jump of $w$ 
over the face $F$, $\mcF_{h}$ is the set of internal faces (i.e. faces with two neighbors) 
in the active mesh $\mcT_h$, $c_{\Gamma,j}>0$, $c_{F,j}>0$ are stabilization 
parameters, and $\gamma \in [0,1]$ is a parameter. 
\end{itemize}

\begin{rem}
In the special case $p=1$ we obtain 
\begin{equation}\label{eq:sh-linear}
s_h (w,v) = c_F h^{2\gamma} ([D_{n_F}^1 w], [D_{n_F}^1 v])_{\mcF_h} +
c_{\Gamma} h^{2\gamma} (D_{n_h}^1 w, D_{n_h}^1 v)_{\mcK_h}, 
\end{equation}
where $\gamma \in [0,1]$ which can be compared with the pure face penalty term 
\begin{equation}
s_h (w,v) =c_{F} ([D_{n_F}^1 w], [D_{n_F}^1 v])_{\mcF_h} 
\end{equation}
developed in~\cite{BurHanLar15}. Note that the scalings with the mesh
parameter $h$ are different in the two terms and that for $\gamma>0$
the face penalty term is weaker in the stabilization term we introduce
here.
\end{rem}

\subsection{Assumptions}
\begin{enumerate}
\item[{\bf A1.}] We assume that the closed, simply connected 
approximation $\Gamma_h$ and its unit normal $n_h$ are 
such that for all $h \in (0,h_0]$, 
\begin{align}
\| d \|_{L^\infty(\Gamma_h)} &\lesssim h^{p+1}, 
\\
\|n- n_h \|_{L^\infty(\Gamma_h)} & \lesssim h^{p},
\end{align}
that $n \cdot n_h>0$ on $\Gamma_h$, and that we have a uniform bound on the curvature $\kappa_h$. 

\item[{\bf A2.}] We assume that the mesh is sufficiently small so that
  there is a constant $c$ independent of the mesh size $h$ such that
  $\mcT_h \subset U_\delta(\Gamma)$ with $\delta=ch \leq \delta_0$ and
  that the closest point mapping $p(x): \Gamma_h \rightarrow \Gamma$
  is a bijection.

\item[{\bf A3.}] We assume that the data approximation $f_h$ is such that 
$\||B|f^e-f_h \|_{L^2(\Gammah)}\lesssim h^{p+1} \|f\|_{L^2(\Gamma)}$. 
Here $ds= |B| ds_h$ (see Section \ref{sec:prel} for details).
\end{enumerate} 

\section{Preliminary Results}\label{sec:prel}

\subsection{Norms}

Given a positive semidefinite bilinear form $b$ we let $\|v \|^2_b = b(v,v)$ 
denote the associated seminorm. In particular, we will use the following norms 
on $V^p_h + H^{p+1}(\mcT_h)$,
\begin{equation}\label{eq:energinorm}
\| v \|_{a_h}^2=a_h(v,v), 
\qquad
 \| v \|_{s_h}^2=\| v \|_{s_{h,F}}^2+\| v \|_{s_{h,\Gamma}}^2,
\qquad
\| v \|_{A_h}^2= \| v \|_{a_h}^2+\| v \|_{s_h}^2. 
\end{equation}
We let $H^s(\omega)$, with $\omega \subset \IR^d$, denote the standard 
Sobolev spaces of order $s$. For $\omega \subset \Gamma$ or 
$\omega \subset \Gamma_h$  $H^s(\omega)$ denotes the surface 
Sobolev space which is based on tangential derivatives.
 
\subsection{Some Inequalities}
\begin{itemize} 
\item Under A1 and A2 there is a constant such that the following 
Poincar\'e inequality holds for $h \in (0,h_0]$, 
\begin{equation}\label{eq:poincare-Kh}
\| v \|_{L^2(\Gamma_h)} \lesssim \| v \|_{a_h} \quad \forall v \in V^p_h,
\end{equation}
see Lemma 4.1 in~\cite{BurHanLar15}.  

\item For $T\in \mcT_h$  we have the trace inequalities
\begin{equation} \label{eq:tracestand}
\| v \|^2_{L^2(\partial T)} \lesssim h^{-1} \| v \|_{L^2(T)}^2 + h \| \nabla v \|^2_{L^2(T)} \qquad v\in H^1(T), 
\end{equation} 
\begin{equation} \label{eq:trace}
\| v \|^2_{L^2(T\cap \Gammah)} \lesssim h^{-1} \| v \|_{L^2(T)}^2 + h \| \nabla v \|^2_{L^2(T)} \qquad v\in H^1(T),
\end{equation} 
where the first inequality is a standard trace inequality, see e.g.~\cite{BrSc}, 
and the second inequality is proven in \cite{HaHaLa04} and the constant 
is independent of how $\Gammah$ intersects $T$ and of $h \in (0,h_0]$. 
\item For $T \in \mcT_h$ we have the inverse inequality
\begin{equation} \label{eq:inverse}
|v|_{H^j(T)}^2\lesssim h^{-2(j-s)}\|v \|_{H^s(T)}^2  \qquad 0\leq s \leq j, \ v\in V_h^p,
\end{equation}
see \cite{BrSc}.
\end{itemize}

\subsection{Interpolation}
Let $\pi_h^p:L^2(\mcT_h) \rightarrow V_h^p$ be the Cl\'{e}ment 
interpolation operator. For all $v\in H^{p+1}(\mcT_h)$ and $T\in
\mcT_h$ the following standard estimate holds
\begin{equation}\label{eq:interpoltets}
\| v - \pi_h^p v \|_{H^m(T)} \lesssim h^{s-m} \| v \|_{H^s(\mcN_h(T))}
\qquad m\leq s \leq p+1, \quad m=0, \cdots, p+1, 
\end{equation}
where $\mcN_h(T)\subset \mcT_h$ is the union of the elements in
$\mcT_h$ which share a node with $T$. In particular, we have the stability estimate 
\begin{equation}\label{eq:stabilityinterpol}
\| \pi_h^p v \|_{H^m(\mcT_h)} \lesssim \| v \|_{H^m(\mcT_h)}.
\end{equation}

\subsection{Some Results for Extended and Lifted Functions}\label{sec:extandlift}

\paragraph{Extension.}
Recall that we define the extension of a function $u$ on $\Gamma$ to 
$U_{\delta_0}(\Gamma)$ by  $u^e(x)=u(p(x))$. For $u \in H^s(\Gamma)$, 
$s \geq 0$ the following holds
\begin{equation}\label{eq:extstab}
\|u^e\|_{H^s(U_{\delta}(\Gamma))}
\lesssim \delta^{1/2}\|u\|_{H^s(\Gamma)}, 
\end{equation}
where $0<\delta \leq \delta_0$ and for $h$ small enough we may take $\delta \sim h$. 
See Lemma 3.1 in \cite{Reu15}.

Using the definition of the extension of a function $u$ on $\Gamma$, the definition 
of the closest point projection \eqref{eq:closestpoint}, the chain rule, and the identity $n=\nabla d$ 
we obtain 
\begin{equation}\label{eq:derext}
\nabla u^e = (\Ps-dH) \nabla u \qquad \text{in $U_{\delta_0}(\Gamma)$},
\end{equation}
where $H(x)=D^2d(x)$. For $x \in U_{\delta_0}(\Gamma)$, the eigenvalues of $H(x)$ 
are $\kappa_1(x)$, $\kappa_2(x)$, and $0$, with corresponding orthonormal eigenvectors 
$a_1(x)$, $a_2(x)$, and $n^e(x)$. We have the identities  
\begin{equation}
\kappa_i(x) = \frac{\kappa_i^e(x)}{1 + d(x)\kappa_i^e(x)},\qquad a_i(x) = a_i^e(x),  \qquad \ i=1,2
\end{equation}
and $H(x)$ may be expressed in the form
\begin{equation}\label{Hform}
H(x) = \sum_{i=1}^2 \frac{\kappa_i^e(x)}{1 + d(x)\kappa_i^e(x)} a_i^e(x) \otimes a_i^e(x),
\end{equation}
see \cite[Lemma 14.7]{GiTr83}. 

Using the fact that $H$ is tangential, $\Ps H=H\Ps=H$, in the identity \eqref{eq:derext} we obtain
\begin{equation}\label{eq:derext2}
\nabla u^e = (I-dH)\nabla_\Gamma u. 
\end{equation}
Thus, 
\begin{equation}\label{eq:tanderext}
\nabla_{\Gamma_h} u^e = B^T\nabla_\Gamma u  
\end{equation}
with 
\begin{equation}\label{Bmap}
B = \Ps(I -  dH)\Psh. 
\end{equation}

\paragraph{Lifting.}
Define the lifting $v^l$ of a function $v$ on $\Gamma_h$ to $\Gamma$ as
\begin{equation}
(v^l)^e = v^l \circ p = v \quad \text{on $\Gamma_h$.}
\end{equation}
Using equation \eqref{eq:tanderext} it follows that
\begin{equation}
\nabla_{\Gamma_h} v = \nabla_{\Gamma_h} (v^l)^e = B^T \nabla_\Gamma (v^l) 
\end{equation}
and thus we have 
\begin{equation}\label{eq:tanderlift}
 \nabla_\Gamma (v^l) = B^{-T} \nabla_{\Gamma_h} v. 
\end{equation}
Here,  
\begin{equation}\label{Binv}
B^{-1} = \Big(I-\frac{n \otimes n_h}{n \cdot n_h}\Big)(I -  dH)^{-1}, 
\end{equation}
see \cite{DeDz07}, and using the fact that $n$ and $n_h$ are unit normals 
it is easy to show that 
\begin{equation}
 B B^{-1} = \Ps, \qquad B^{-1} B = \Psh.
 \end{equation}

\paragraph{Estimates Related to $\boldsymbol{B}$.}
It can be shown, see \cite{BurHanLar15} for instance, that the following estimates hold
\begin{equation}
\label{B-uniform-bounds}
  \| B \|_{L^\infty(\Gamma_h)} \lesssim 1,
 \qquad \| B^{-1} \|_{L^\infty(\Gamma)} \lesssim 1,
 \qquad  \| BB^T-I\|_{L^\infty(\Gamma_h)} \lesssim h^{p+1}.
\end{equation}

\paragraph{Surface Measures.}
For the surface measure on $\Gammah$, $ds_h(x)$, and the surface measure on 
$\Gamma$, $ds(p(x))$, we have the following identity  
\begin{equation}\label{eq:measure}
ds(p(x)) = |B(x)| ds_h(x), \qquad x\in \Gamma_h,
\end{equation}
where 
\begin{equation}
|B(x)|=n \cdot n_h(1-d(x)\kappa_1(x))(1-d(x)\kappa_2(x)),
\end{equation}
see Proposition A.1 in \cite{DeDz07}. Using the identity $\|n(x)-n_h(x)\|_{\IR^3}^2 = 2(1-n(x) \cdot n_h(x))$ 
and our assumptions we obtain the following estimates:
\begin{equation}\label{eq:B-detbound}
\left\| 1 - |B| \right\|_{L^\infty(\Gamma_h)} \lesssim h^{p+1}, 
\qquad \left\| |B| \right\|_{L^\infty(\Gamma_h)} \lesssim 1,
\qquad \left\| |B|^{-1} \right\|_{L^\infty(\Gamma_h)} \lesssim 1.
\end{equation}

\paragraph{Norm Equivalences.}

Using the identities \eqref{eq:tanderext} and \eqref{eq:tanderlift} and the bounds in equations \eqref{B-uniform-bounds} and \eqref{eq:B-detbound} we obtain the following equivalences
\begin{equation}\label{eq:normequ}
\| v^l \|_{L^2(\Gamma)}
\sim \| v \|_{L^2(\Gammah)}, \qquad \| u \|_{L^2(\Gamma)} \sim
\| u^e \|_{L^2(\Gammah)},
\end{equation}
and
\begin{equation}\label{eq:normequgrad}
\| \nabla_\Gamma v^l \|_{L^2(\Gamma)} \sim \| \nabla_{\Gamma_h} v \|_{L^2(\Gammah)},
\qquad \| \nabla_\Gamma u \|_{L^2(\Gamma)} \sim
\| \nabla_{\Gamma_h} u^e \|_{L^2(\Gammah)}.
\end{equation}
We will also use the following bound
\begin{equation}\label{eq:normineqder}
\|D^ju^e\|_{L^2(\Gamma_h)} \lesssim  \|u\|_{H^m(\Gamma)} \qquad u\in H^m(\Gamma), \ j \leq m.
\end{equation} 
The constant in~\eqref{eq:normineqder} depends on the higher-order derivatives of the geometry.

\section{Properties of the Stabilization Term}\label{sec:condonstab}

Here we formulate properties of a general stabilization term that are
sufficient to prove that the resulting linear system of equations have
an optimal scaling of the condition number with the mesh parameter
and that the convergence of the method is of optimal order.

\paragraph{Properties.} Let $s_h$ be a semi positive definite bilinear form such that
\begin{enumerate}

\item[{\bf P1.}] For $v \in H^{p+1}(\Gamma)\cap H_0^1(\Gamma)$,
\begin{equation}\label{eq:property1}
\| v^e- \pi_h^p v^e\|_{s_h} \lesssim h^{p} \| v \|_{H^{p+1}(\Gamma)}. 
 \end{equation}
 
\item[{\bf P2.}] 
For $v \in H^{p+1}(\Gamma)\cap H_0^1(\Gamma)$,
 \begin{equation}\label{eq:property2}
 \|v^e\|_{s_h} \lesssim h^{p}  \| v  \|_{H^{p+1}(\Gamma)}. 
 \end{equation}

\item[{\bf P3.}]  For $v \in H^2(\Gamma)$, 
\begin{equation}\label{eq:property3}
\|\pi_h^p v^e\|_{s_h} \lesssim  h\|v \|_{H^{2}(\Gamma)}.
\end{equation}

\item[{\bf P4.}] For $ v \in V_h^p$, 
\begin{equation}
\| v \|_{s_h}^2 \lesssim h^{-3}\| v \|_{L^2(\mcT_h)}^2.
\end{equation}
 
\item[{\bf P5.}]  For $v \in V_h^p$, 
\begin{equation}\label{eq:property4}
\| v \|_{L^2(\mcT_h)}^2 \lesssim h (\| v \|_{a_h}^2+ \| v \|_{s_h}^2 ).
\end{equation}
\end{enumerate}
\begin{rem} The first three properties are used to establish the a priori error estimates: 
P1 is needed to show optimal interpolation error estimates, P2 is used to estimate the consistency error, and P3 is used to show optimal $L^2$-error estimates. In P3 the assumption that $v \in H^2(\Gamma)$ reflects the fact that the solution to the dual problem resides in $H^2(\Gamma)$. The two last properties are used to show the condition number estimate: P4 is the inverse inequality and P5 is the Poincar\'e inequality.  For the extension of these properties to the more general case of an $n$-dimensional smooth manifold embedded in $\IR^d$,  with codimension $cd = d - n$ see Section~\ref{sec:codim}.
\end{rem}

\paragraph{Verification of P1--P5.} 
We now show that the stabilization term defined in (\ref{eq:sh}) satisfies P1-P5. 

\begin{enumerate}
\item[{\bf P1.}] Using a standard trace inequality on the faces (equation \eqref{eq:tracestand}) and the trace inequality 
\eqref{eq:trace} for the contributions on $\Gamma_h$ we obtain  
\begin{align}
\| w \|^2_{s_h} 
&=
\sum_{j=1}^{p} \sum_{F\in \mcF_h} c_{F,j} h^{2(j-1+\gamma)} \|[D_{n_F}^{j} w ]\|^2_{L^2(F)}
\\ \nonumber
&\qquad +\sum_{j=1}^p \sum_{K \in \mcK_h}   c_{\Gamma,j}  h^{2(j-1 + \gamma)} \| D_{n_h}^j w \|^2_{L^2(K)}
 \\
&\lesssim 
\sum_{j=1}^p \sum_{T\in \mcT_h} h^{2(j-1 + \gamma) } \left(h^{-1}\|D^j w \|^2_{L^2(T)}+h\|D^{j+1} w \|^2_{L^2(T)} \right)
 \\ 
&\lesssim 
\sum_{j=0}^p \sum_{T\in \mcT_h} h^{2(j+\gamma)-1} \|D^{j+1} w \|^2_{L^2(T)} 
\\
&=
\sum_{j=0}^p  h^{2(j+\gamma)-1} \|D^{j+1} w \|^2_{L^2(\mcT_h)}. 
 \label{eq:traceonstab}
\end{align}  
Next, setting $w = v^e - \pi_h^p v^e$, using the interpolation error estimate 
(\ref{eq:interpoltets}) and the stability \eqref{eq:extstab}  of the extension operator 
with $\delta \sim h$, we obtain 
\begin{align}
\| v^e - \pi_h^p v^e \|^2_{s_h} 
&\lesssim 
\sum_{j=0}^p  h^{2(j+\gamma)-1} \|D^{j+1} (v^e - \pi_h^p v^e ) \|^2_{L^2(\mcT_h)} 
\\ 
&\lesssim 
\sum_{j=0}^p  h^{2(j+\gamma)-1} h^{2(p-j)} \|D^{j+1} v^e \|^2_{H^{p+1} (\mcT_h)} 
\\ 
& \lesssim 
 h^{2(p + \gamma)}  \| v \|^2_{H^{p+1}(\Gamma)}.
\end{align}
Finally, $h^{2(p+\gamma)} \lesssim h^{2p}$ for $\gamma \in [0,1]$ since $h\in (0,h_0]$, 
and thus P1 follows.

\item[{\bf P2.}] 
Note first that for $v \in H^{p+1}(\Gamma)$ we have $\|v^e\|_{s_{h,F}}=0$ and thus 
$\| v^e \|_{s_h} = \| v^e \|_{s_{h,\Gamma}}$. Next, subtracting  $D_{n}^j v^e=0$, using 
Assumption A1, and inequality \eqref{eq:normineqder} we obtain
\begin{align}
\|v^e\|^2_{s_h}&=\sum_{j=1}^p c_{\Gamma,j}  h^{2(j-1+\gamma)} \|D_{n_h}^j v^e\|^2_{\mcK_h}
\\
&\lesssim 
\sum_{j=1}^p h^{2(j-1+\gamma)} \|(D_{n_h}^j-D_{n}^j) v^e \|^2_{L^2(\mcK_h)} 
\\
& \lesssim \sum_{j=1}^p  h^{2(j-1+\gamma)} \|n_h-n\|^2_{L^\infty(\mcK_h)}\|D^j v^e\|^2_{L^2(\mcK_h)}
\\ 
&\lesssim 
h^{2(p+\gamma)} \|v\|^2_{H^p(\Gamma)}, 
\end{align}
where we used the estimate $\|n_h-n\|_{L^\infty(\mcK_h)}\lesssim h^p$. We conclude that 
P2 holds for $\gamma\in [0,1]$ since  $h \in (0,h_0]$. 

\item[{\bf P3.}]  For $w\in V_h^p$ we have the estimate 
\begin{equation}\label{eq:prop3a}
\| w \|^2_{s_h} \lesssim 
h^{2\gamma} \|[D_{n_F}^1 w  ]\|^2_{L^2(\mcF_h)}
+
h^{2\gamma} \| D_{n_h}^1 w  \|^2_{L^2(\mcK_h)}
+
h^{2(1+\gamma)-1} \| D^2 w \|^2_{L^2(\mcT_h)}.
\end{equation}
\paragraph{Verification of (\ref{eq:prop3a}).} We start from the definition
\begin{equation}
\| w \|^2_{s_h} = \| w \|^2_{s_{h,F}} + \| w \|^2_{s_{h,\Gamma}}
\end{equation}
and estimate the two terms on the right hand side as follows. First
\begin{align}
\| w \|^2_{s_{h,F}} 
&\lesssim h^{2\gamma} \|[D_{n_F}^1 w ]\|^2_{L^2(\mcF_h)} 
+  \sum_{j=2}^{p} h^{2(j-1+\gamma)}     \|[D_{n_F}^{j} w ]\|^2_{L^2(\mcF_h)}
\\
&\lesssim h^{2\gamma} \|[D_{n_F}^1 w ]\|^2_{L^2(\mcF_h)}      
+  h^{2(1+\gamma)-1} \|[D^{2} w ]\|^2_{L^2(\mcT_h)},
\end{align}
where for each $j=2,\dots,p$ and each face $F \in \mcF_h$ with 
neighboring elements $T_1$ and $T_2$ we used the inverse estimate      
\begin{equation}
 \|[D_{n_F}^{j} w ]\|^2_{L^2(F)} \lesssim  h^{-1}  \|D^{j} w\|^2_{L^2(T_1 \cup T_2)} 
\lesssim  h^{-1}  h^{2(2-j)}\|D^{2} w\|^2_{L^2(T_1 \cup T_2)}. 
\end{equation}
Second, using a similar approach
\begin{align}
\| w \|^2_{s_{h,\Gamma}} 
&\lesssim h^{2\gamma} \| D_{n_h}^1 w \|^2_{L^2(\mcK_h)}  
+ \sum_{j=2}^p  h^{2(j-1 + \gamma)} \| D_{n_h}^j w \|^2_{L^2(\mcK_h)}
\\
 &\lesssim h^{2\gamma} \| D_{n_h}^1 w \|^2_{L^2(\mcK_h)}  
+ \sum_{j=2}^p  \underbrace{h^{2(j-1 + \gamma)} h^{-1}  h^{2(2-j)}}_{=h^{2(1+\gamma)-1}} 
\| D^2 w \|^2_{L^2(\mcT_h)}
\\
&\lesssim 
 h^{2\gamma} \| D_{n_h}^1 w \|^2_{L^2(\mcK_h)}  
 +  
h^{2(1+\gamma)-1} \| D^2 w \|^2_{L^2(\mcT_h)},
\end{align}
where we used the inverse estimate 
\begin{equation}
 \| D_{n_h}^j w \|^2_{L^2(K)} 
 \leq
 \| D^j  w \|^2_{L^2(K)}  
\lesssim 
  h^{-1}\| D^j w \|^2_{L^2(T)} 
 \lesssim 
  h^{-1} h^{2(2-j)} \| D^2 w \|^2_{L^2(T)} 
\end{equation}
for $j=2,\dots,p$. Thus (\ref{eq:prop3a}) holds.

Setting $w = \pi_h^p v^e$  in (\ref{eq:prop3a})  we obtain 
\begin{align}\label{eq:prop3b}
\| \pi_h^p v^e  \|^2_{s_h} &\lesssim 
\underbrace{h^{2\gamma} \|[D_{n_F}^1 \pi_h^p v^e  ]\|^2_{L^2(\mcF_h)}}_{I}
 +
\underbrace{h^{2\gamma} \| D_{n_h}^1 \pi_h^p v^e  \|^2_{L^2(\mcK_h)}}_{II} 
\\ 
\nonumber & \qquad  
 + 
 h^{2(1+\gamma)-1} \| D^2 \pi_h^p v^e \|^2_{L^2(\mcT_h)}
 \\
&= I + II +  h^{2(1+\gamma)} \| v \|^2_{H^2(\Gamma)},
\end{align}
where we used the stability of the interpolation operator
\eqref{eq:stabilityinterpol} and the stability of the extension
operator \eqref{eq:extstab}. Next we have the following bounds.

\paragraph{Term $\bfI$.} Using the fact that $[D_{n_F}^1 v^e] = 0$ for $v\in H^2(\Gamma)$ we 
obtain
\begin{align}
\|[D_{n_F}^1 \pi_h^p v^e  ]\|^2_{L^2(\mcF_h)} &= \|[D_{n_F}^1 (\pi_h^p v^e - v^e)  ]\|^2_{L^2(\mcF_h)}
\\
&\lesssim 
h^{-1} \|[D (\pi_h^p v^e - v^e)  ]\|^2_{L^2(\mcT_h)}
+ h \|[D^2  (\pi_h^p v^e - v^e)  ]\|^2_{L^2(\mcT_h)}
\\
&\lesssim 
h \| D^2 v^e \|^2_{L^2(\mcT_h)}
\\
&\lesssim 
h^2 \| v \|^2_{H^2(\Gamma)},
\end{align}
where we used the trace inequality (\ref{eq:tracestand}), the interpolation estimate (\ref{eq:interpoltets}), 
and finally the stability (\ref{eq:extstab}) of the extension operator. Thus we have
\begin{equation}
I \lesssim h^{2(1+\gamma)} \| v \|^2_{H^2(\Gamma)}.
\end{equation}

\paragraph{Term $\bfI\bfI$.} Using the fact that $D_{n^e}^1 v^e = 0$ we obtain
\begin{align}
\| D_{n_h}^1 \pi_h^p v^e  \|^2_{L^2(\mcK_h)}   
&\lesssim 
\| D_{n^e}^1 \pi_h^p v^e  \|^2_{L^2(\mcK_h)} + \| n^e - n_h \|^2_{L^\infty(\mcKh)} \| D \pi_h^p v^e \|^2_{L^2(\mcKh)}
\\
&\lesssim 
\| D_{n^e}^1 (\pi_h^p v^e - v^e ) \|^2_{L^2(\mcK_h)} + h^{2p} \| D \pi_h^p v^e \|^2_{L^2(\mcTh)}
\\
&\lesssim 
h^2 \| v \|^2_{H^2(\Gamma)} + h^{2p} \| v \|^2_{H^1(\Gamma)}
\end{align}
and we arrive at 
\begin{equation}
II \lesssim h^{2(1+\gamma)} \| v \|^2_{H^2(\Gamma)}
\end{equation}
since $p\geq 1$ and $h \in (0,h_0]$.

Collecting the bounds we obtain
\begin{equation}
\| \pi_h^p v^e \|_{s_h} 
\lesssim 
h^{1+\gamma}\| v \|_{H^2(\Gamma)}
\end{equation}
and thus P3 holds.

\paragraph{Simplified Proof of P3 in the Case $\boldsymbol{\gamma}{\bf=1}$.}
Using estimate (\ref{eq:traceonstab}) with $w\in V^p_h$ and inverse inequality 
\eqref{eq:inverse} we obtain 
\begin{align}
\| w \|^2_{s_h} 
&\lesssim 
\sum_{j=0}^p  h^{2j +1} \|D^{j+1} w  \|^2_{L^2(\mcT_h)}
\\
&\lesssim  
h \|D w \|^2_{L^2(\mcT_h)} 
+  \underbrace{\sum_{j=1}^p  h^{2j+1} \|D^{j+1} w \|^2_{L^2(\mcT_h)}}_{\lesssim h^3 \| D^2 w \|^2_{L^2(\mcTh)} } 
\\
&\lesssim  
h \|D w \|^2_{L^2(\mcT_h)} + h^3 \| D^2 w \|^2_{L^2(\mcTh)}.
\end{align}
Setting $w = \pi_h^p v^e$ and using the stability of the interpolation operator
\eqref{eq:stabilityinterpol} and the stability of the extension
operator \eqref{eq:extstab} we obtain 
\begin{align}
\| \pi_h^p v^e \|^2_{s_h} 
&\lesssim 
h \|D \pi_h^p v^e  \|^2_{L^2(\mcT_h)} + h^3 \| D^2 \pi_h^p v^e  \|^2_{L^2(\mcTh)}
\\
&\lesssim 
h^2 \| v \|^2_{H^1(\Gamma)} + h^4 \| v  \|^2_{H^2(\Gamma)}.
\end{align}
\item[{\bf P4.}] Starting from (\ref{eq:traceonstab}) with $w \in V^p_h$ and using the inverse 
inequality we  get
\begin{align}\label{eq:inversestab}
\| w \|_{s_h}^2 
&\lesssim 
\sum_{j=0}^p h^{2j+2\gamma -1} \|D^{j+1} w \|^2_{L^2(\mcT_h)} 
\lesssim 
h^{2\gamma-3} \| w \|^2_{L^2(\mcT_h)}. 
\end{align}
Using that $h \in (0,h_0]$ we have $h^{2\gamma-3} \lesssim h^{-3}$ for $\gamma \geq 0$ and thus Property P4 holds. 

\item[{\bf P5.}]
See the proof of Lemma \ref{lem:Poincare-Th} below.
\end{enumerate}

\begin{rem} We note that the normal gradient stabilization (\ref{eq:normalgradel}) satisfies 
$P_1,P_2,P_4$, and $P_5$, see \cite{GraLehReu16} and \cite{BurHanLarMas16}. To verify $P_3$ we use that $n^e \cdot \nabla v^e=0$, interpolation error estimates, the $H^1$-stability of the interpolant, assumption A2 (the inclusion $\mcT_h \subset U_\delta (\Gamma )$ with $\delta \sim h$), the stability of the extension operator \eqref{eq:extstab}, and that $\|  n_h -n^e \|^2_{L^\infty(\mcT_h)} \lesssim h^{2p}$ to conclude that
\begin{align}
\| \pi_h^p v^e\|_{s_h}^2 
& \lesssim 
 h^{\alpha} \left(
 \|n^e \cdot  \nabla \pi_h^p v^e  \|^2_{L^2(\mcT_h)} + \|  n_h -n^e \|^2_{L^\infty(\mcT_h)} \|  \nabla \pi_h^p v^e \|^2_{L^2(\mcTh)} \right)
 \\
&\lesssim 
 h^{\alpha} \left (\| n^e \cdot \nabla(\pi_h^p v^e-v^e)\|_{L^2(\mcT_h)}^2+ h^{2p} \delta \| v \|_{H^1(\Gamma)}^2 \right)
 \lesssim 
h^{\alpha} h^2 \delta \| v \|^2_{H^2(\Gamma)}   \\
&\lesssim 
 h^{3+\alpha}\| v \|^2_{H^2(\Gamma)}.
\end{align}
We also used that $p\geq 1$. Thus we conclude that $P_3$ holds for $\alpha\geq -1$.
\end{rem}

\section{Condition Number Estimate} \label{sec:conditionnr}

We shall show that the spectral condition number of the resulting stiffness matrix scales 
as $h^{-2}$, independent of the position of the geometry relative to the background mesh. 
In particular, we show that the stabilization term defined in (\ref{eq:sh}) has Property P5 and 
controls the condition number for linear as well as for higher-order elements.

Let $\{\varphi_i\}_{i=1}^N$ be a basis in $V_h^p$ and given $v \in V_h^p$ let
$\widehat{v} \in \widehat{\IR}^N\subset \IR^N $ denote the vector containing the coefficients 
in the expansion: 
\begin{equation} 
v = \sum_{i=1}^N \widehat{v}_i \varphi_i.
\end{equation}
Recall that functions $v$ in $V_h^p$ satisfy the condition $\int_{\Gamma_h} v \,ds_h=0$ 
and therefore $\widehat{\IR}^N$ is 
\begin{equation}
\widehat{\IR}^N=\{\widehat{v} \in \IR^N | \ \widehat{v} \cdot ( \int_{\Gamma_h}\varphi_1 \,ds_h, \cdots, \int_{\Gamma_h}\varphi_N \,ds_h)=0 \}.
\end{equation}
Since $\mcT_h$ is quasi-uniform we have the equivalence 
\begin{equation}\label{eq:eql2normdisc}
\| v \|_{\mcT_h} \sim h^{d/2}\| \widehat{v} \|_{\widehat{\IR}^N},
\end{equation}
where $d$ is the dimension of the embedding space $\IR^d$. Let $\mcA_h$ 
be the stiffness matrix associated with $A_h$, 
\begin{equation}
(\mcA_h \widehat{w},\widehat{v})_{\widehat{\IR}^N} =A_h(w,v)\qquad \forall v,w \in V_h^p
\end{equation} 
and recall that the condition number is defined by 
\begin{equation}
\kappa(\mcA_h) = \| \mcA_h \|_{\IR^N} \| \mcA_h^{-1} \|_{\IR^N}
\end{equation}
which in terms of the eigenvalues of the symmetric positive definite matrix $\mcA_h$ 
is equal to 
\begin{equation}\label{eq:spectral-cond-number}
\kappa(\mcA_h) = \frac{\lambda_N}{\lambda_1}, 
\end{equation}
where $\lambda_N$ ($\lambda_1$) is the largest (smallest) eigenvalue of $\mcA_h$.

\begin{thm}\label{thm:condition} If there are constants, independent of the mesh 
size $h$ and of how the surface cuts the background mesh, such that the following 
hold:
\begin{itemize}
\item $A_h$ is continuous:
$A_h(w,v) \lesssim  \| w \|_{A_h} \| v \|_{A_h} \qquad \forall v, w \in V_h^p + H^{p+1}(\mcT_h)$ 
\item $A_h$ is coercive:
$\|v \|_{A_h}^2 \lesssim A_h(v,v) \qquad \forall v \in V_h^p$ 
\item The inverse inequality: 
$\| v \|_{A_h} \lesssim h^{-3/2} \| v \|_{L^2(\mcT_h)} \qquad v \in V_h^p$
\item The Poincar\'e inequality: 
$\| v \|_{L^2(\mcT_h)} \lesssim h^{1/2} \| v \|_{A_h} \qquad v \in V_h^p$
\end{itemize}
Then, the spectral condition number $\kappa(\mcA_h)$ satisfies
\begin{equation}
\kappa(\mcA_h) \lesssim h^{-2}.
\end{equation}
\end{thm}
\begin{proof}
For $v, \ w \in V_h^p$ it follows from the continuity of the form $A_h$, the inverse inequality, 
and the equivalence between the norms, equation \eqref{eq:eql2normdisc}, that
\begin{equation}
A_h(w,v)\lesssim \|w\|_{A_h}\|v\|_{A_h}\lesssim h^{-3} \|w\|_{L^2(\mcT_h)}\|v\|_{L^2(\mcT_h)}\lesssim h^{-3} h^{d}  \| \widehat{w} \|_{\widehat{\IR}^N}\| \widehat{v} \|_{\widehat{\IR}^N}.
\end{equation}
From coercivity, the Poincar\'e inequality, and the equivalence between the norms, equation \eqref{eq:eql2normdisc}, we obtain
\begin{equation}
A_h(v,v) \gtrsim \|v\|_{A_h}^2\gtrsim h^{-1} \|v\|_{L^2(\mcT_h)}^2\gtrsim h^{-1} h^{d}  \| \widehat{v} \|_{\widehat{\IR}^N}^2.
\end{equation}
Using the definition (\ref{eq:spectral-cond-number})  of the spectral condition number 
we get
\begin{equation}
\kappa(\mcA_h)=\frac{\max_{u\in \widehat{\IR}^N,\| \widehat{u} \|_{\IR^N}=1} (\mcA_h \widehat{u},\widehat{u})_{\widehat{\IR}^N}}{\min_{u\in \widehat{\IR}^N,\| \widehat{u} \|_{\IR^N}=1} (\mcA_h \widehat{u},\widehat{u})_{\widehat{\IR}^N}} \lesssim \frac{h^{-3} h^{d}}{h^{-1} h^{d}}=h^{-2}.
\end{equation}
\end{proof}

In the following Lemmas we show that all the conditions in Theorem \ref{thm:condition}  are satisfied. 

\begin{lem} ({Continuity and Inf-Sup Condition}) There is a constant independent of the mesh size $h$ and of how the surface 
cuts the background mesh, such that $A_h$ is continuous: 
\begin{equation}\label{eq:Ah-cont}
A_h(w,v) \leq  \| w \|_{A_h} \| v \|_{A_h}, \quad  \forall v, w\in V^p_h + H^{p+1}(\mcT_h)
\end{equation}
and $A_h$ satisfies the inf-sup condition 
\begin{equation}\label{eq:inf-sup}
\| w \|_{A_h} \lesssim \sup_{v \in V_h^p \setminus  \{0\}} \frac{A_h(w,v)}{\| v \|_{A_h}}, \quad \forall w\in V_h^p.
\end{equation}
\end{lem} 
\begin{proof}
The continuity follows directly from the Cauchy-Schwarz inequality. The bilinear form $A_h$ is 
coercive $A_h(v,v)=\|v \|_{A_h}^2$ by definition and the inf-sup condition \eqref{eq:inf-sup} follows 
from coercivity.
\end{proof}

\begin{lem} ({Inverse Inequality}) There are constants, independent of the mesh size $h$ and of how the surface cuts 
the background mesh, such that the following inverse inequality holds
\begin{equation}\label{eq:inverseA}
\| v \|_{A_h} \lesssim h^{-3/2} \| v \|_{L^2(\mcT_h)} \qquad v \in V_h^p.
\end{equation} 
\end{lem} 
\begin{proof} 
Using the element wise trace inequality \eqref{eq:trace}, the inverse inequality \eqref{eq:inverse}, and Property P4 of the stabilization term we obtain
\begin{align}
\|v \|_{a_h}^2 &\lesssim h^{-1} \| v\|^2_{H^1(\mcT_h)} +h\| v\|^2_{H^2(\mcT_h)} \lesssim h^{-3} \| v\|_{L^2(\mcT_h)}^2, \nonumber \\
\|v \|_{s_h}^2 &\lesssim h^{-3} \| v \|_{L^2(\mcT_h)}^2 \label{eq:inversesh}.
\end{align}
Using the above estimates and recalling the definition of $\| v \|_{A_h}$ the result follows. 
\end{proof}

\begin{lem} ({Poincar\'e Inequality}) There are constants, independent of the mesh size $h$ and of how the surface cuts the background mesh, such that the following Poincar\'e inequality holds \label{lem:Poincare-Th}
\begin{equation}\label{eq:poincare-Th}
\| v \|_{L^2(\mcT_h)} \lesssim h^{1/2} \| v \|_{A_h} \qquad v \in V_h^p.
\end{equation}
\end{lem}

\begin{proof}  To prove the Poincar\'e inequality we will proceed in two main steps: 1. We use the 
face penalty to reach an element which has a large intersection with $\Gamma_h$. Here 
we employ a covering of $\mcT_h$, where each covering set contains elements that have a 
so called large intersection property, see~\cite{BurHanLar15}. 2. For elements with a large intersection 
the normal derivative control on $\Gamma_h$ is used to control the $L^2$ norm on the element. 

To establish the Poincar\'e inequality  (\ref{eq:poincare-Th}) we shall prove 
that \begin{equation}\label{eq:main-a}
\| v \|^2_{L^2(\mcT_h)} 
\lesssim {\sum_{j =0}^p  h^{2j+1}\| D^j_{n_h}  v \|^2_{L^2(\Gamma_h)} + \sum_{j=1}^p h^{2j+1} \| [ D_{n_F}^j v ] \|^2_{L^2(\mcF_h)}}.
\end{equation} 
Here we note that (\ref{eq:poincare-Th}) follows from (\ref{eq:main-a}) since the 
right hand side of (\ref{eq:main-a}) satisfies the estimate
\begin{align}\nonumber
&{\sum_{j =0}^p  h^{2j+1}\| D^j_{n_h}  v \|^2_{L^2(\Gamma_h)} + \sum_{j=1}^p h^{2j+1} \| [ D_{n_F}^j v ] \|^2_{L^2(\mcF_h)}}
\\
&\qquad 
\lesssim 
 h ( \| v \|^2_{\Gamma_h}  +  h^{2(1-\gamma)} \| v \|^2_{s_h}) 
\lesssim
 h ( \| v \|^2_{a_h}  +  h^{2(1-\gamma)} \| v \|^2_{s_h}) 
\lesssim 
h \| v \|^2_{A_h},
\end{align}
where we used  the definition (\ref{eq:sh}) of $s_h$, the Poincar\'e inequality (\ref{eq:poincare-Kh}), 
and the fact that $\gamma \in [0,1]$.
 
\paragraph{Large Intersection Coverings of $\mcT_h$.} 
There is a covering 
$\{\mcT_{h,x} : x \in \mcX_h\}$ of $\mcT_h$, where $\mcX_h$ is an index set,  
such that: (1) Each set $\mcT_{h,x}$ contains a uniformly bounded number 
of elements. (2) Each element in $\mcT_{h,x}$ share at least one face with 
another element in $\mcT_{h,x}$. (3) In each set $\mcT_{h,x}$ there is one 
element $T_{x} \in \mcT_{h,x}$ which has a large intersection with 
$\Gamma_h$, 
\begin{equation}\label{eq:large-intersection}
h^{d-1} \lesssim |T_{x} \cap \Gamma_h |=|K_{x}|,
\end{equation}
where $d$ is the dimension. (4) The number of sets $\mcT_{h,x}$ to 
which an element $T$ belongs is uniformly bounded for all $T\in \mcTh$. 
See~\cite{BurHanLar15} for the construction of the covering.

We prove (\ref{eq:main-a}) by considering a set $\mcT_{h,x}$ in the 
covering and the following steps:

\paragraph{Step 1.} We shall show that 
\begin{equation}\label{eq:step1}
\| v \|^2_{L^2(\mcT_{h,x})} \lesssim \|v \|^2_{L^2(T_{x})}  
+ \sum_{j=1}^p h^{2j+1} ([D_{n_F}^j v], [D_{n_F}^j w])_{\mcF_{h,x}}, 
\end{equation} 
where $\mcF_{h,x}$ is the set of interior faces in $\mcT_{h,x}$. To prove (\ref{eq:step1}) 
we recall that for two elements $T_1$ and $T_2$ that share a face $F$ it holds 
\begin{equation}\label{eq:proof-b}
\| v \|^2_{L^2(T_{1})} \lesssim \| v \|^2_{L^2(T_{2})} +  \sum_{j=1}^p h^{2j+1} \|[D_{n_F}^j v]\|_{L^2(F)}^2,   
\end{equation} 
see \cite{HaLaLa17} for instance. Now let $\mcT_{h,x}^0 = \{T_{x}\}$ and for 
$j=1,2,\dots$ let  $\mcT_{h,x}^{j}$ be the set of all elements that share a face with an 
element in $\mcT_{h,x}^{j-1}$. Then there is a uniform constant $J$ such that  
$\mcT_{h,x}^{j} = \mcT_{h,x}^{j-1}$ for $j \geq J$ and we also have the estimate 
\begin{equation}\label{eq:proof-c}
\| v \|_{\mcT_{h,x}^{j}} 
\lesssim \| v \|_{\mcT_{h,x}^{j-1}}  
+   \sum_{F \subset {\mcF^{j}_{h,x}} \setminus \mcF^{j-1}_{h,x}} \sum_{j=1}^p h^{2j+1} \|[D_{n_F}^j v]\|_{L^2(F)}^2,
\end{equation}
where $\mcF^{j}_{h,x}$ is the set of interior faces in $\mcT_{h,x}^{j}$. Iterating 
(\ref{eq:proof-c}) we obtain (\ref{eq:step1}).

\paragraph{Step 2.} We shall show that there is a constant such that for all 
$T_{x}\in \mcTh$  which have the large intersection property (\ref{eq:large-intersection}),  
\begin{equation}\label{eq:step2}
\| v \|^2_{L^2(T_{x})} \lesssim \sum_{j=0}^{p} h^{2j+1} \| D_{n_h}^j  v \|_{L^2(K_{x})}^2 
\end{equation}
where $K_x = \Gamma_h \cap T_x$. To verify (\ref{eq:step2}) we define  the  cylinder
\begin{equation}\label{eq:cylKx}
\text{Cyl}_{\delta} ({K}_{x}) = \{ x \in \IR^d : x = y + t {n}_h, y \in {K}_{x}, |t|\leq \delta \}
\end{equation}
and using Taylor's formula for a polynomial $v \in P_p(\text{Cyl}_{\delta} (K_{x}) )$, 
we obtain the bound
\begin{equation}\label{eq:codim-mod}
\| v \|^2_{L^2(\text{Cyl}_{\delta} ({K}_{x}))} \lesssim  \sum_{j=0}^{p} \delta^{2j+1} \| D_{{n}_h}^j  v \|_{L^2(K_{x})}^2.
\end{equation}
Using the following estimate, which we verify below, there is a constant such that  
for all $v \in P_p(T_{x})$,
\begin{equation}\label{eq:inverse-step-2}
\| v \|_{L^2(T_{x})} \lesssim \| v \|_{L^2(\text{Cyl}_{\delta} ({K}_{x}))},
\end{equation}
the desired estimate (\ref{eq:step2}) follows for $\delta \sim h$. 

\paragraph{Verification of (\ref{eq:inverse-step-2}).} To employ a scaling argument 
we will construct two regular cylinders with circular cross section and the same center 
line that may be mapped to a reference configuration, one containing $T_x$ and one 
contained in $\text{Cyl}_{\delta} ({K}_{x})$. See Figure \ref{fig:proof-step2}.

Let $\overline{F}_x$ be a plane with unit normal $\overline{n}_x$, which is tangent to 
$K_x$ at an interior point of $K_x$. Then we have the bound 
\begin{equation}
\| \overline{n}_h - n_h \|_{L^\infty(K_x)} \lesssim h
\end{equation}
and with $\overline{K}_x$ the closest point projection of $K_x$ onto
$\overline{F}_x$ we conclude using shape regularity and a uniform
bound on the curvature $\kappa_h$ (Assumption A1) that there is a ball
$\overline{B}_{\overline{r}_1,x} \subset \overline{K}_x \subset
\overline{F}_x$ with radius $\overline{r}_1\sim h$ and a
$\overline{\delta}_1 \sim h$ such that
\begin{equation}\label{eq:inverse-step-2-a}
\text{Cyl}_{\overline{\delta}_1} (\overline{K}_{x}) = 
\{ x \in \IR^d : x = y + t \overline{n}_h, y \in \overline{K}_{x}, |t|\leq \overline{\delta}_1 \}
\subset \text{Cyl}_{\delta} ({K}_{x}).
\end{equation}
Next using shape regularity there is a larger ball $\overline{B}_{\overline{r}_2,x} \subset 
\overline{F}_x$ with the same center as  $\overline{B}_{\overline{r}_1,x}$ such that 
\begin{equation}
\overline{T}_x \subset \overline{B}_{\overline{r}_2,x}  \subset \overline{F}_x,
\end{equation}
where $\overline{T}_{x}$ is the closest point projection of $T_x$ onto $\overline{F}_x$, 
and a $\overline{\delta}_2\sim h$ such that 
\begin{equation}\label{eq:inverse-step-2-b}
T_x \subset \text{Cyl}_{\overline{\delta}_2}(\overline{B}_{\overline{r}_2,x}).
\end{equation}
Clearly, $\text{Cyl}_{\overline{\delta}_1}(\overline{B}_{\overline{r}_1,x}) 
\subset \text{Cyl}_{\overline{\delta}_2}(\overline{B}_{\overline{r}_2,x}) $ and 
using a mapping to a reference configuration  we conclude that there is a 
constant such that for all polynomials $v \in P_p(   \text{Cyl}_{\overline{\delta}_2} )(\overline{B}_{\overline{r}_2,x})$,
\begin{equation}
\| v \|_{ \text{Cyl}_{\overline{\delta}_2}(\overline{B}_{\overline{r}_2,x})} 
\lesssim 
 \| v \|_{ \text{Cyl}_{\overline{\delta}_1}(\overline{B}_{\overline{r}_1,x})} 
\end{equation}
which in view of the inclusions (\ref{eq:inverse-step-2-a}) and (\ref{eq:inverse-step-2-b}) 
concludes the proof of (\ref{eq:inverse-step-2}).
\paragraph{Step 3.}
Combining equation \eqref{eq:step1} and \eqref{eq:step2} and using that $\{\mcT_{h,x} : x \in \mcX_h\}$ is a cover of $\mcT_h$ completes the proof of \eqref{eq:main-a} and the lemma. 
\end{proof}
\begin{figure}\label{fig:cover-set}
\begin{center}
\includegraphics[scale=0.25]{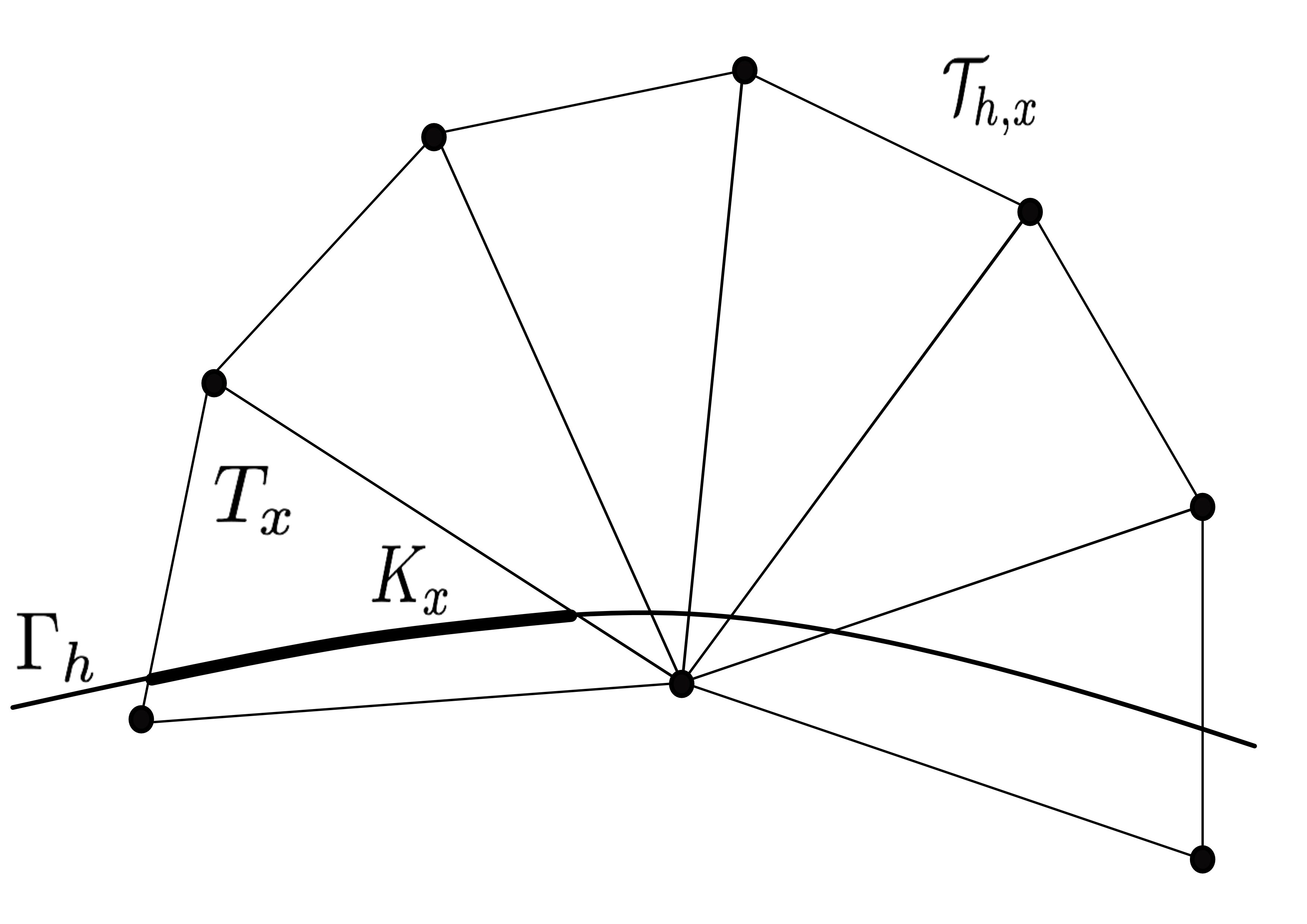}
\end{center}
\caption{A set of elements $\mcT_{h,x}$ in the cover of $\mcT_h$, an element $T_x$ with large intersection $K_x = T_x \cap \Gamma_h$.}
\end{figure}

\begin{figure}\label{fig:proof-step2}
\begin{center}
\includegraphics[scale=0.65]{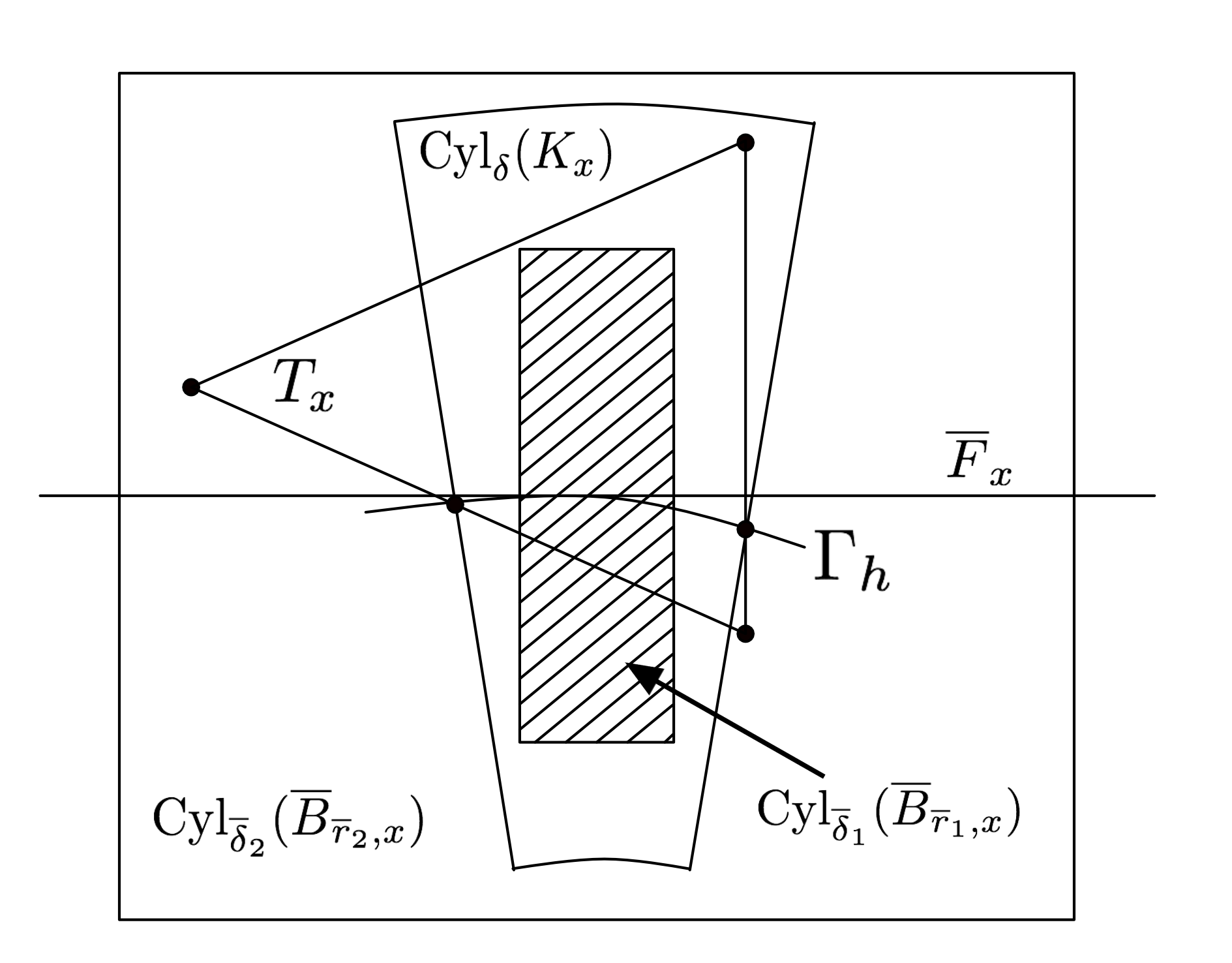}
\end{center}
\caption{Schematic figure illustrating a typical configuration of an element $T$, the curved intersection 
with $\Gamma_h$, the plane $\overline{F}_x$ which is tangent to $K_x$, the cylinders $\text{Cyl}_\delta(K_x)$, $\text{Cyl}_{\overline{\delta}_1}(\overline{B}_{\overline{r}_1})$, and $\text{Cyl}_{\overline{\delta}_2}(\overline{B}_{\overline{r}_2})$. Note that we have the 
inclusions $\overline{B}_{\overline{r}_1,x} \subset \overline{K}_x \subset \overline{T}_x \subset \overline{B}_{\overline{r}_2,x}$, 
$T_x \subset \text{Cyl}_{\overline{\delta}_2}(\overline{B}_{\overline{r}_2})$, and 
$\text{Cyl}_{\overline{\delta}_1}(\overline{B}_{\overline{r}_1}) \subset \text{Cyl}_\delta(K_x)$.}
\end{figure}

\begin{remark} Using the same technique as in the proof of Lemma \ref{lem:Poincare-Th}  
we may show the stronger estimate 
\begin{equation}\label{eq:poincare-stronger}
\| v \|^2_{\mcTh} +  h^{2\gamma} \| \nabla v \|^2_{\mcT_h} \lesssim h \| v \|^2_{A_h},
\end{equation}
where $\gamma \in [0,1]$ is the scaling parameter in $s_h$, see (\ref{eq:sh-face}) 
and (\ref{eq:sh-gamma}). The modifications are as follows. In Step 1 we show that 
\begin{equation}\label{eq:step1-grad}
\| \nabla v \|^2_{L^2(\mcT_{h,x})} \lesssim \| \nabla v \|^2_{L^2(T_{x})}  
+ \sum_{j=1}^p h^{2j-1} \|[D_{n_F}^j v]\|_{L^2(F)}^2
\end{equation} 
which after multiplication by $h^{2\gamma}$ corresponds to 
\begin{equation}\label{eq:step1-grad-b}
h^{2\gamma} \| \nabla v \|^2_{L^2(\mcT_{h,x})} \lesssim 
h^{2\gamma} \| \nabla v \|^2_{L^2(T_{x})}  
+ h \left( \sum_{j=1}^p h^{2(j-1+\gamma)} \|[D_{n_F}^j v]\|_{L^2(F)}^2 \right). 
\end{equation} 
To show (\ref{eq:step1-grad}) we use the estimate 
\begin{equation}\label{eq:step1-twoelem-grad}
\| \nabla v \|^2_{L^2(T_{1})} 
\lesssim 
\| \nabla v \|^2_{L^2(T_{2})}  
+  \sum_{j=1}^p h^{2j-1} \|[D_{n_F}^j v]\|_{L^2(F)}^2.   
\end{equation} 
In Step 2 we show that
\begin{equation}\label{eq:step2-grad}
\| \nabla v \|^2_{L^2(T_{x})} 
\lesssim  \| \nablash v \|^2_{K_x} 
+ \sum_{j=1}^{p} h^{2j - 1} \| D_{n_h}^j  v \|_{L^2(K_{x})}^2 
\end{equation}
which after multiplication by $h^{2\gamma}$ corresponds to
\begin{equation}\label{eq:step2-grad-b}
h^{2\gamma} \| \nabla v \|^2_{L^2(T_{x})} 
\lesssim \underbrace{h^{2\gamma} \| \nablash v \|^2_{K_x}}_{\lesssim \| \nablash v \|^2_{K_x} }
+ h \left( \sum_{j=1}^{p} h^{2(j - 1 + \gamma)} \| D_{n_h}^j  v \|_{L^2(K_{x})}^2 \right).
\end{equation}
Combining (\ref{eq:step1-grad-b}) and (\ref{eq:step2-grad-b}), summing over 
the cover, and using Property (4) of the cover we obtain (\ref{eq:poincare-stronger}).
\end{remark}

\begin{rem}
Note the following:
\begin{itemize}
\item From~\eqref{eq:inversestab} we see that the scaling of the stabilization term corresponds to the mass matrix.  
This scaling is the weakest which guarantees the Poincar\'e inequality (\ref{eq:poincare-Th}) (Property P4 of the stabilization term).
\item For the operator 
\begin{equation}
L = \alpha \Delta_\Gamma + \beta
\end{equation}
where $\alpha$ and $\beta$ are constants, we would have the inverse inequality  
\begin{equation}\label{eq:inverse-L}
\| v \|_{A_h} \lesssim( \alpha h^{-3/2} + \beta h^{-1/2}  )\| v \|_{\mcT_h} \qquad v \in V_h^p
\end{equation} 
instead of (\ref{eq:inverseA}), and the resulting condition number estimate is then 
\begin{equation}
\kappa(\mcA_h) \lesssim \alpha h^{-2} + \beta. 
\end{equation}
\item In the case when
\begin{equation}
\| v \|_{a_h} \lesssim (\alpha h^{-3/2} + \beta h^{-1/2} ) \| v \|_{\mcT_h}
\end{equation}
we note that the desired inverse inequality (\ref{eq:inverse-L}) holds 
for 
\begin{equation}
\widetilde{s}_h = (\alpha h^{-\tau} + \beta ) s_h,  
\end{equation}
with $\tau$ in the interval $0\leq \tau \leq 2$, since 
\begin{equation}
\| v \|^2_{\widetilde{s}_h} 
=
(\alpha h^{-\tau} + \beta) \| v \|^2_{s_h} 
\lesssim 
(\alpha h^{-\tau} + \beta) h^{-1} \| v \|^2_{\mcT_h}
\lesssim 
(\alpha h^{-(1+\tau)} + \beta h^{-1} ) \| v \|^2_{\mcT_h}.
\end{equation}
\end{itemize}
\end{rem}

\section{Extension to Problems on Manifolds with General Codimension Embeddings} \label{sec:codim}

The stabilization term (\ref{eq:sh}) may be extended to the more general case of an $n$-dimensional smooth manifold embedded in $\IR^d$, 
with codimension $cd = d - n > 1$, by scaling the face penalty term (\ref{eq:sh-face}) in 
such a way that the stabilization terms associated with the faces and surface scale in the same way,
\begin{align} \label{eq:sh-face-cd}
s_{h,F}(w,v)&= \sum_{j=1}^p c_{F,j} h^{2(j-1+\gamma) } h^{1-{cd}}([D_{n_F}^j w], [D_{n_F}^j v])_{\mcF_h}, 
\\ \label{eq:sh-gamma-cd}
s_{h,\Gamma}(w,v)&=\sum_{j=1}^p c_{\Gamma,j} h^{2(j-1+ \gamma)}
(D_{n_h}^j w, D_{n_h}^j v)_{\mcK_h}.
\end{align}
This is the same scaling as is used for the face stabilization term in the case of piecewise 
linear elements in general codimension, see Table 1 in \cite{BurHanLarMas16}. With this 
definition P1-P3 remain the same and may be verified using the results in \cite{BurHanLarMas16}. P4-P5 take the more general form
\begin{enumerate}
\item[{\bf P4.}] For $ v \in V_h^p$, 
\begin{equation}
\| v \|^2_{s_h} \lesssim h^{-(cd+2)}\| v \|^2_{L^2(\mcT_h)}.
\end{equation}
\item[{\bf P5.}]  For $v \in V_h^p$, 
\begin{equation}\label{eq:property5-cd}
\| v \|_{L^2(\mcT_h)}^2 \lesssim h^{cd} (\| v \|_{a_h}^2+ \| v \|_{s_h}^2 ).
\end{equation}
\end{enumerate}
P4 is verified as in equation \eqref{eq:inversestab} and below we comment on the minor modifications in 
the proof of Lemma \ref{lem:Poincare-Th} necessary to verify P5.

We next extend the proof of Lemma \ref{lem:Poincare-Th}, to the case 
of a general embedding of an $n$-dimensional surface in $\IR^d$. To that end we first note that 
Step 1 can be directly carried out for $d$-dimensional simplices and we obtain 
\begin{equation}\label{eq:step1-codim}
\| v \|^2_{L^2(\mcT_{h,x})} \lesssim \|v \|^2_{L^2(T_{x})}  
+ h^{cd} \left(\sum_{j=1}^p h^{2j} h^{1-cd} ([D_{n_F}^j v], [D_{n_F}^j w])_{\mcF_{h,x}} \right).
\end{equation} 
For Step 2, the only 
difference is that we will have $cd = d-n$ orthonormal normal directions $\{n_{h,i}\}_{i=1}^{cd}$, which we may choose to vary smoothly over $K_x$, see \cite{BurHanLarMas16} 
for details, and therefore the definition of the cylinder over $K_x$ takes the form
\begin{equation}\label{eq:cylKx-codim}
\text{Cyl}_{\delta} ({K}_{x}) = \{ x \in \IR^d : x = y + \sum_{i=1}^{cd} t_i {n}_{h,i}, y  \in {K}_{x}, |t_i|\leq \delta, i=1,\dots,cd \}.
\end{equation}
Using Taylor's formula in several dimensions and integrating over the cylinder 
(\ref{eq:cylKx-codim}), we obtain the following generalization of \eqref{eq:codim-mod}, 
\begin{equation}
\| v \|^2_{L^2(\text{Cyl}_{\delta} ({K}_{x}))} \lesssim  \sum_{j=0}^{p} \delta^{2j+cd} \| D_{{n}_h}^j  v \|_{L^2(K_{x})}^2
\end{equation}
where we used the fact that, for a monomial  $\Pi_{i=1}^{cd} x_i^{2l_i}$, with 
$\sum_{i=1}^{cd} l_i = j$, it holds
\begin{equation}
\int_{[-\delta,\delta]^{cd}} \Pi_{i=1}^{cd} x_i^{2l_i}   
=  \Pi_{i=1}^{cd} \int_{-\delta}^\delta x_i^{2l_i} dx_i
\sim \Pi_{i=1}^{cd} \delta^{2l_i+1} 
= \delta^{2j+cd}.
\end{equation}
We finally note that (\ref{eq:inverse-step-2}) holds in any dimension $d$ and the desired estimate in Step 2 follows for $\delta \sim h$, 
and Step 3 is just a combination of Step 1 and 2. Thus the proof of 
Lemma \ref{lem:Poincare-Th} easily extends to the case of general 
codimension embeddings and we conclude that (\ref{eq:property5-cd}) holds.

\section{A Priori Error Estimates}
In this section we prove optimal error estimates in the energy norm and in the $L^2$ norm. 
 
\subsection{Strang Lemma}
Using that $A_h$ is continuous and satisfies the inf-sup condition we first show a Strang Lemma which connects the error in the energy norm to the interpolation error and the consistency error.

\begin{lem}
Let $u \in H^1_0(\Gamma)$ be the unique solution of \eqref{eq:LB} and $u_h\in V_{h}^p$ the finite element approximation defined by \eqref{eq:Ah}. Then the following discretization error bound holds
\begin{equation}
\|u^e-u_h\|_{A_h} \lesssim \|u^e-\pi_h^p u^e\|_{A_h}+\sup_{v \in V_h^p \setminus  \{0\}} \frac{|A_h (u^e,v)  - l_h(v)|}{\| v \|_{A_h}}.
\end{equation}
\end{lem}

\begin{proof}
Adding and subtracting an interpolant and using the triangle inequality we get
\begin{equation}
\|u^e-u_h\|_{A_h} \leq \|u^e-\pi_h^p u^e\|_{A_h}+\|\pi_h^p u^e-u_h\|_{A_h}.
\end{equation}
Using the inf-sup condition (\ref{eq:inf-sup}) for $A_h$  we have
\begin{equation}
\|\pi_h^p u^e-u_h\|_{A_h} \lesssim \sup_{v \in V_h^p \setminus  \{0\}} \frac{A_h(\pi_h^p u^e-u_h,v)}{\| v \|_{A_h}}.
\end{equation}
Adding and subtracting $A_h (u^e,v)$ and using the weak formulation \eqref{eq:Ah} yields
\begin{align}
A_h(\pi_h^p u^e - u_h,v) 
&=
A_h (\pi_h^p u^e - u^e,v) + A_h (u^e - u_h,v)
\\
&=
A_h (\pi_h^p u^e - u^e,v) + A_h (u^e,v)  - l_h(v).
\end{align}
Finally, using the continuity of $A_h$ we obtain 
\begin{equation}
A_h(\pi_h^p u^e - u^e,v) \leq \|  \pi_h^p u^e - u^e \|_{A_h} \| v \|_{A_h}.
\end{equation}
Collecting the estimates we end up with the desired bound 
\begin{equation}
\|u^e-u_h\|_{A_h} \lesssim \|u^e-\pi_h^p u^e\|_{A_h}+\sup_{v \in V_h^p \setminus  \{0\}} \frac{|A_h (u^e,v)  - l_h(v)|}{\| v \|_{A_h}}.
\end{equation}
\end{proof}

\subsection{The Interpolation Error} \label{sec:interp}
Next we prove optimal interpolation error estimates using Property P1 of the 
stabilization term. 

\begin{lem} For all $u\in H^{p+1}(\Gamma)$ we have the following estimate \label{lem:interp}
\begin{equation}\label{eq:interpolenergysurf}
\| u^e - \pi_h^p u^e \|_{A_h} \lesssim h^{p}\|u\|_{H^{p+1}(\Gamma)}.
\end{equation}
\end{lem}
\begin{proof} 
From the definition of the energy norm $\| \cdot \|_{A_h}$, equation (\ref{eq:energinorm}), we have
\begin{equation}
\| u^e - \pi_h^p u^e \|_{A_h}^2 =\underbrace{\| u^e- \pi_h^p u^e \|^2_{a_h}}_{I}+\underbrace{\| u^e- \pi_h^p u^e\|^2_{s_h}}_{II}.
\end{equation}

\paragraph{Term $\bfI$.} Using the element wise trace inequality~\eqref{eq:trace}, standard 
interpolation estimates \eqref{eq:interpoltets} on elements $T\in \mcT_h$, and the stability estimate \eqref{eq:extstab} for the extension operator with $\delta \sim h$, we obtain 
\begin{align}\label{eq:interpolsurf}
I&=\| u^e - \pi_h^p u^e \|^2_{a_h} 
\\
&\lesssim 
\sum_{T\in \mcT_h} \left(h^{-1}\| u^e - \pi_h^p u^e \|^2_{H^1(T)}+h\|u^e - \pi_h^p u^e \|^2_{H^2(T)} \right) 
\\
&\lesssim
h^{2p} \| u\|^2_{H^{p+1}(\Gamma)}.
\end{align}
\paragraph{Term $\bfI\bfI$.}  From Property P1 of the stabilization term we have  that
\begin{align}\label{eq:interpolstab}
II &\lesssim  h^{2p} \| u\|^2_{H^{p+1}(\Gamma)}. 
 \end{align}

Combining the two estimates \eqref{eq:interpolsurf} and \eqref{eq:interpolstab} the result follows.
\end{proof}

\subsection{The Consistency Error} 
The approximation of the geometry $\Gamma$ by $\Gamma_h$ and the approximation 
of the data $f$ by $f_h$ leads to a consistency error that we estimate in the next lemma.

\begin{lem}  \label{lem:consistency} Let $u\in H^{p+1}(\Gamma)\cap H^1_0(\Gamma)$ 
be the solution of \eqref{eq:LB} then, the following bound holds
\begin{equation}\label{eq:consistency}
\sup_{v \in V_h^p \setminus  \{0\}} \frac{|A_h (u^e,v)  - l_h(v)|}{\| v \|_{A_h}} \lesssim h^{p}  \| u  \|_{H^{p+1}(\Gamma)} +h^{p+1} \|f\|_{L^2(\Gamma)}. 
\end{equation}
\end{lem} 
\begin{proof}
We have the identity
\begin{align}
|A_h ( u^e,v)-l_h(v)|&
 \leq \underbrace{|a_h (u^e,v)  - l_h(v)|}_{I} +
\underbrace{|s_{h}(u^e,v)|}_{II}. 
 \end{align}
\paragraph{Term $\bfI$.} Adding $-a(u,v^l)+l(v^l)=0$, and using the triangle inequality
\begin{equation}
I \leq \underbrace{|a_h(u^e,v) - a(u,v^l)|}_{I_I} + \underbrace{|l(v^l) - l_h(v)|}_{I_{II}}.
\end{equation}
\paragraph{Term $\bfI_{\bfI}$.} Adding and subtracting  $\int_{\Gamma_h} \nabla_{\Gamma_h} u^e \cdot \nabla_{\Gamma_h} v |B|\,ds_h$, using that $B^TB^{-T}=P_{\Gamma_h}$, changing the domain of integration from $\Gamma$ to $\Gamma_h$, using the triangle inequality, norm equivalences in \eqref{eq:normequgrad}, estimates \eqref{B-uniform-bounds}, and \eqref{eq:B-detbound} we obtain
\begin{align} 
|a_h(u^e,v) - a(u,v^l)| 
&=\Big|\int_{\Gamma_h} (1-|B|) \nabla_{\Gamma_h} u^e \cdot \nabla_{\Gamma_h} v \,ds_h\nonumber 
\\
&\qquad + \int_{\Gamma_h}  B^T(B^{-T}\nabla_{\Gamma_h} u^e) \cdot B^T(B^{-T}\nabla_{\Gamma_h} v) |B|\,ds_h 
\nonumber 
\\
&\qquad - \int_{\Gamma_h} (B^{-T}\nabla_{\Gamma_h} u^e) \cdot (B^{-T}\nabla_{\Gamma_h} v) |B| \,ds_h \Big|
\nonumber 
\\
&\lesssim \left( \left\| 1 - |B| \right\|_{L^\infty(\Gamma_h)}+\| BB^T-I\|_{L^\infty(\Gamma_h)}  \right) \| \nabla_{\Gamma} u\|_{L^2(\Gamma)} \| \nabla_{\Gamma_h} v \|_{L^2(\Gammah)} \nonumber \\
&\lesssim h^{p+1}\| \nabla_{\Gamma} u\|_{L^2(\Gamma)}   \| v \|_{a_h}.
\end{align}
\paragraph{Term $\bfI_{\bfI\bfI}$.} Changing the domain of integration we get
\begin{align} 
|l(v^l) - l_h(v)|
&=\Big|\int_{\Gamma} fv^l \,ds-\int_{\Gamma_h} f_hv \,ds_h \Big|
\\  
&=
\Big|\int_{\Gamma_h} f^e v|B|\,ds_h-\int_{\Gamma_h} f_hv \,ds_h\Big| 
\\ 
&\lesssim  \||B|f^e-f_h \|_{L^2(\Gammah)}\| v \|_{L^2(\Gammah)}
\\
&\lesssim h^{p+1} \| f \|_{L^2(\Gamma)} \| v \|_{L^2(\Gamma)},
\end{align}
where we at last used Assumption A3 on the data approximation $f_h$. 
Together, the bounds of $I_I$ and $I_{II}$, and the Poincar\'e inequality 
\eqref{eq:poincare-Kh}, imply
\begin{align}
I &\lesssim h^{p+1} \| \nabla_{\Gamma} u\|_{L^2(\Gamma)} \| v \|_{a_h} 
+h^{p+1} \|f\|_{L^2(\Gamma)} \| v \|_{L^2(\Gamma_h)}
\\ \label{eq:estI}
&\lesssim 
( \| \nabla_{\Gamma} u\|_{L^2(\Gamma)} + \|f\|_{L^2(\Gamma)} )h^{p+1} \| v \|_{a_h}. 
\end{align}
Finally, using that $u$ is the solution of \eqref{eq:LB} we have the stability estimate 
$\| \nabla_{\Gamma} u\|_{L^2(\Gamma)}\lesssim \|f\|_{L^2(\Gamma)}$, and we obtain
\begin{equation}
I \lesssim h^{p+1} \|f\|_{L^2(\Gamma)} \| v \|_{a_h}. 
\end{equation}
\paragraph{Term $\bfI\bfI$.} 
Using the Cauchy-Schwarz inequality and Property P2 of the stabilization term, see Section \ref{sec:condonstab},  we obtain
\begin{align}\label{eq:estII}
|s_{h}(u^e,v)|
&\lesssim \|u^e\|_{s_h}\| v\|_{s_{h}} \lesssim
h^{p}\|u\|_{H^{p+1}(\Gamma)} \| v\|_{s_{h}}.
\end{align}

Combining the estimates of $I$ and $II$ we obtain the desired estimate
\begin{equation}
| A_h ( u^e,v)-l_h(v)| \lesssim h^{p}  \| u  \|_{H^{p+1}(\Gamma)} \| v \|_{A_h} 
+h^{p+1} \|f\|_{L^2(\Gamma)} \| v \|_{a_h}.
\end{equation}
\end{proof}

\subsection{Error Estimates}
We first prove optimal error estimates in the energy norm and then apply a duality argument to obtain $L^2$-error estimates

\begin{thm}\label{thm:apriori-energy} Let $u \in H^{p+1}(\Gamma)\cap H^1_0(\Gamma)$ be 
the solution of \eqref{eq:LB} and $u_h \in V_{h}^p$ the finite element approximation defined by \eqref{eq:Ah}. If assumptions A1-A3 hold and the stabilization term has properties P1-P5 then, 
there is a constant independent of the mesh size $h$ such that the following error bound holds
\begin{equation}
\|u^e-u_h\|_{A_h}  \lesssim h^{p}  \| u  \|_{H^{p+1}(\Gamma)} +h^{p+1} \|f\|_{L^2(\Gamma)}. 
\end{equation}
\end{thm}
\begin{proof}
Using the Strang Lemma followed by the bounds on the interpolation error (\ref{eq:interpolenergysurf}) 
and the consistency error (\ref{eq:consistency}) we get the desired bound.
\end{proof}

\begin{thm}\label{thm:apriori-L2} Let $u \in H^{p+1}(\Gamma)\cap H^1_0(\Gamma)$ be the solution of \eqref{eq:LB} and $u_h\in V_{h}^p$ the finite element approximation defined by \eqref{eq:Ah}. If assumptions A1-A3 hold and the stabilization term has properties P1-P5 then, there is a constant independent of the mesh size $h$ such that the following error bound holds 
\begin{equation}
\|u^e-u_h\|_{L^2(\Gamma_h)}  \lesssim h^{p+1}  \| u  \|_{H^{p+1}(\Gamma)} +h^{p+1} \|f\|_{L^2(\Gamma)}. 
\end{equation}
\end{thm}
\begin{proof} 
The proof is similar to the proof of the $L^2$-error estimate in \cite{BurHanLar15} for $p=1$.
Let $e_h=u^e-u_h|_{\Gamma_h}$ and its lift on $\Gamma$ be $e_h^l=u-u_h^l$. Recall that $\int_{\Gamma} u \,ds=0$ and $\int_{\Gamma_h} u_h \,ds_h=0$. We now define $\tilde{u}_h\in V_h^p$ as
\begin{equation}\label{eq:uhtilde}
\tilde{u}_h=u_h-|\Gamma|^{-1} \int_\Gamma u_h^l ds
\end{equation}
so that we have $\int_\Gamma \tilde{u}_h^l=0$. By adding and subtracting $\tilde{u}_h^l$ and using the triangle inequality we obtain
\begin{equation}\label{eq:ehl}
\|e_h^l\|_{L^2(\Gamma)} 
\leq \underbrace{\|u-\tilde{u}_h^l \|_{L^2(\Gamma)}}_{I} +\underbrace{\|\tilde{u}_h^l-u_h^l \|_{L^2(\Gamma)}}_{II}.
\end{equation}
\paragraph{Term $\bfI$.} Consider the dual problem:
\begin{equation} \label{eq:dual}
a(v,\phi)=(\psi,v)_\Gamma, \qquad \psi \in L^2(\Gamma)\setminus \IR.
\end{equation} 
It follows from the Lax-Milgram lemma that there exists a unique solution in 
$H^1(\Gamma)\setminus \IR$ and we also have the elliptic regularity estimate 
\begin{equation}\label{eq:ellipticreg}
\|\phi\|_{H^2(\Gamma)}\lesssim \| \psi \|_{L^2(\Gamma)}.
\end{equation}
Setting $v = u-\tilde{u}_h^l$ in (\ref{eq:dual}), adding and subtracting an interpolant, and 
using the weak formulations in \eqref{eq:LB} and \eqref{eq:Ah} we obtain
\begin{align}
(u-\tilde{u}_h^l, \psi )_\Gamma
&=a(u-\tilde{u}_h^l,\phi)
\\
&=a(u-\tilde{u}_h^l,\phi-(\pi_h^p \phi^e)^l)+a(u-\tilde{u}_h^l,(\pi_h^p \phi^e)^l)
\\
&=\underbrace{a(u-\tilde{u}_h^l,\phi-(\pi_h^p \phi^e)^l)}_{I_I}+\underbrace{l((\pi_h^p \phi^e)^l)-l_h(\pi_h^p \phi^e)}_{I_{II}} 
\\
&\qquad +\underbrace{a_h(\tilde{u}_h,\pi_h^p \phi^e) -a(\tilde{u}_h^l,(\pi_h^p \phi^e)^l)}_{I_{III}}+\underbrace{s_h(\tilde{u}_h,\pi_h^p. \phi^e)}_{I_{IV}}
\end{align}  
\paragraph{Term $\bfI_{\bfI}$.}
Using the Cauchy-Schwarz inequality, the norm equivalences (see Section \ref{sec:prel}), 
and the energy norm estimate (Theorem \ref{thm:apriori-energy}) we obtain
\begin{align}
I_{I}&\lesssim \|\nabla_\Gamma (u-\tilde{u}_h^l) \|_{L^2(\Gamma)}\|\nabla_\Gamma (\phi-(\pi_h^p \phi^e)^l) \|_{L^2(\Gamma)} 
\\
&\lesssim
\|u^e -u_h \|_{a_h} \|\phi^e-\pi_h^p \phi^e \|_{a_h} 
\\
&\lesssim  (h^{p}\|u\|_{H^{p+1}(\Gamma)}+h^{p+1}\|f\|_{L^{2}(\Gamma)}) \|\phi^e-\pi_h^p \phi^e \|_{a_h}.
\end{align}
The interpolation error estimate \eqref{eq:interpolsurf} with $s=1$ yields
\begin{equation} \label{eq:phiinterp}
\|\phi^e-\pi_h^p \phi^e \|_{a_h} \lesssim h \|\phi\|_{H^2(\Gamma)}
\end{equation}
and finally using the elliptic regularity estimate \eqref{eq:ellipticreg} we get  
\begin{equation}
I_{I}\lesssim  (h^{p+1}\|u\|_{H^{p+1}(\Gamma)}+h^{p+2}\|f\|_{L^{2}(\Gamma)}) \| \psi \|_{L^2(\Gamma)}.
\end{equation}
\paragraph{Term $\bfI_{\bfI\bfI}-\bfI_{\bfI\bfI\bfI}$.}
We use the estimate of Term $I$ in the proof of Lemma \ref{lem:consistency},  
together with the following stability estimate
\begin{equation} \label{eq:stabest}
\| \nabla_{\Gamma} \tilde{u}_h^l\|_{L^2(\Gamma)}=\| \nabla_{\Gamma} u_h^l\|_{L^2(\Gamma)} \lesssim \|u_h\|_{a_h}\lesssim  \|f_h\|_{L^2(\Gamma_h)} \lesssim \|f\|_{L^2(\Gamma)}
\end{equation}
to get
\begin{align}
|I_{II}+I_{III}| 
&\lesssim (h^{p+1} \| \nabla_{\Gamma} \tilde{u}_h^l\|_{L^2(\Gamma)}+h^{p+1} \| f\|_{L^{2}(\Gamma)})\|\pi_h^p \phi^e\|_{a_h}.
\lesssim h^{p+1}  \| f\|_{L^{2}(\Gamma)} \|\pi_h^p \phi^e\|_{a_h}
\end{align} 
Adding and subtracting an interpolant, using the interpolation error estimate 
\eqref{eq:phiinterp} followed by the elliptic regularity estimate \eqref{eq:ellipticreg} yield
\begin{align}
\|\pi_h^p \phi^e\|_{a_h}\lesssim \|\pi_h^p \phi^e-\phi^e\|_{a_h}+\|\phi^e\|_{a_h}\lesssim \|\phi\|_{H^2(\Gamma)}\lesssim \| \psi \|_{L^2(\Gamma)}.
\end{align}
Thus, 
\begin{align}
|I_{II}+I_{III}| &\lesssim h^{p+1}  \| f\|_{L^{2}(\Gamma)} \| \psi \|_{L^2(\Gamma)}.
\end{align}
\paragraph{Term $\bfI_{\bfI\bfV}$.} We have
\begin{equation}
|I_{IV}|=|s_h(u_h,\pi_h^p \phi^e)|\lesssim \|u_h \|_{s_h}\|\pi_h^p \phi^e\|_{s_h}.
\end{equation}
Using the energy norm error estimate, Theorem \ref{thm:apriori-energy}, and 
Property P2 of the stabilization  we obtain 
\begin{align}
\|u_h \|_{s_h}
&\lesssim \|u_h -u^e\|_{s_h}+ \| u^e \|_{s_h}
\\
&\lesssim h^{p}\|u\|_{H^{p+1}(\Gamma)}+h^{p+1}\|f\|_{L^{2}(\Gamma)}+h^{p}\|u\|_{H^{p+1}(\Gamma)}
\end{align}
and using Property P3 and the elliptic regularity estimate \eqref{eq:ellipticreg} 
we obtain
\begin{equation}
\|\pi_h^p \phi^e\|_{s_h} \lesssim h \| \phi \|_{H^2(\Gamma)} \lesssim h \| \psi \|_{L^2(\Gamma)}.
\end{equation}
We thus have
\begin{equation}
|I_{IV}| 
\lesssim
( h^{p+1}\| u \|_{H^{p+1}(\Gamma_h)}+h^{p+2}\|f\|_{L^{2}(\Gamma)})
\| \psi \|_{L^2(\Gamma)}.
\end{equation}

\paragraph{Final Estimate of Term $\bfI.$} Collecting the estimates of Terms $I_I-I_{IV}$ 
and taking $\psi=u-\tilde{u}_h^l/ \|u-\tilde{u}_h^l\|_{L^2(\Gamma)}$ we obtain
\begin{equation}
I=\|u-\tilde{u}_h^l \|_{L^{2}(\Gamma)} \lesssim h^{p+1}\| u \|_{H^{p+1}(\Gamma_h)}+h^{p+1}\|f\|_{L^{2}(\Gamma)}.
\end{equation}

\paragraph{Term $\bfI\bfI$.} Using the definition \eqref{eq:uhtilde} of $\tilde{u}_h$, adding $\int_{\Gamma_h} u_h \,ds_h=0$, changing the domain of integration and using related bounds, 
employing the Poincar\'e inequality \eqref{eq:poincare-Kh}, and the stability estimate \eqref{eq:stabest}, we obtain
\begin{align} 
\|\tilde{u}_h^l-u_h^l \|_{L^2(\Gamma)}
&=\left|\int_\Gamma u_h^l ds- \int_{\Gamma_h} u_h ds_h \right|
\\
&\lesssim \left|\int_{\Gamma_h} u_h(|B|-1) ds_h \right|
\\ 
& \lesssim h^{p+1} \|u_h \|_{L^2(\Gamma_h)}
\\
&\lesssim h^{p+1} \|u_h \|_{a_h}
\\
&\lesssim h^{p+1} \|f\|_{L^2(\Gamma)}. 
\end{align}

\paragraph{Conclusion of the Proof.} Using the estimates of Terms $I$ and $II$ in equation 
\eqref{eq:ehl} and the equivalence $\|e_h ^l\|_{L^2(\Gamma)}\sim \| e_h \|_{L^2(\Gammah)}$ 
conclude the proof. 
\end{proof}

\section{Numerical Examples}\label{sec:numexpS}
We consider three different problems. The mesh on the domain where the interface is embedded is always generated independently of the position of the interface $\Gamma$ and we use Lagrange basis functions.

\subsection{The Laplace-Beltrami Problem}
Let the interface $\Gamma$ be a circle of radius 1 centered at the origin. We generate a uniform triangular mesh with $h=h_{x_1}=h_{x_2}$ on the computational domain: $[-1.5, \ 1.5]\times [-1.5, \ 1.5]$. 
A right-hand side $f$ to equation~\eqref{eq:LB} is calculated so that
\begin{equation} \label{eq:sol2}
u(x)=\frac{x_1^3x_2^3}{(x_1^2+x_2^2)^3}
\end{equation} 
is the exact solution. We discretize~\eqref{eq:LB} using the proposed CutFEM with Lagrange basis functions of degree p=1, 2, and 3.

In Fig.~\ref{fig:errordiffelem} we show the error and the spectral condition number of the linear systems as a function of the mesh size $h$. The error measured in the L$^2$-norm behaves as $\mathcal{O}(h^{p+1})$ and in the $H^1$-norm as $\mathcal{O}(h^{p})$, as expected. The condition number behaves as $\mathcal{O}(h^{-2})$ independent of the polynomial degree and how the interface cuts the background mesh. 
The magnitude of the error and also the condition number depends on the stabilization constants in \eqref{eq:sh-face}-\eqref{eq:sh-gamma} and also on the choice of basis functions. For the results in Fig.~\ref{fig:errordiffelem} we have used $\gamma=1$, $c_{F,j}=c_{\Gamma,j}=2.5\cdot 10^{-j}$, $j=1, \cdots, p$. We have not optimized this choice of parameters and other choices may give better results for example smaller condition numbers with almost the same error. We show this for example in Fig.~\ref{fig:errordiffstabP3} and~\ref{fig:diffparamP3}.  Other stabilization constants or basis functions can also give better scaling of the condition number with respect to the polynomial degree.  

\begin{figure}\centering
\includegraphics[width=0.45\textwidth]{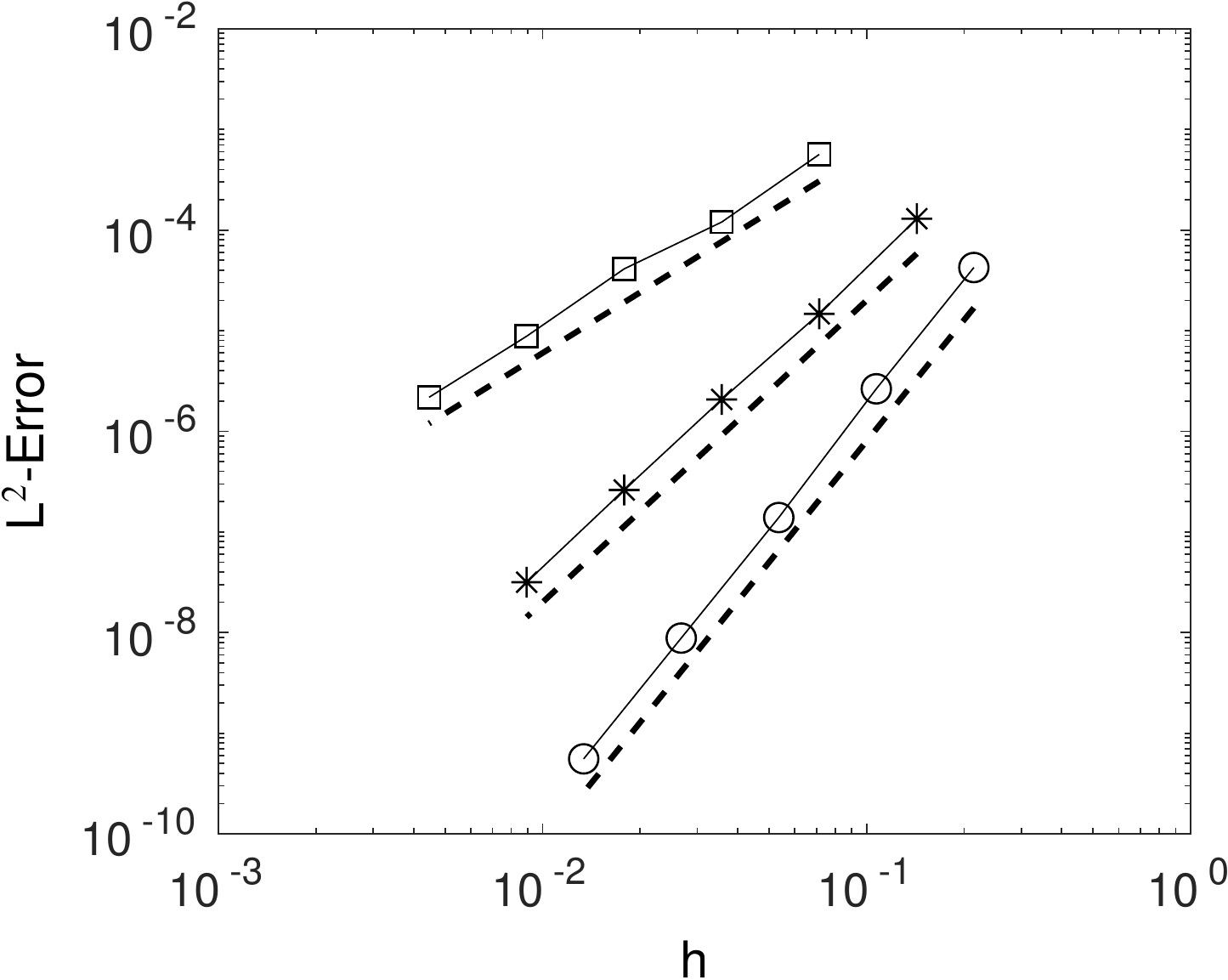} \\
\includegraphics[width=0.45\textwidth]{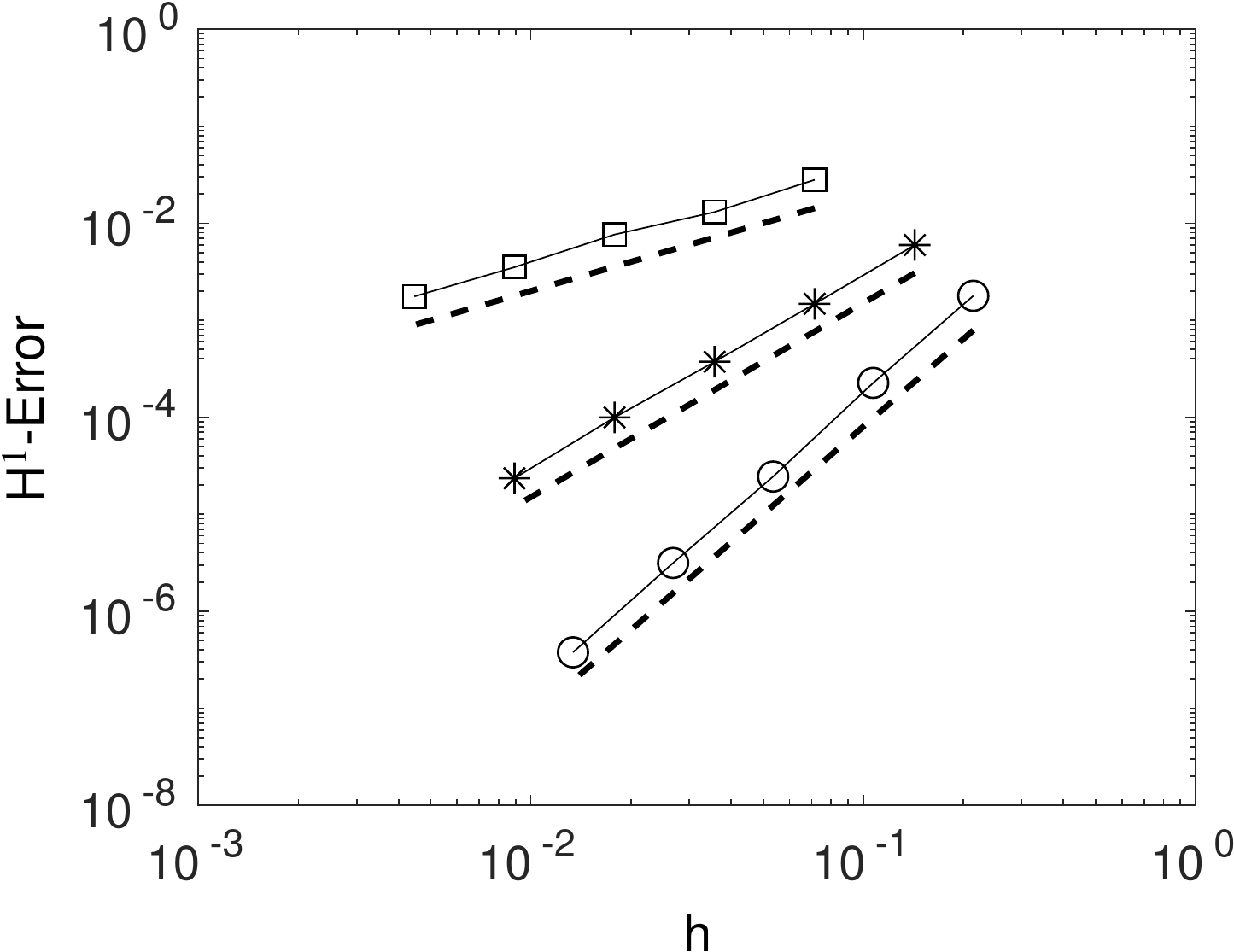} \\
\includegraphics[width=0.45\textwidth]{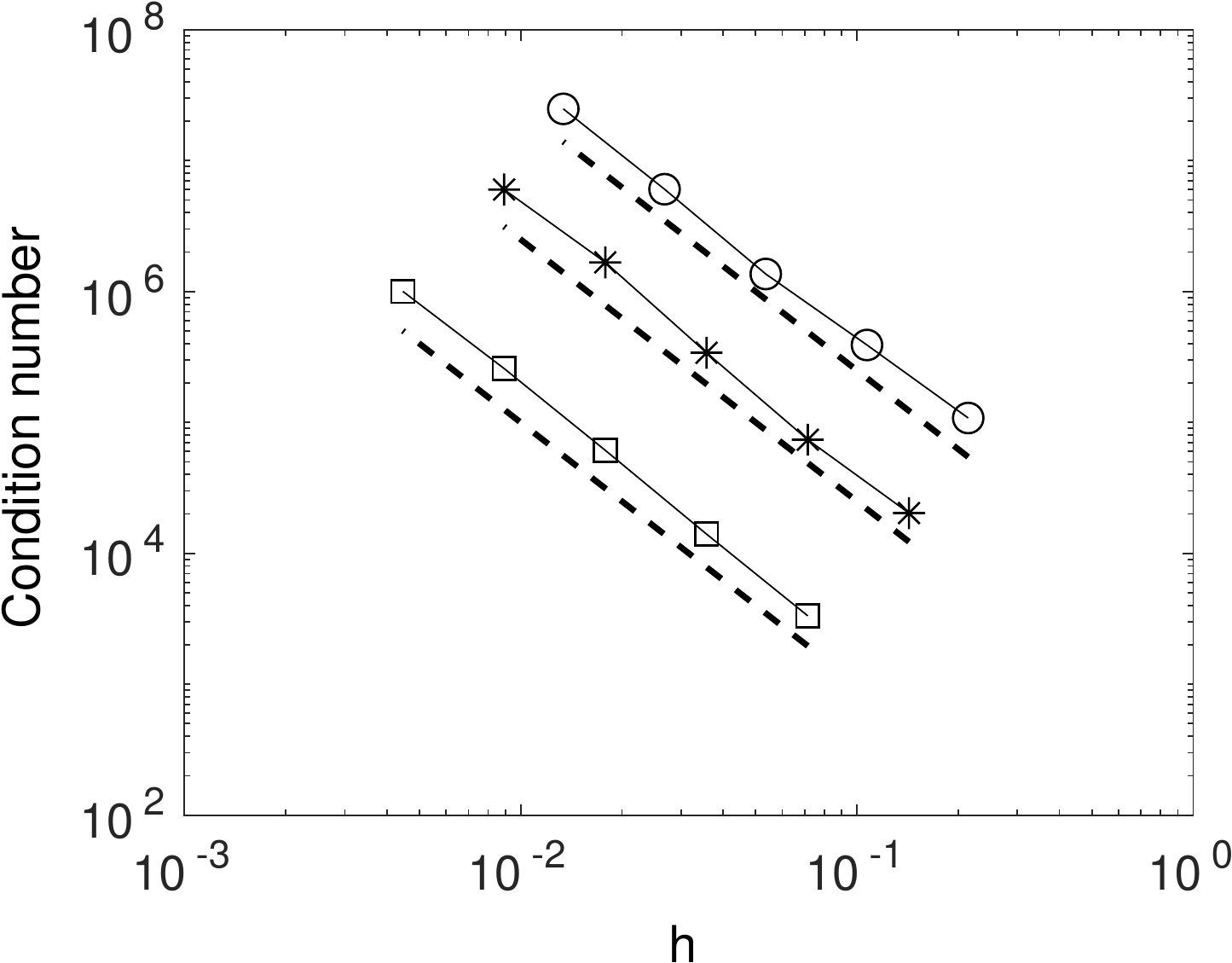} 
\caption{The Laplace-Beltrami problem:
 The error and condition number versus mesh size $h$ for different degrees of polynomials, $p$, in the discretization. Squares: p=1. Stars: p=2. Circles: p=3.  Top: The error measured in the L$^2$-norm versus mesh size $h$.  The dashed lines are indicating the expected rate of convergence and are proportional to  $h^{p+1}$.  Middle: The error measured in the $H^1$-norm versus mesh size $h$.  The dashed lines are indicating the expected rate of convergence and are proportional to $h^p$.  Bottom: The condition number versus mesh size $h$. The dashed lines are all proportional to $h^{-2}$.  \label{fig:errordiffelem}}
\end{figure}

In Fig.~\ref{fig:errordiffstabP1} and~\ref{fig:errordiffstabP3} we show the effect of the stabilization on the error and the condition number for linear (p=1) and cubic elements (p=3), respectively. We show the convergence of the error with respect to the $L^2$-norm. The results are similar when the error is measured in the $H^1$-norm. 
We compare the proposed stabilization with the stabilization terms~\eqref{eq:stabface} and \eqref{eq:normalgradel} and also show the condition number after diagonal scaling. 
\begin{figure}
\centering
\includegraphics[width=0.43\textwidth]{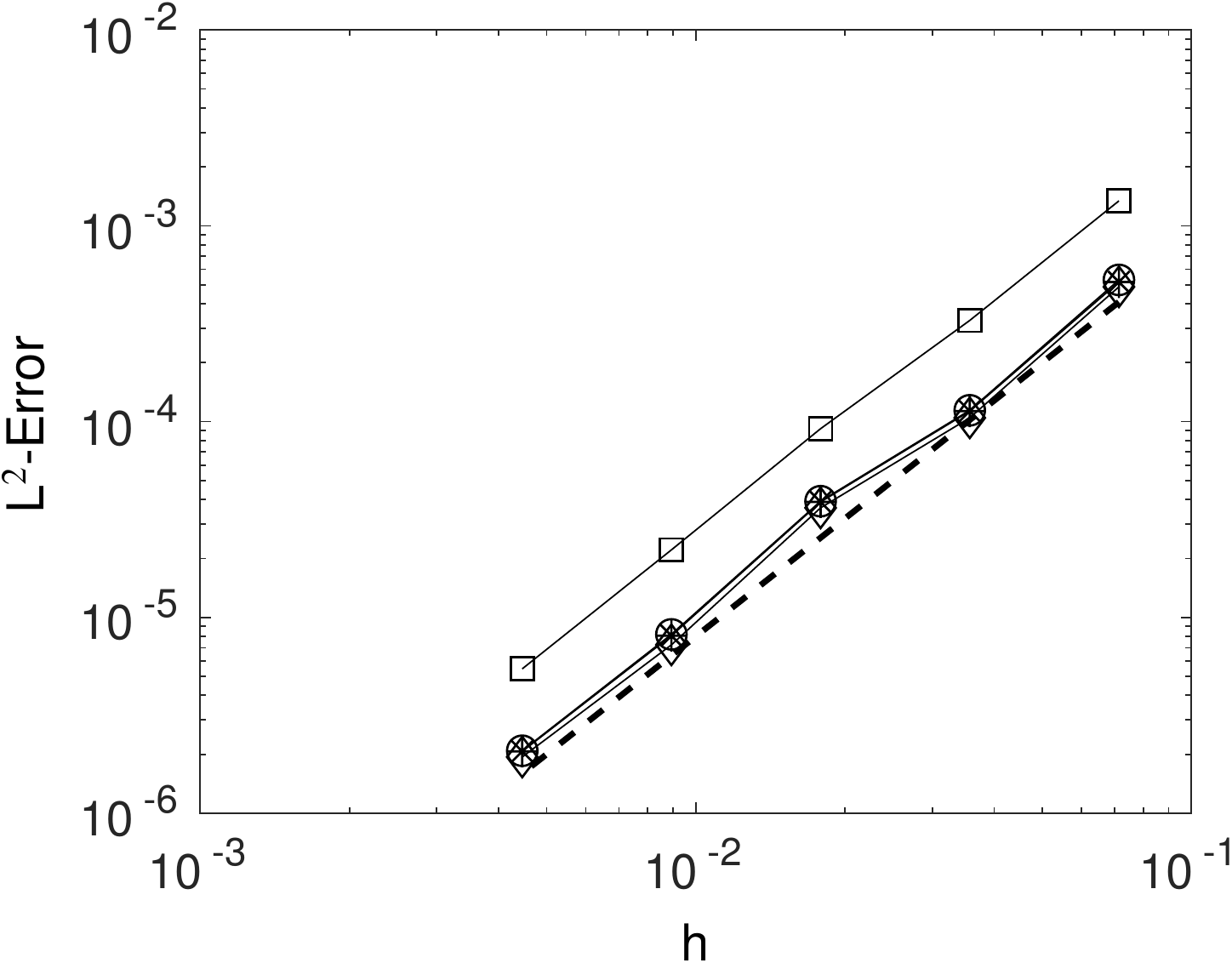} \\
\vspace{0.4cm}
\includegraphics[width=0.4\textwidth]{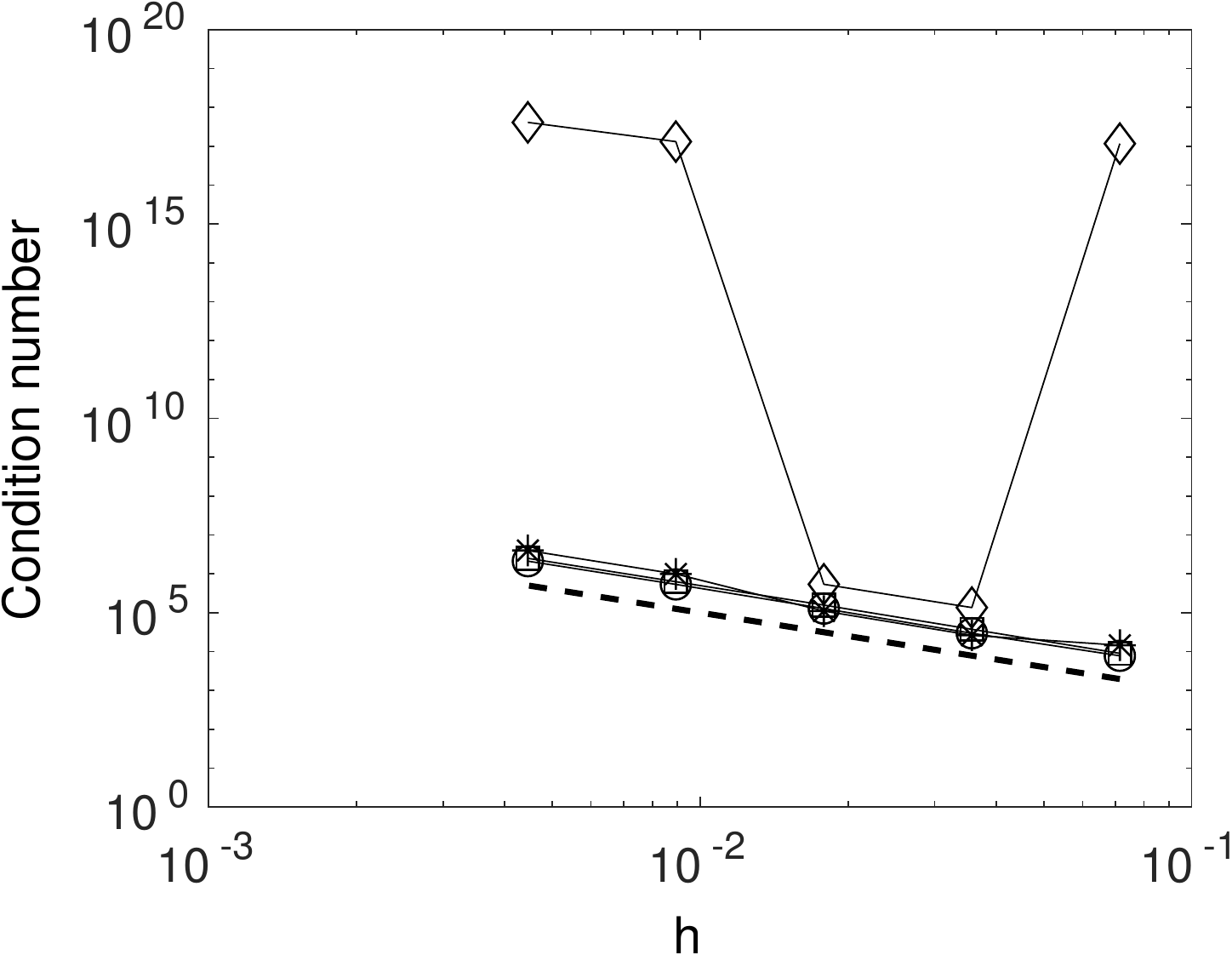} \hspace{0.3cm}
 \includegraphics[width=0.4\textwidth]{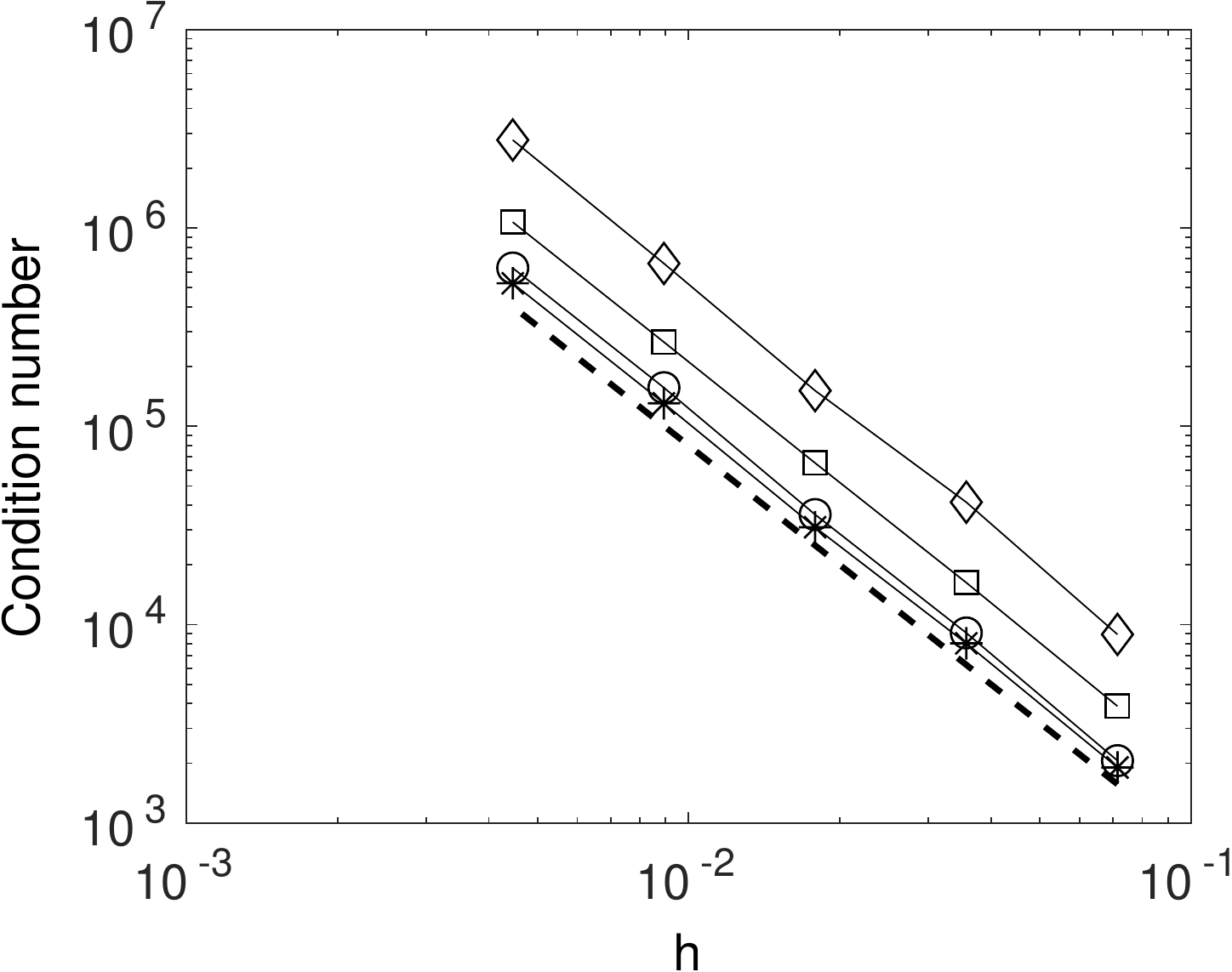} 
\caption{The Laplace-Beltrami problem:
 The error and condition number versus mesh size $h$ using different stabilization terms. Linear elements are used, i.e. p=1. Diamonds: no stabilization is added. Circles: the proposed stabilization with $\gamma=1$. Squares: the pure face stabilization~\eqref{eq:stabface}. Stars: the normal gradient stabilization~\eqref{eq:normalgradel} with $\alpha=1$. Top: The error measured in the $L^2$-norm versus mesh size $h$.  The dashed line is indicating the expected rate of convergence and is proportional to  $h^{2}$.  Bottom left: The spectral condition number versus mesh size $h$. The dashed line is proportional to $h^{-2}$. Bottom right: The spectral condition number after diagonal scaling versus mesh size $h$. The dashed line is proportional to $h^{-2}$. \label{fig:errordiffstabP1}}
\end{figure} 

We see in Fig.~\ref{fig:errordiffstabP1} that for linear elements the errors, measured in the $L^2$-norm, using the different stabilization terms and also no stabilization are very similar but we clearly see that when no stabilization is used the condition number can be  extremely large depending on the position of the interface relative the background mesh. We have used $c_{F,1}=c_{\Gamma,1}=10^{-1}$ and $\gamma=1$ in the proposed stabilization. Note that the stabilization term \eqref{eq:stabface} developed in~\cite{BurHanLar15}  can be written in the form of the proposed stabilization, see \eqref{eq:sh}-\eqref{eq:sh-gamma}, but with $c_{\Gamma,j}=0$ for all $j$ and $\gamma=0$. For this stabilization we have chosen $c_{F,1}=10^{-1}$ but since $\gamma=0$ we get slightly larger errors than the other alternatives. Decreasing the stabilization constant $c_{F,1}$ will decrease the error. In the stabilization term \eqref{eq:normalgradel} we have chosen $c_{\mcT}=10^{-1}$ and $\alpha=1$. We see in Fig.~\ref{fig:errordiffstabP1} that an alternative to adding stabilization terms is to do a simple diagonal scaling.

For higher order elements than linear the resulting linear systems from unstabilized CutFEM are severely ill-conditioned and neither a simple diagonal scaling alone or a pure face stabilization term such as \eqref{eq:stabface} improves the condition number significantly, see Fig.~\ref{fig:errordiffstabP3}. When the condition number becomes to large the convergence fails, the error is dominated by roundoff errors and does not decrease with mesh refinement. In general, when the stabilization is not enough the behavior of the error depends on how the geometry cuts the background mesh. Therefore, if the approximation of the interface is slightly changed the magnitude of the error and the convergence shown in Fig.~\ref{fig:errordiffstabP3}, where cubic elements are used, may change a lot for the unstabilized method and the method with only face stabilization.
With the proposed stabilization and also with the stabilization term \eqref{eq:normalgradel} we obtain as expected optimal convergence orders and the condition number of the linear systems scales as $\mathcal{O}(h^{-2})$ independent of how the interface cuts the background mesh.
In Fig.~\ref{fig:errordiffstabP3} we also show results with different constants in the stabilization terms. We see that the magnitude of the error decreases as the stabilization term gets weaker. However, when the stabilization is too weak the condition number is large and roundoff errors dominate.
\begin{figure}\centering
\includegraphics[width=0.45\textwidth]{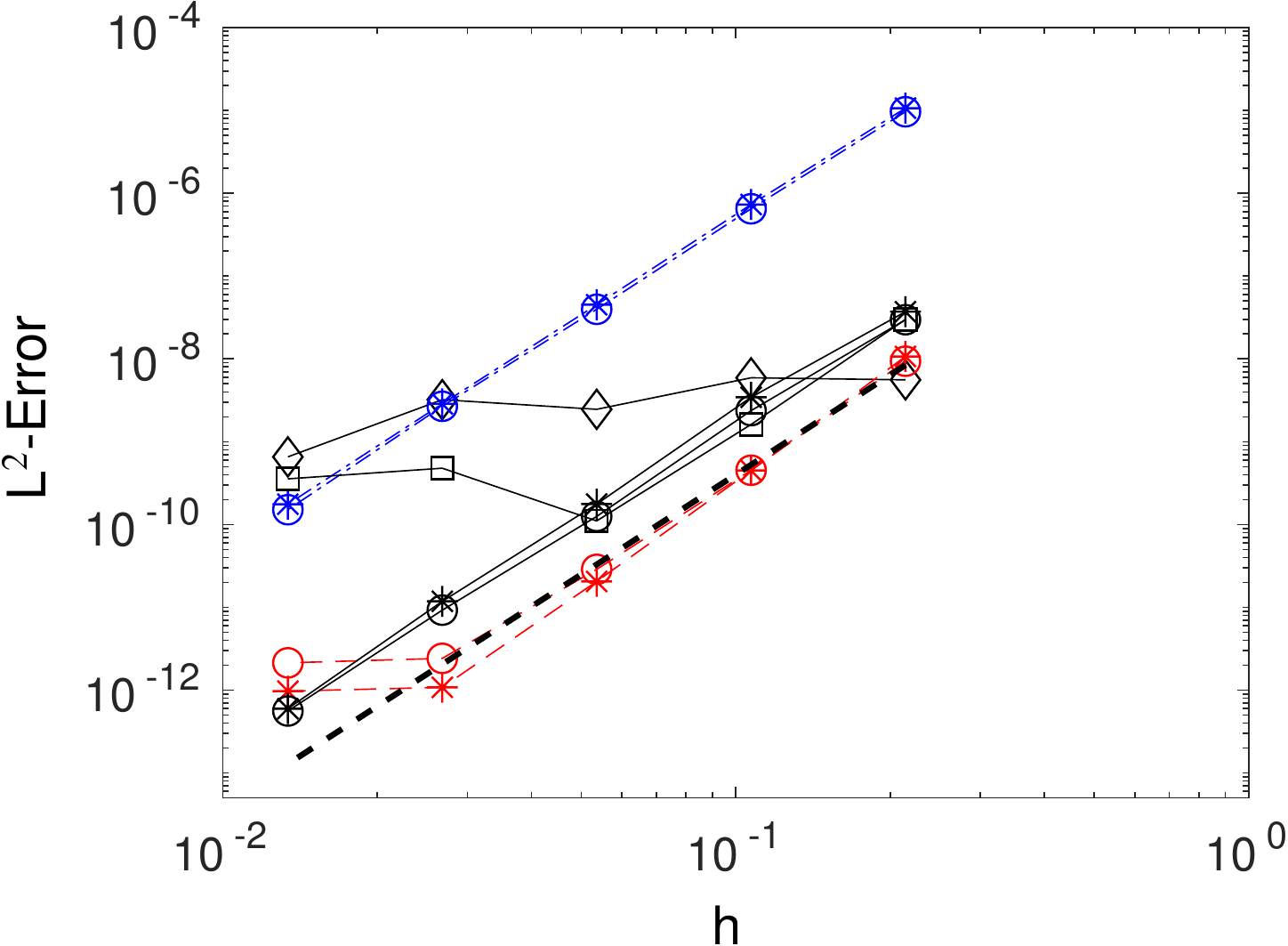}\\
\vspace{0.2cm}
\includegraphics[width=0.4\textwidth]{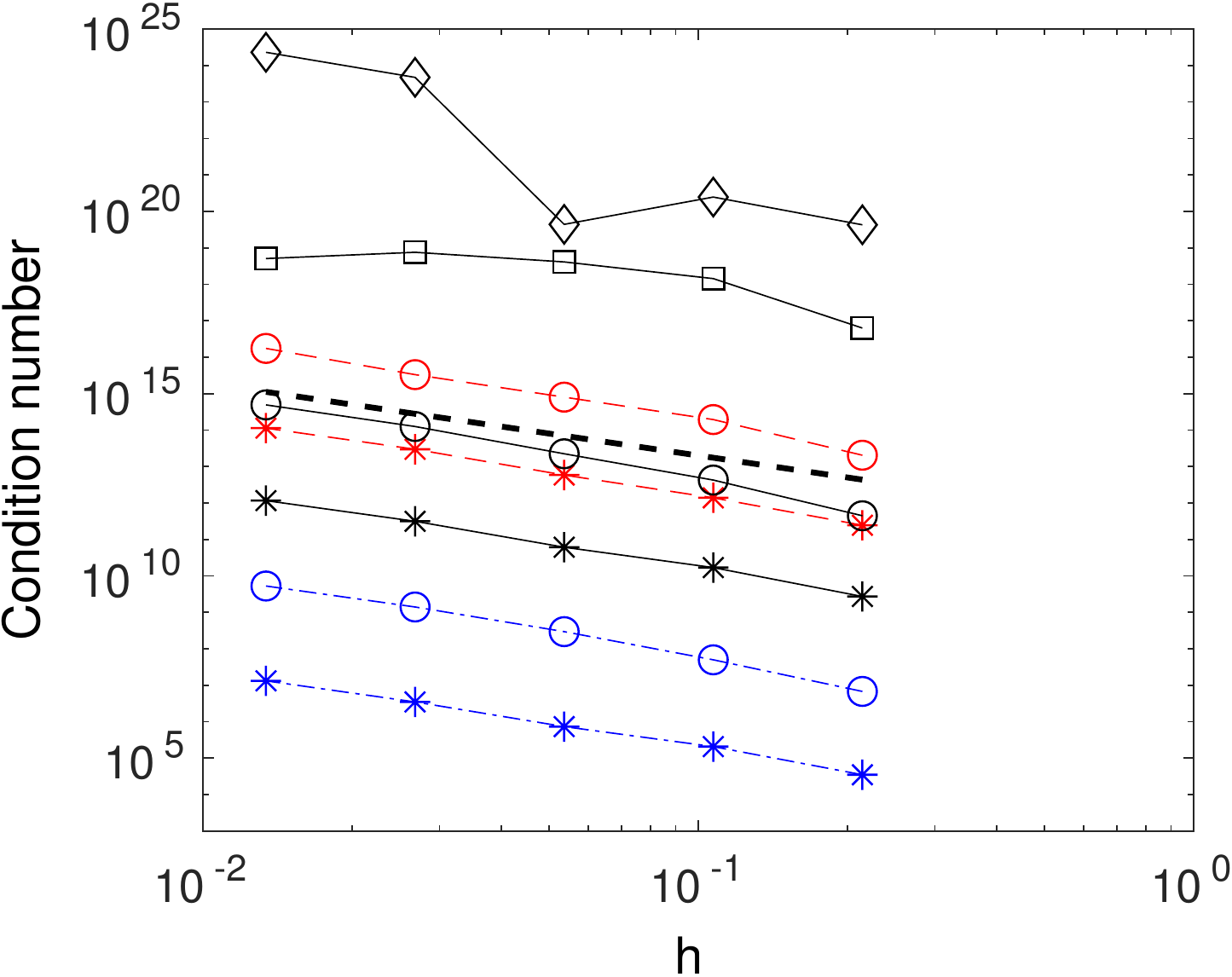} \hspace{0.3cm}
\includegraphics[width=0.4\textwidth]{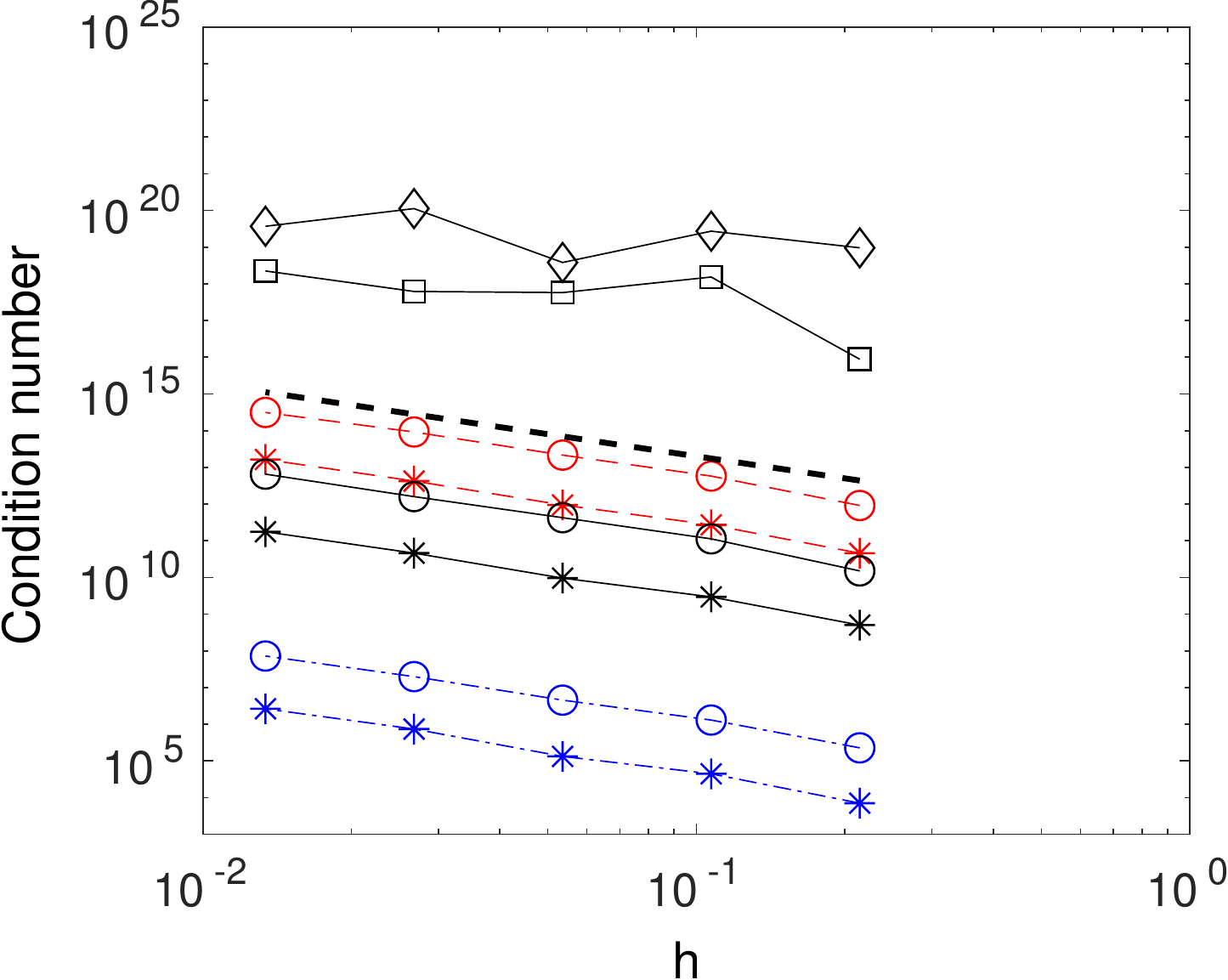} 
\caption{The Laplace-Beltrami problem: 
 The error and condition number versus mesh size $h$ using different stabilization terms. Cubic elements are used, i.e. p=3. Diamonds: no stabilization is added. Circles: the proposed stabilization with $\gamma=1$. Squares: pure face stabilization. Stars: the stabilization term \eqref{eq:normalgradel} with $\alpha=1$. Top: The error measured in the $L^2$-norm versus mesh size $h$.  The dashed line is indicating the expected rate of convergence and is proportional to $h^{4}$.  Bottom left: The spectral condition number versus mesh size $h$. The dashed line is proportional to $h^{-2}$. Bottom right: The spectral condition number after diagonal scaling versus mesh size $h$. The dashed line is proportional to $h^{-2}$. Symbols connected with dashed dotted lines have larger constants in the stabilization terms than symbols connected with solid lines and symbols connected with dashed lines which have the smallest constants in the stabilization terms.
  \label{fig:errordiffstabP3}}
\end{figure}

In Fig.~\ref{fig:errordiffstabP3} we see that the proposed stabilization and the stabilization term~\eqref{eq:normalgradel} resulted in very similar errors however the condition number is smaller when the stabilization term \eqref{eq:normalgradel} is added to the weak form. The difference in the condition number decreases after a diagonal scaling of both matrices. We emphasize that all these results depend on the choice of parameters in the stabilization terms and the basis functions used. We have not optimized the parameters in the stabilization terms. For example, using cubic elements and the proposed stabilization term there are six constants and the parameter $\gamma$ that one can modify to optimize the magnitude of the error and the condition number.

Next we vary both $\gamma$ and the constants $c_{F,j}$ and $c_{\Gamma,j}$ in the proposed stabilization term and study the error and the condition number in Fig.~\ref{fig:diffparamP1} and~\ref{fig:diffparamP3} for linear and cubic elements, respectively. By decreasing the constants $c_{F,j}$ and $c_{\Gamma,j}$ and/or increasing $\gamma$ we weaken the stabilization.  In general, the error decreases while the condition number increases as the stabilization is weakened. The condition number also increases when the stabilization becomes too strong. We see this for the linear elements as well as the cubic elements. A strengthening of the stabilization sometimes decreases the condition number significantly but only increases the error slightly. Compare for example the black circles connected with a dashed dotted line, $\gamma=0$ in Fig.~\ref{fig:diffparamP3} with the black circles connected with a solid line $\gamma=1$ in the same figure. By lowering $\gamma$ the error increases with a factor of $2.5$ in this case but the condition number decreases with a factor $850$. Comparing the results in Fig.~\ref{fig:diffparamP1} and~\ref{fig:diffparamP3} we see that the choice of parameters in the stabilization term have larger impact on the cubic elements than the linear elements. We also see that with other choices of parameters than the ones used for the results in Fig.~\ref{fig:errordiffelem} the condition number can scale much better with respect to the polynomial degree.

\begin{figure}\centering
\includegraphics[width=0.45\textwidth]{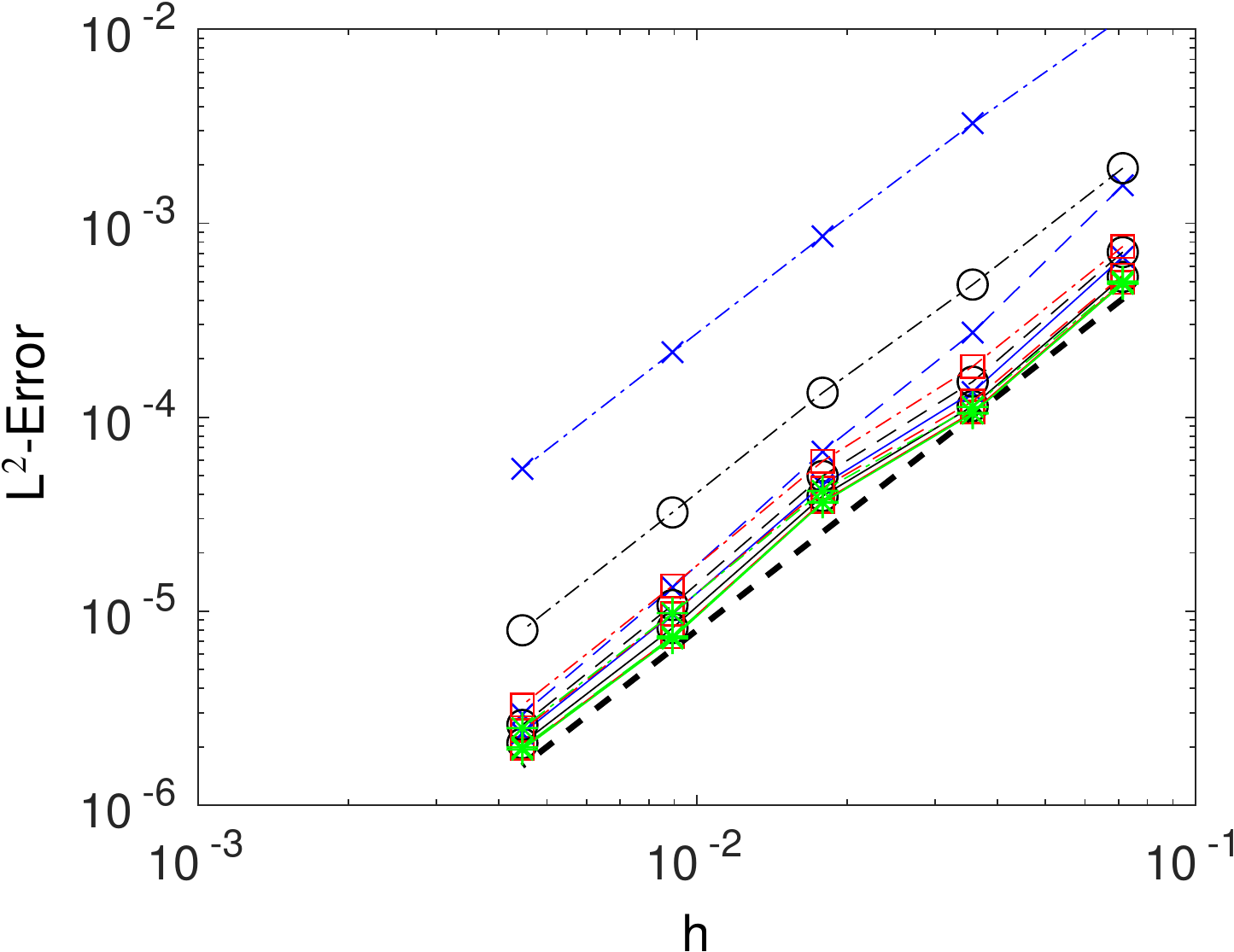} \hspace{0.2cm}
\includegraphics[width=0.44\textwidth]{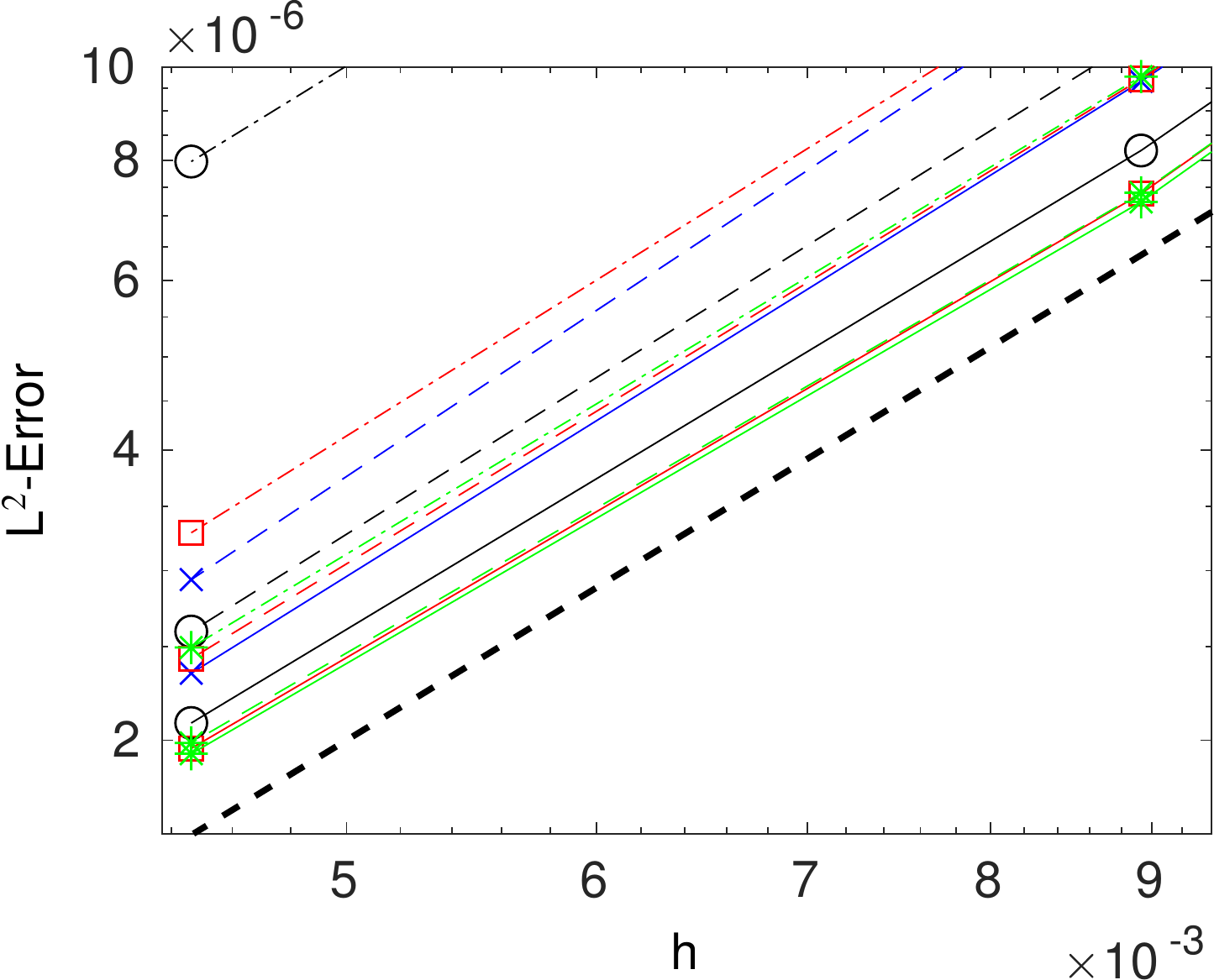} \\
\includegraphics[width=0.45\textwidth]{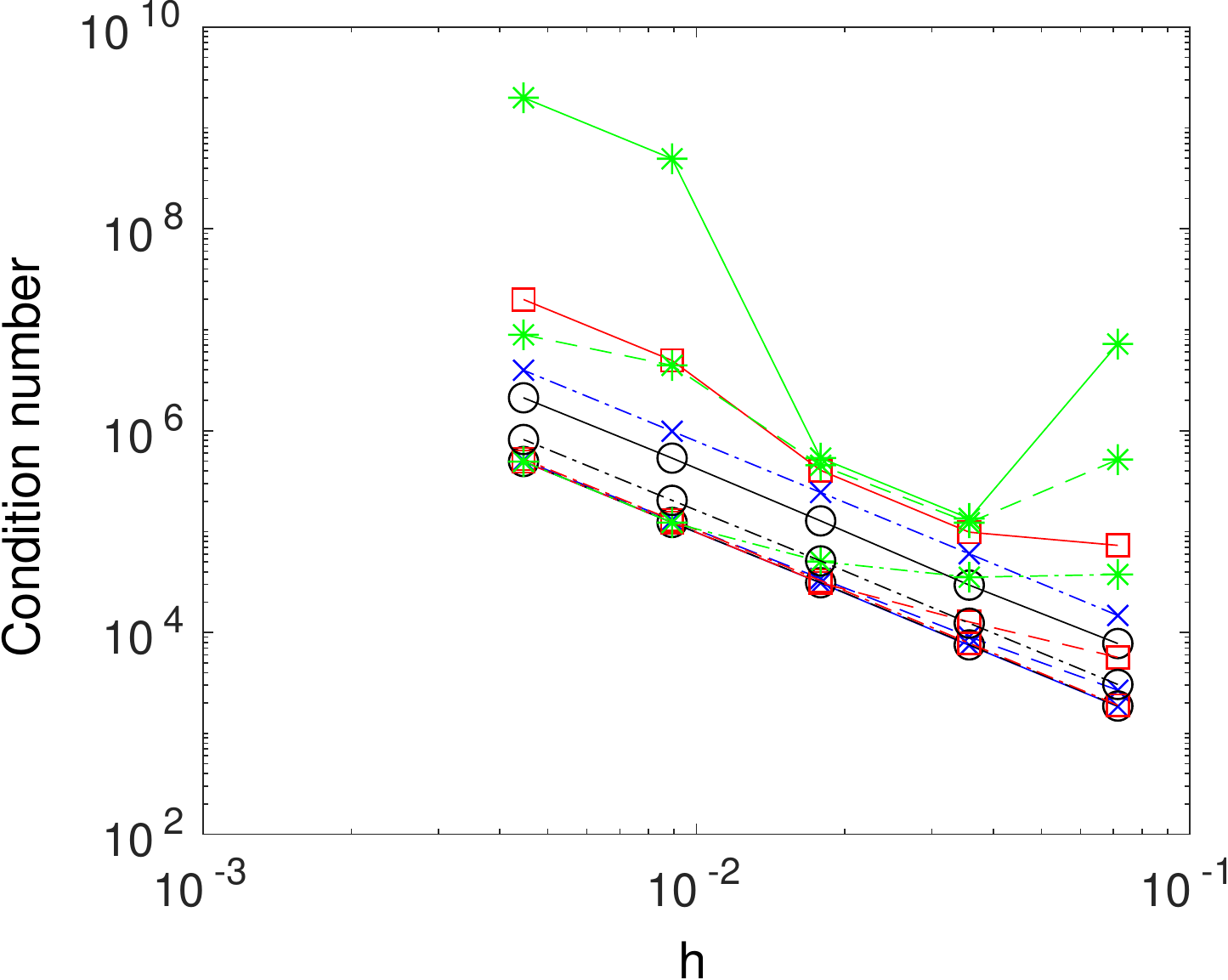} 
\hspace{0.2cm}
\includegraphics[width=0.44\textwidth]{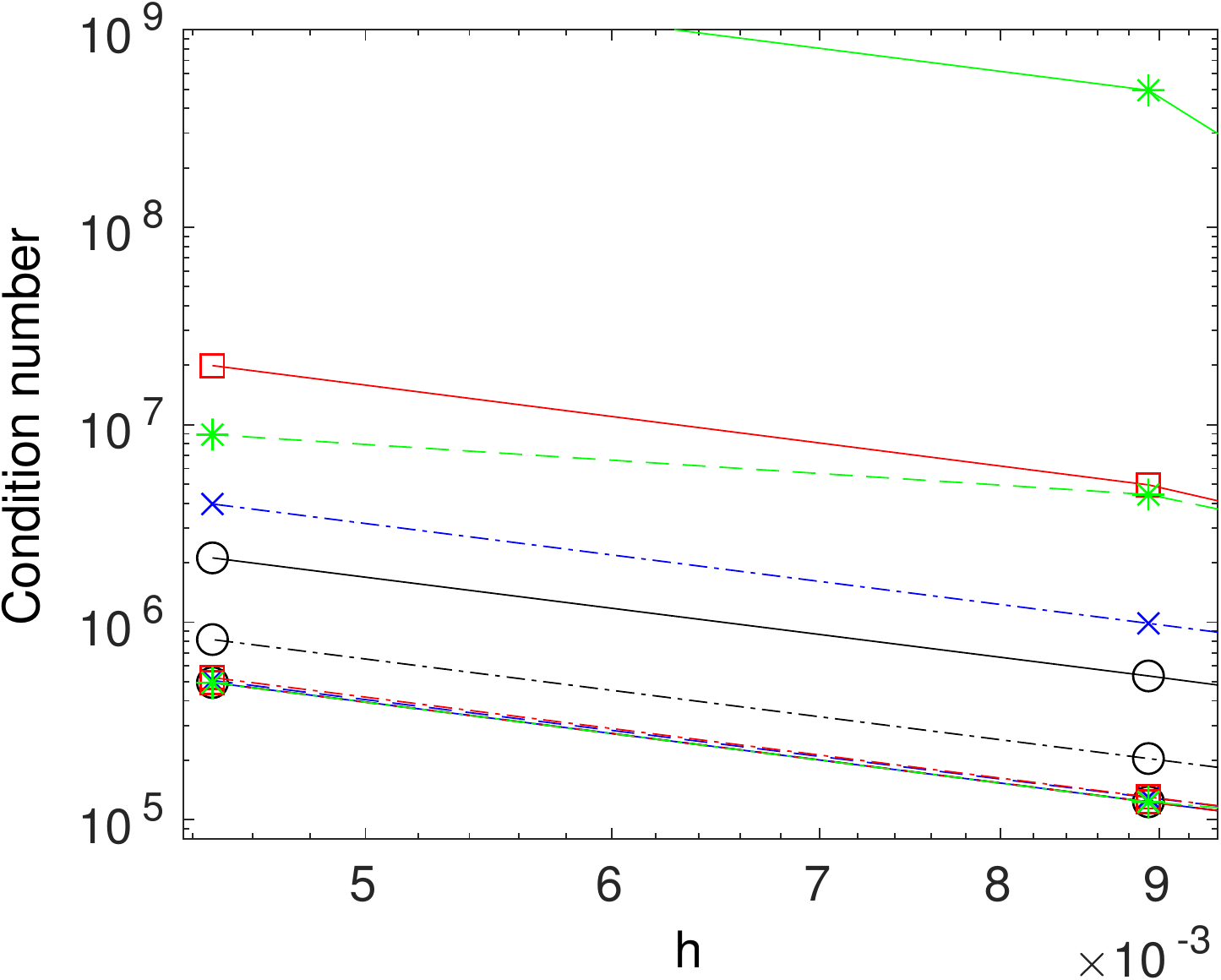} 
\caption{The Laplace-Beltrami problem: 
 The error and the condition number versus mesh size $h$ varying the stabilization constants and the parameter $\gamma$ in \eqref{eq:sh}-\eqref{eq:sh-gamma}. Linear elements are used, i.e. p=1. Solid lines: $\gamma=1$. Dashed lines: $\gamma=1/2$. Dashed dotted lines: $\gamma=0$. $c_{F,1}=c_{\Gamma,1}= c$. Blue crosses: $c= 1$, Black circles: $c= 10^{-1}$. Red squares: $c= 10^{-2}$. Green stars: $c= 10^{-4}$. The figures to the right are close ups on part of the figures to the left so that the different curves can be distinguished.  \label{fig:diffparamP1}}
\end{figure} 

\begin{figure}\centering
\includegraphics[width=0.45\textwidth]{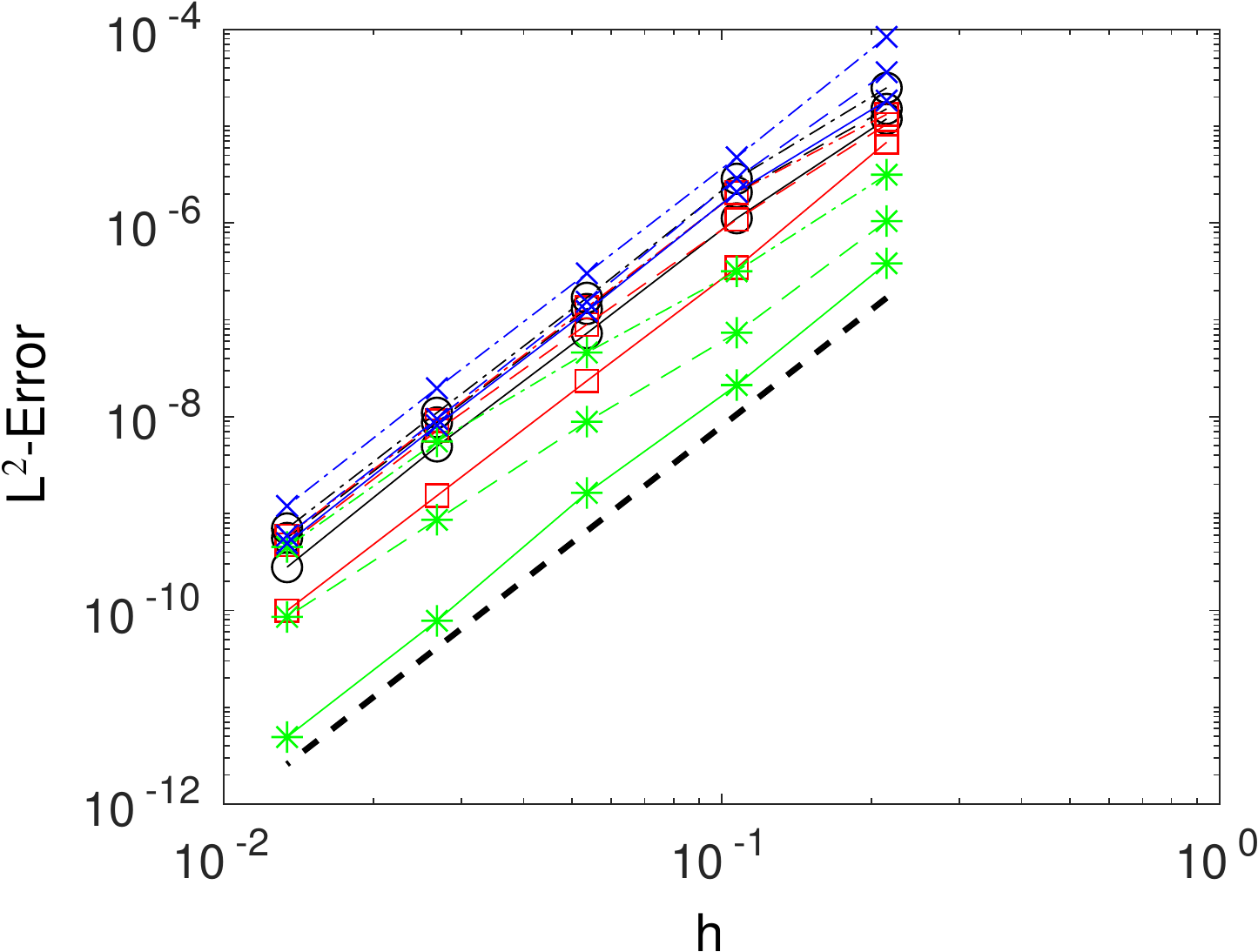} \hspace{0.2cm}
\includegraphics[width=0.42\textwidth]{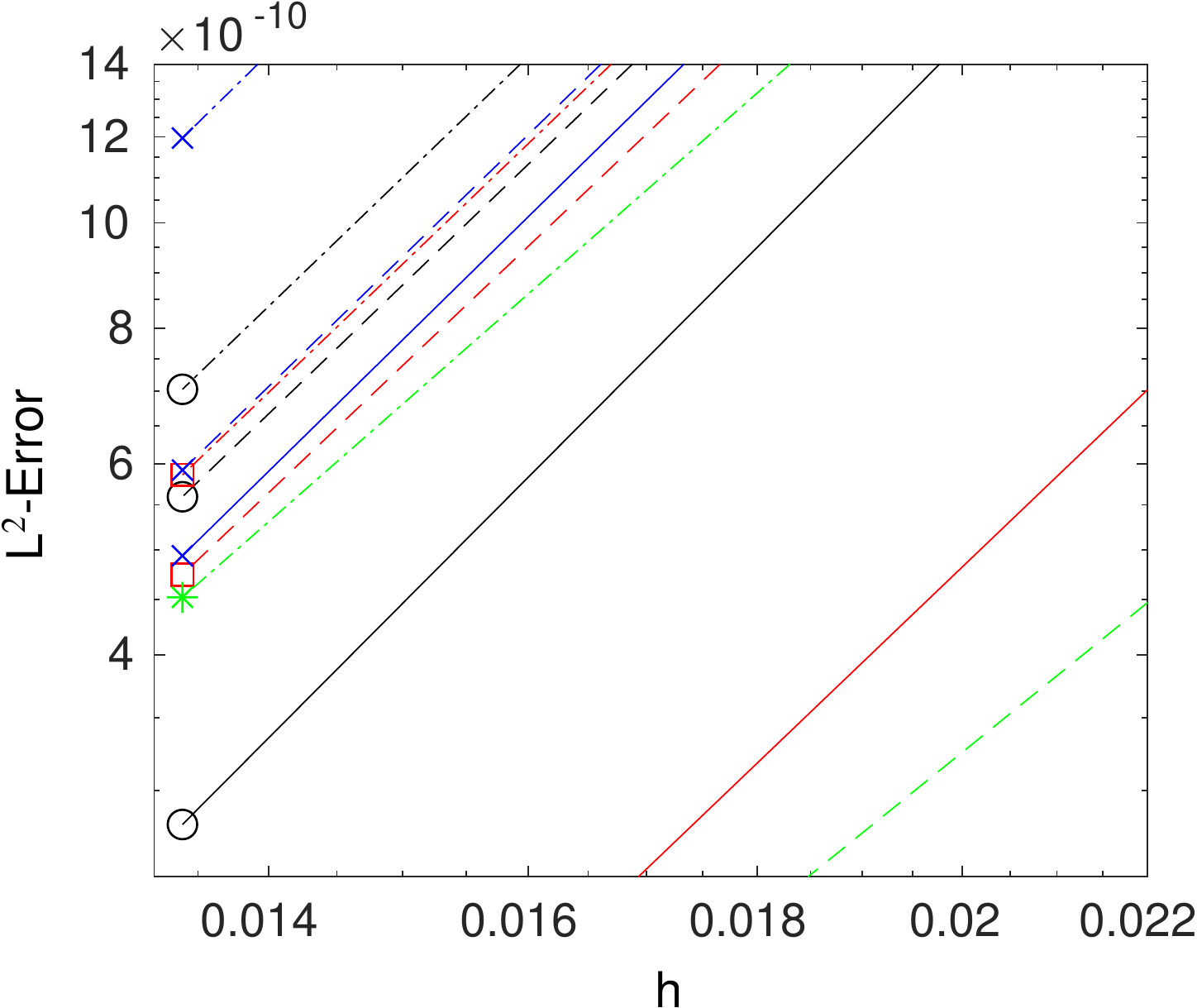} \\
\includegraphics[width=0.45\textwidth]{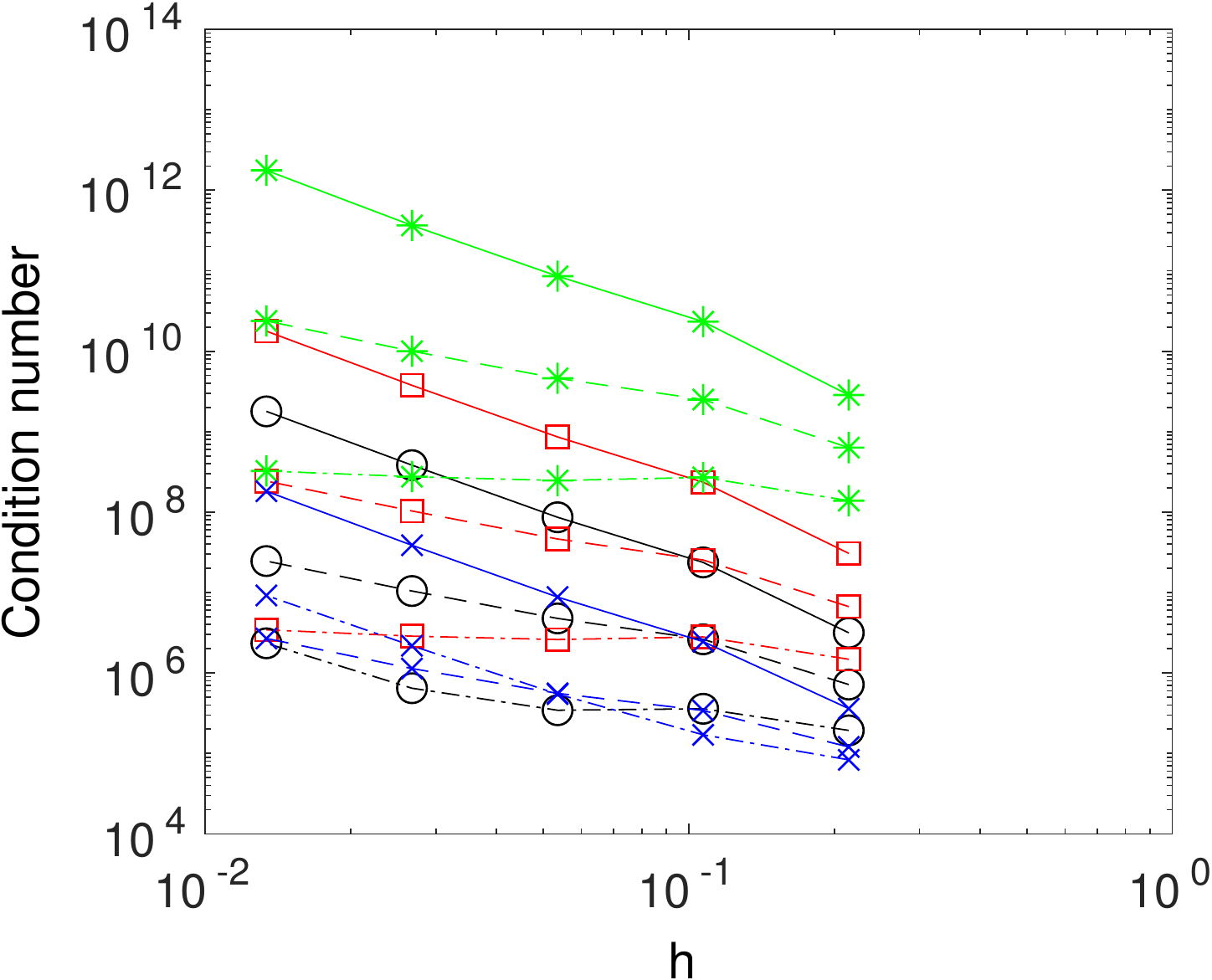}  \hspace{0.2cm}
\includegraphics[width=0.44\textwidth]{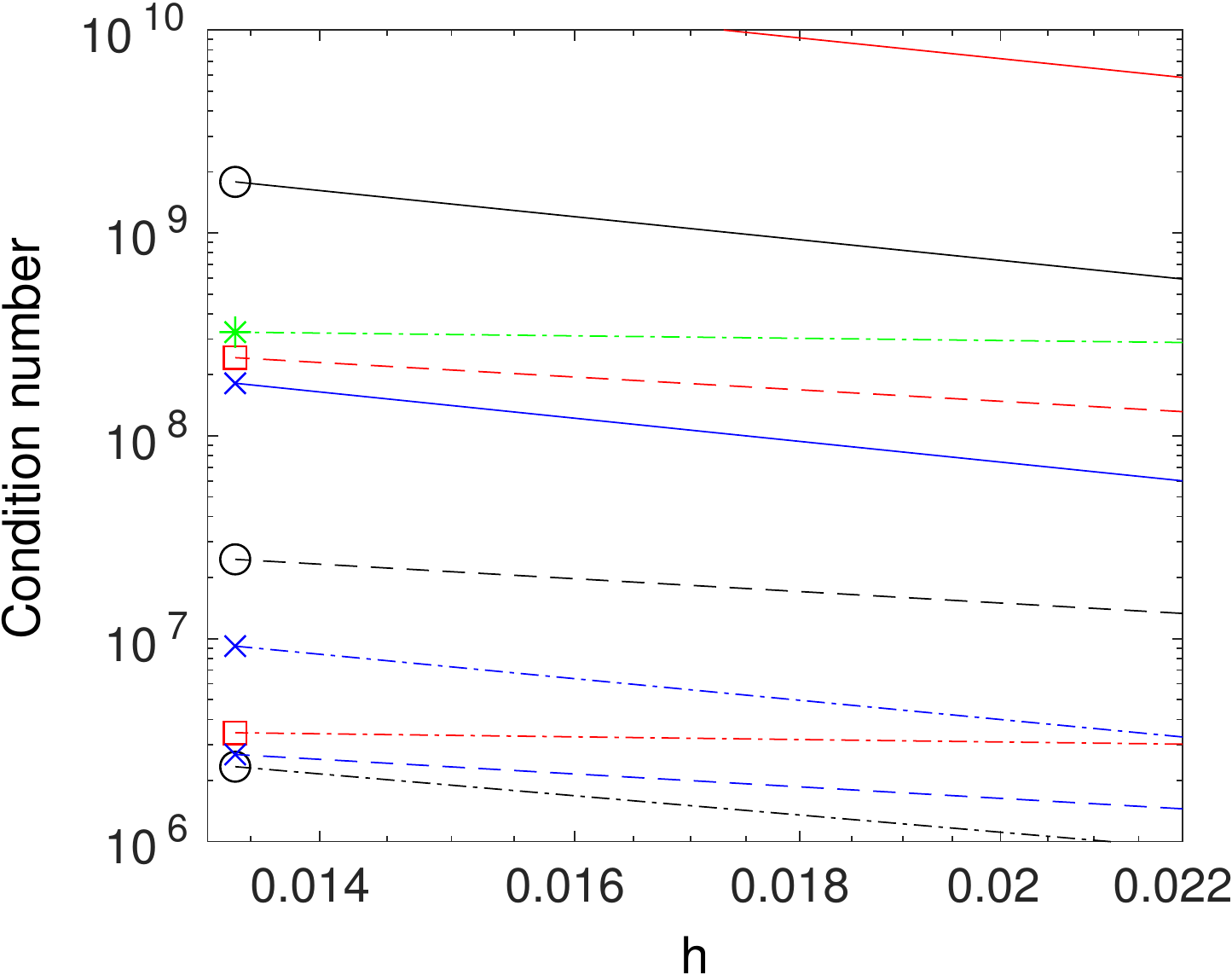}
\caption{The Laplace-Beltrami problem: 
 The error and the condition number versus mesh size $h$ varying the stabilization constants and the parameter $\gamma$ in \eqref{eq:sh}-\eqref{eq:sh-gamma}. Cubic elements are used, i.e. p=3. Solid lines: $\gamma=1$. Dashed lines: $\gamma=1/2$. Dashed dotted lines: $\gamma=0$. $c_{F,1}=c_{\Gamma,1}= c\cdot 0.07$, $c_{F,2}=c_{\Gamma,2}=c\cdot 0.004$, $c_{F,3}=c_{\Gamma,3}= c\cdot 2 \cdot 10^{-4}$. Blue crosses: $c=1$. Black circles: $c=10^{-1}$. Red squares: $c= 10^{-2}$. Green stars: $c= 10^{-4}$.  The figures to the right are close ups on part of the figures to the left so that the different curves can be distinguished. \label{fig:diffparamP3}}
\end{figure}

\subsection{The Mass Matrix}
We now consider the problem of finding $u_h \in V_h^p$ such that 
\begin{equation}\label{eq:Mh}
\int_{\Gamma_h}  u_h  v_h \ ds_h+ s_h(u_h,v_h)=\int_{\Gamma_h} f_h v_h \ ds_h,
\end{equation}
where $s_h$ is the stabilization term defined in~\eqref{eq:sh}. The interface $\Gamma$ is a circle of radius 1 centered at the origin and we generate a uniform triangular mesh with $h=h_{x_1}=h_{x_2}$ on the computational domain: $[-1.5, \ 1.5]\times [-1.5, \ 1.5]$. 
We let $f_h$ be
\begin{equation}
f_h=\frac{-(6x_1x_2(x_1^4 - 4x_1^2x_2^2 + x_2^4))}{(x^2 + y^2)^4}.
\end{equation}

We study the condition number of the mass matrix using different
degrees of polynomials, $p$, and different stabilization
terms. Fig.~\ref{fig:errordiffelemM} shows that the stabilization term
we propose results in optimal convergence rates in the $L^2$-norm and
optimal scaling of the condition number ($\mathcal{O}(1)$) of the mass
matrix as it did for the Laplace-Beltrami operator. 
As before, the parameters used in \eqref{eq:sh}-\eqref{eq:sh-gamma} influence the magnitude of the error and the condition number. For the results shown in Fig.~\ref{fig:errordiffelemM} we have used  $c_{F,j}=c_{\Gamma,j}=0.03\cdot 20^{-j}$, $j=1, \cdots, p$. Other choices could give smaller condition numbers without increasing the error significantly and also give better scaling with respect to the polynomial degree of the condition number.

In Fig.~\ref{fig:errordiffstabP1M} and~\ref{fig:errordiffstabP3M} we
compare different stabilization terms for linear and cubic elements, respectively. 
The proposed stabilization with $\gamma=1$ is compared with using no stabilization i.e.,  $c_{F,j}=c_{\Gamma,j}= 0$, a pure face stabilization i.e. $c_{\Gamma,j}= 0$ with $\gamma=1$, a pure interface stabilization i.e. $c_{F,j}=0$ with $\gamma=1$, and the stabilization term \eqref{eq:normalgradel} with $\alpha=1$.
In Fig.~\ref{fig:errordiffstabP1M} we also show results with two different constants in the stabilization terms.  We see that only the proposed stabilization and the stabilization term \eqref{eq:normalgradel} yield optimal scaling of the condition number ($\mathcal{O}(1)$) and both the error and the condition number are very similar for these two stabilization terms.

For cubic elements we show results with three different constants in the stabilization terms, see Fig.~\ref{fig:errordiffstabP3M}.  Decreasing the constants in the stabilization decreases the magnitude of the error but the condition number increases. When the stabilization is too weak the convergence rate is not optimal.

Our observations are similar to what we observed for the Laplace-Beltrami operator. For higher order elements the linear systems are severely ill-conditioned without stabilization and the behavior of the error is oscillatory and the convergence rate will depend on how the geometry cuts the background mesh. We also see that neither the pure face stabilization nor a pure interface stabilization give enough control of the condition number and that we need, as we propose, the combination of them.  

\begin{figure}\centering
\includegraphics[width=0.4\textwidth]{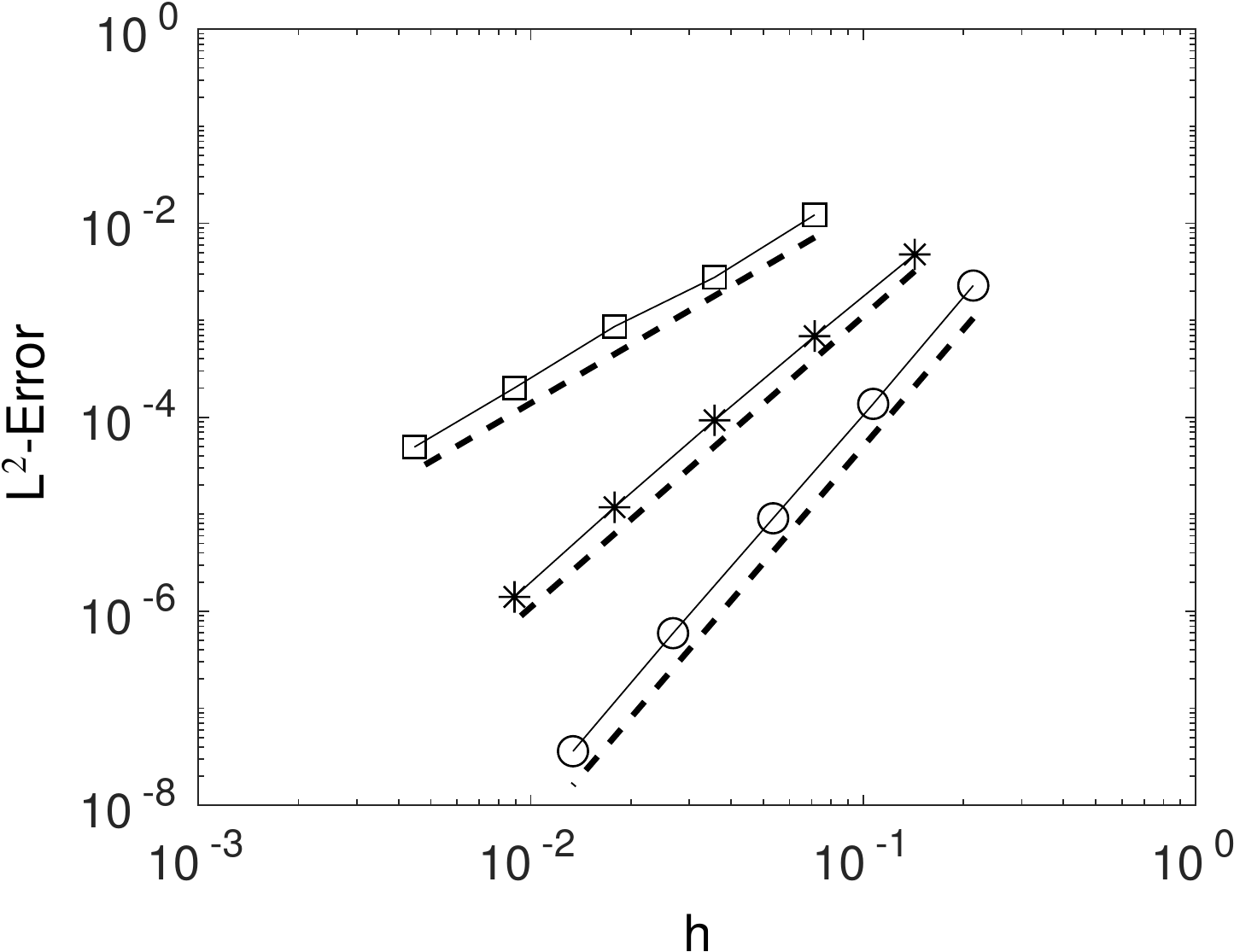} \hspace{0.2cm}
\includegraphics[width=0.4\textwidth]{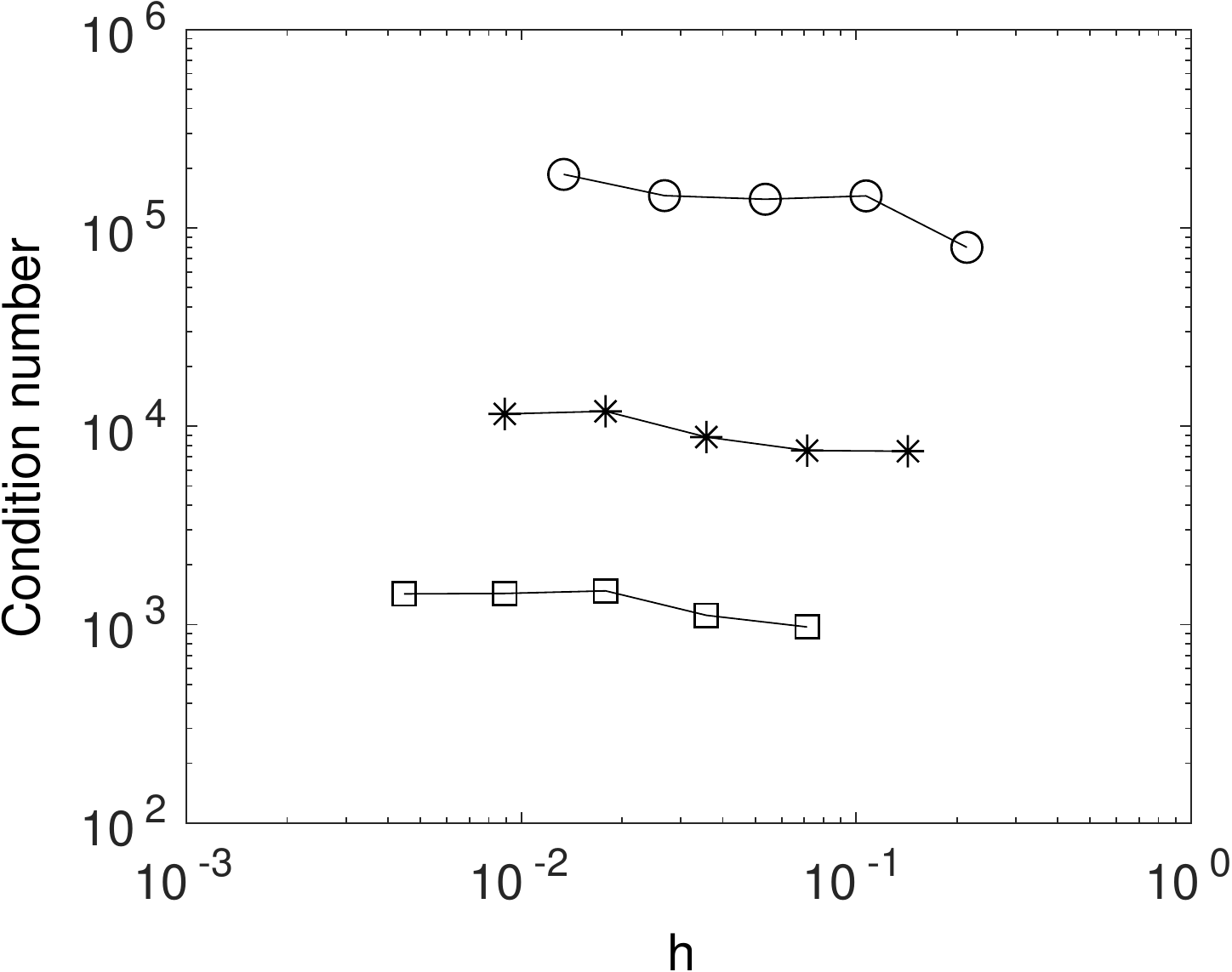}
\caption{The mass matrix: 
 The error and condition number versus mesh size $h$ for different degrees  of polynomials, p,  in the space discretization. Squares: p=1. Stars: p=2. Circles: p=3. Left: The error measured in the $L^2$-norm versus mesh size $h$.  The dashed lines are indicating the expected rate of convergence and are proportional to  $h^{p+1}$. Right: The condition number versus mesh size $h$.  \label{fig:errordiffelemM}}
\end{figure} 

\begin{figure}\centering
\includegraphics[width=0.4\textwidth]{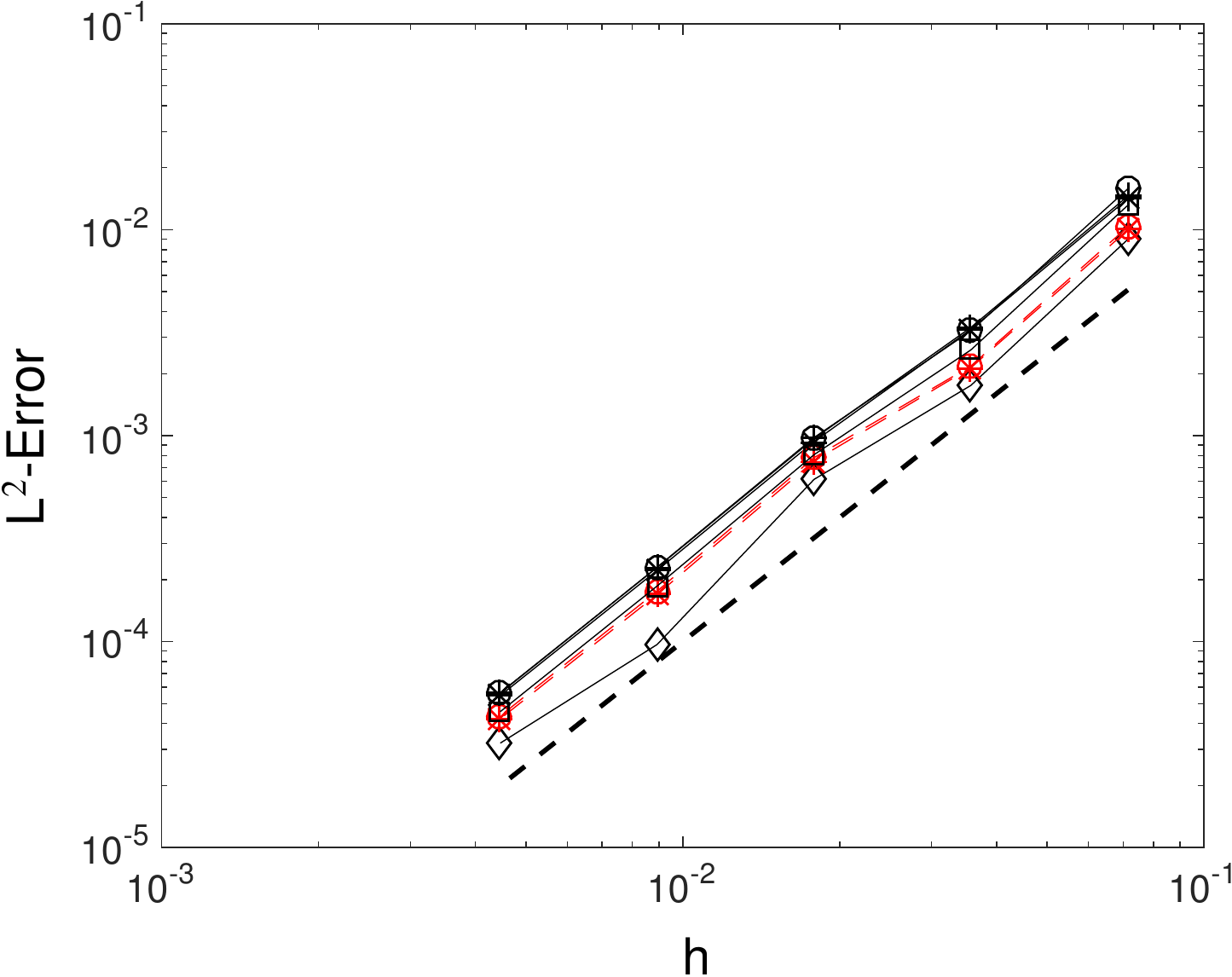}
\hspace{0.3cm}
\includegraphics[width=0.4\textwidth]{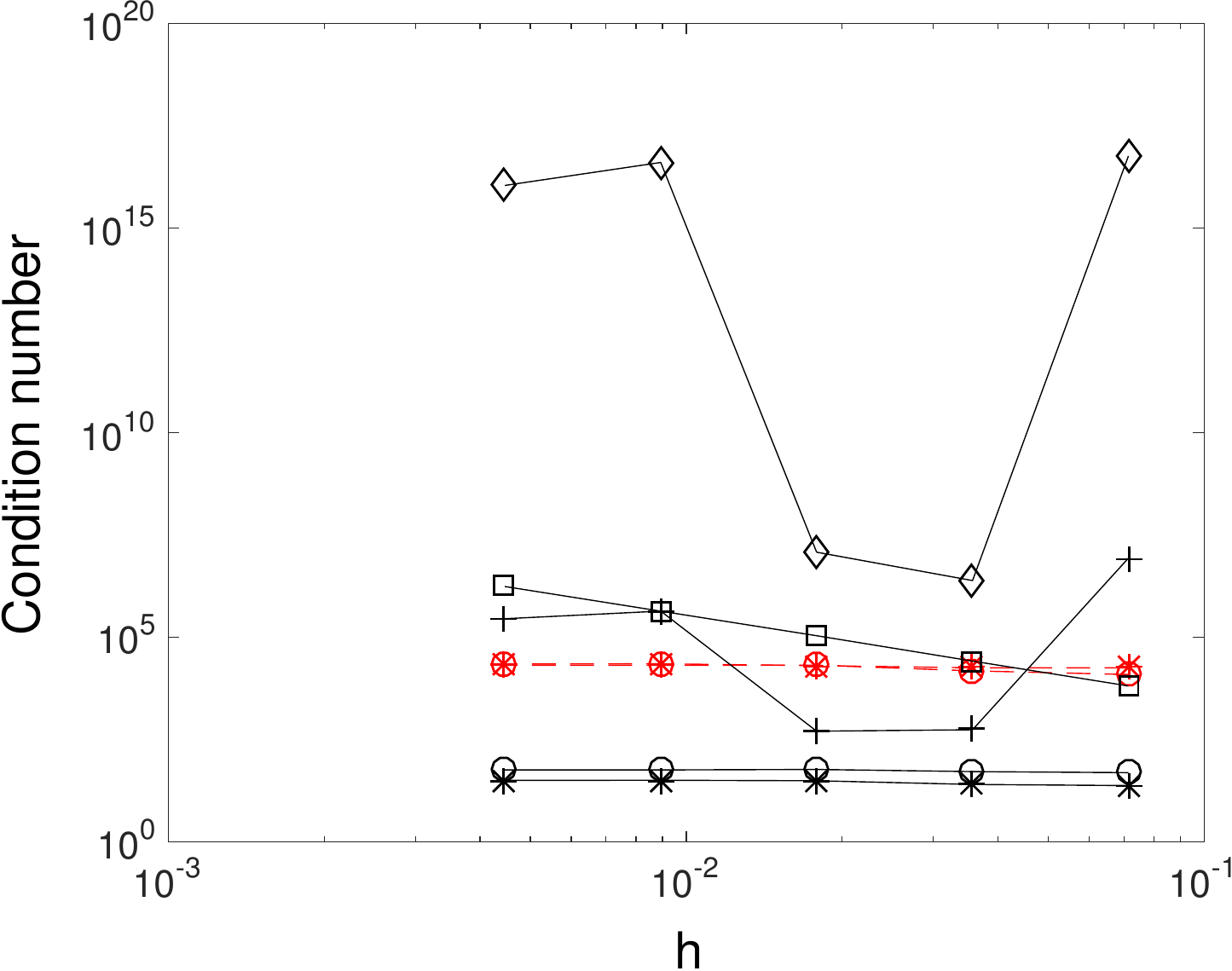}
\caption{The mass matrix:
The error and condition number versus mesh size $h$ using different stabilization terms. Linear elements are used, i.e. p=1. 
Diamonds: no stabilization is added. Circles: the proposed stabilization with $\gamma=1$. Squares: the pure face stabilization. Stars: the stabilization term \eqref{eq:normalgradel} with $\alpha=1$.
Left: The error measured in the $L^2$-norm versus mesh size $h$.  The dashed line is indicating the expected rate of convergence and is proportional to $h^{2}$.  Right: The condition number versus mesh size $h$. Symbols connected with solid lines: $c_{F,j}=10^{-1}$, $c_{\Gamma,j}=10^{-1}$ and $c_{\mcT}=10^{-1}$ in~\eqref{eq:normalgradel}. Symbols connected with dashed lines: $c_{F,j}=c_{\Gamma,j}=10^{-4}$, and $c_{\mcT}=10^{-4}$.
 \label{fig:errordiffstabP1M}}
\end{figure} 

\begin{figure}\centering
\includegraphics[width=0.4\textwidth]{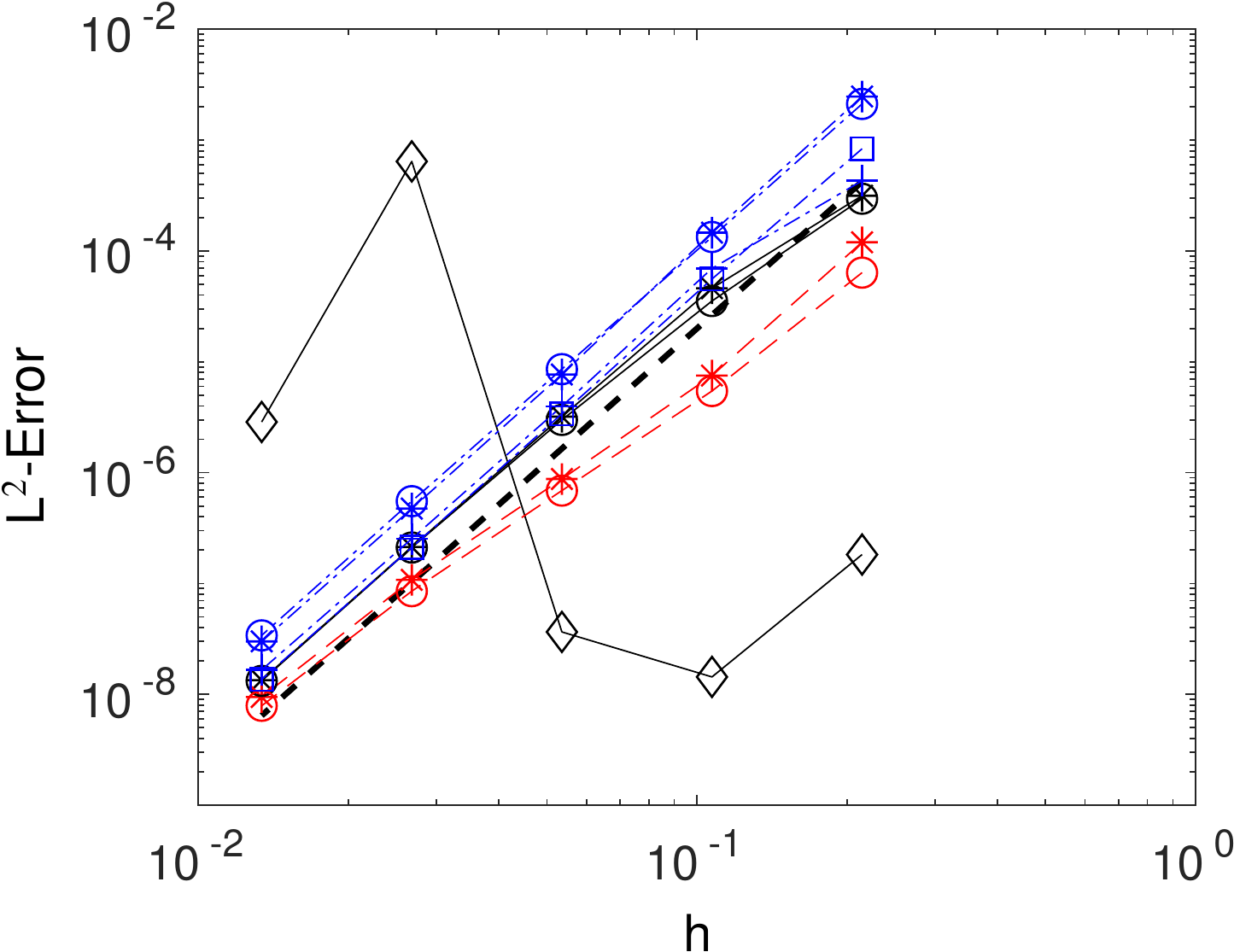} 
\hspace{0.3cm}
\includegraphics[width=0.4\textwidth]{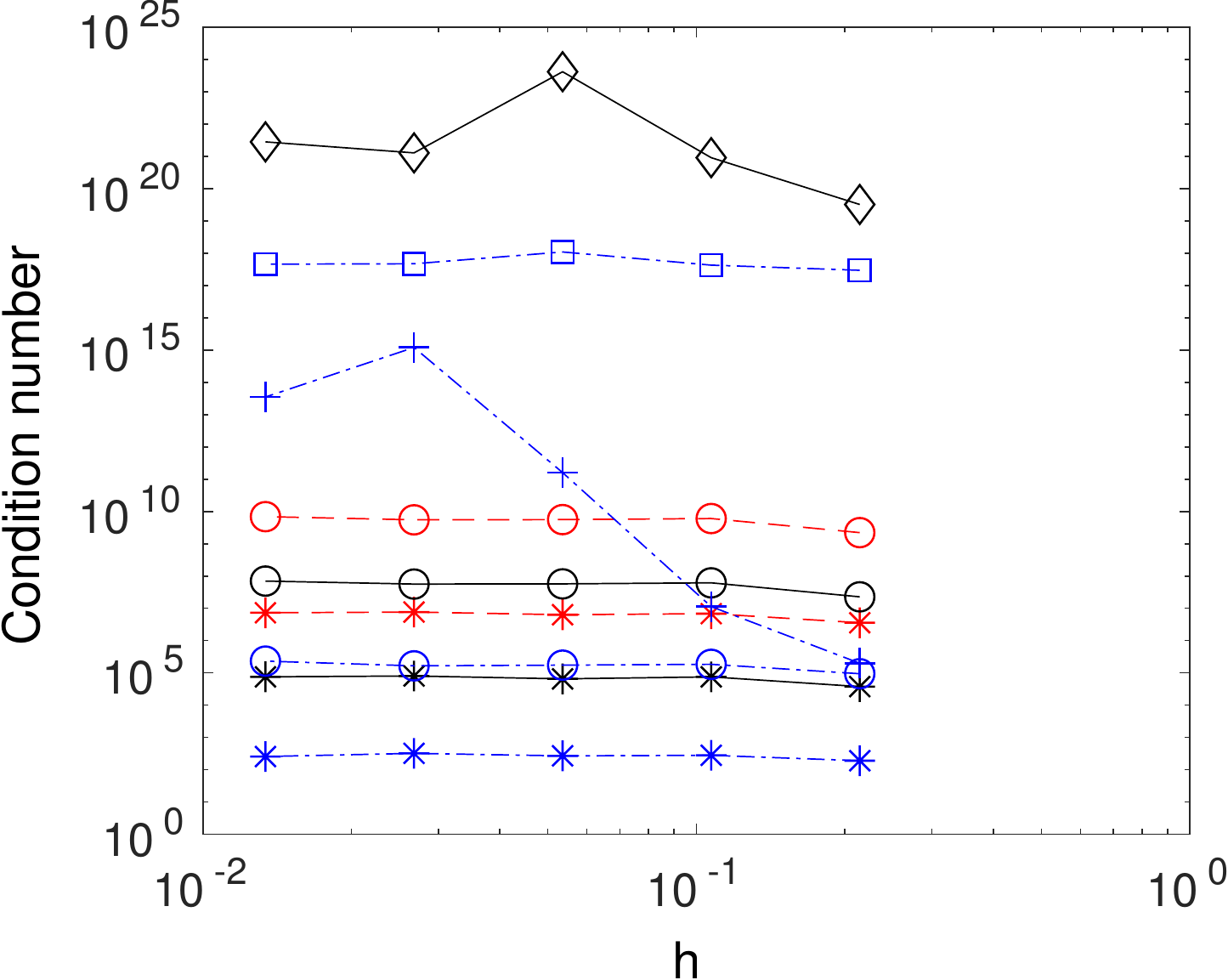}
\caption{The mass matrix: The error and condition number versus mesh
  size $h$ using different stabilization parameters. Cubic elements
  are used, i.e. p=3. Diamonds: no stabilization is added. Circles: the proposed stabilization with $\gamma=1$. Squares: the pure face stabilization. Stars: the stabilization term \eqref{eq:normalgradel} with $\alpha=1$. 
Left: The error measured in the $L^2$-norm versus mesh size $h$.  The dashed line is indicating the
  expected rate of convergence and is proportional to $h^{4}$.  Right:
  The condition number versus mesh size $h$.
 Symbols connected with dashed dotted lines (blue) have larger constants in the stabilization terms than symbols connected with solid lines (black) and symbols connected with dashed lines (red) which have the smallest constants in the stabilization terms. 
\label{fig:errordiffstabP3M}}
\end{figure}

\subsection{The Mean Curvature Vector} \label{sec:meancurv}
We consider the interfaces given by the zero contour of the following level set functions
\begin{align}
\phi&=\frac{x^2}{0.64}+y^2-0.25  \quad &&\textrm{Ellipse} \\
\phi&=\lp z^2+\lp (x^2+y^2)^{1/2}-1 \rp^2\rp^{1/2}-0.5  \quad &&\textrm{Torus} \\
\phi&= 0.02-((x^2 +y^2 -0.75)^2 +(z^2 -1)^2)((y^2 +z^2 -0.75)^2 + && \nonumber \\
&+(x^2 -1)^2)((z^2 +x^2 -0.75)^2 +(y^2 -1)^2) \quad &&\textrm{Deco-cube} \\
\end{align}
We generated meshes independently of the position of the given interface. We define the mesh size by $h=1/N^\frac{1}{d}$ where $N$ denotes the total number of nodes.  We construct an approximate level set function $\phi_h$ using the nodal interpolant $\pi_h^1 \phi$ on the background mesh and let the interface be the zero level set of $\phi_h$. We use linear Lagrange basis functions.  
Given the discrete coordinate map $x_{\Gamma_h} : \Gamma_h \ni x \mapsto x \in {\bf R}^d$ we want to find the stabilized discrete mean curvature vector $\mcu_h \in [V_h^1]^d$  such that
\begin{equation}\label{eq:meancurvdisc}
(\mcu_h, v_h)_{\Gamma_h}+ s_h(\mcu_h,v_h)= (\nabla_{\Gamma_h} x_{\Gamma_h}, \nabla_{\Gamma_h} v_h)_{\Gamma_h},
\end{equation}
where $s_h$ is a stabilization term. The torus and the deco-cube are examples from~\cite{HanLarZah15}. In~\cite{HanLarZah15} we proved that using the ghost penalty stabilization~\eqref{eq:stabface} 
the method \eqref{eq:meancurvdisc} is a convergent first order method with respect to the $L^2$-norm. Therefore, we also expect to obtain first order accurate approximations of the mean curvature vector using the proposed stabilization term.

We compare the approximation from~\eqref{eq:meancurvdisc}  with the vector  $\mcu=-\lp \nabla \cdot n \rp n$, where $n=\frac{\nabla \phi}{|\nabla \phi|}$. In Fig.~\ref{fig:errormeancurv} we show the error in the $L^2$-norm using the proposed stabilization with $\gamma=0$ and the stabilization term~\eqref{eq:normalgradel} developed in~\cite{BurHanLarMas16} and \cite{GraLehReu16} with $\alpha=-1$. We emphasize that the analysis in \cite{HanLarZah15} requires control over the jumps in the 
normal derivatives across edges on $\Gamma_h$, which can be controlled using 
(\ref{eq:stabface}), and without appropriate stabilization we don't expect any convergence in $L^2$. Note also that we do not expect the element normal gradient stabilization to provide enough stabilization to ensure convergence of the mean curvature vector since some additional tangential control is required and thus this is a situation where face stabilization comes in naturally. In practice, we note that depending on the stabilization parameter we see that we are able to obtain first order convergence with both stabilization terms. However, the magnitude of the errors are always smaller using the proposed stabilization compared to using the stabilization term~\eqref{eq:normalgradel}. The magnitude of the error using the pure face stabilization is similar to the results with the proposed stabilization and not shown here. Furthermore, the condition number is better using the proposed stabilization term.  We emphasize that when higher order elements than linear are used it is of vital importance to use the proposed stabilization term to control the condition number since the pure face stabilization does not control the condition number and therefore the error may be dominated by roundoff errors. Full details of high order approximations of the mean curvature vector using this approach are presented in~\cite{FrZa}. 

In this example the face stabilization is of importance and in such cases the proposed stabilization is preferable compared to the stabilization term~\eqref{eq:normalgradel}.

\begin{figure}
\begin{center} 
\includegraphics[width=0.48\textwidth]{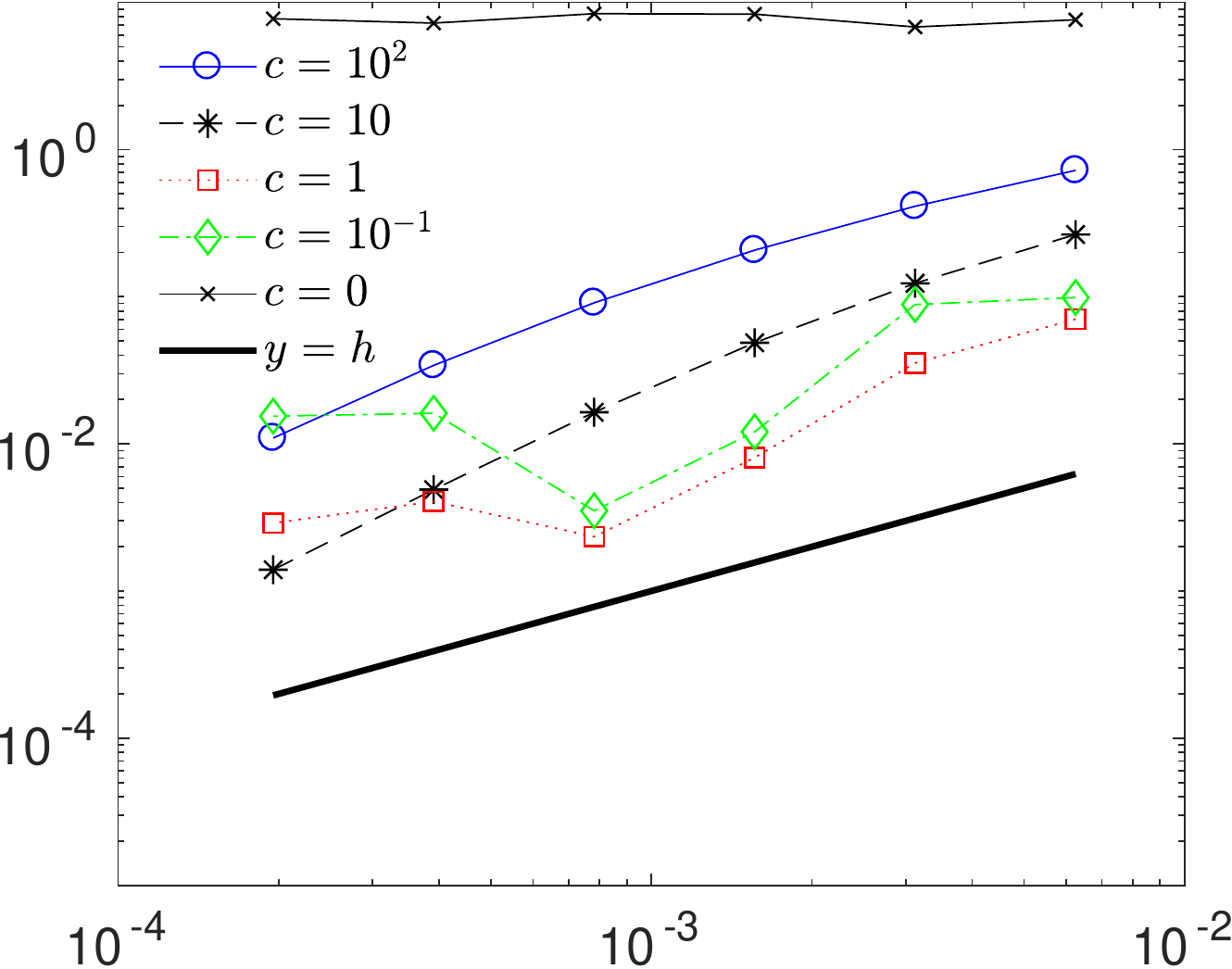} 
\includegraphics[width=0.48\textwidth]{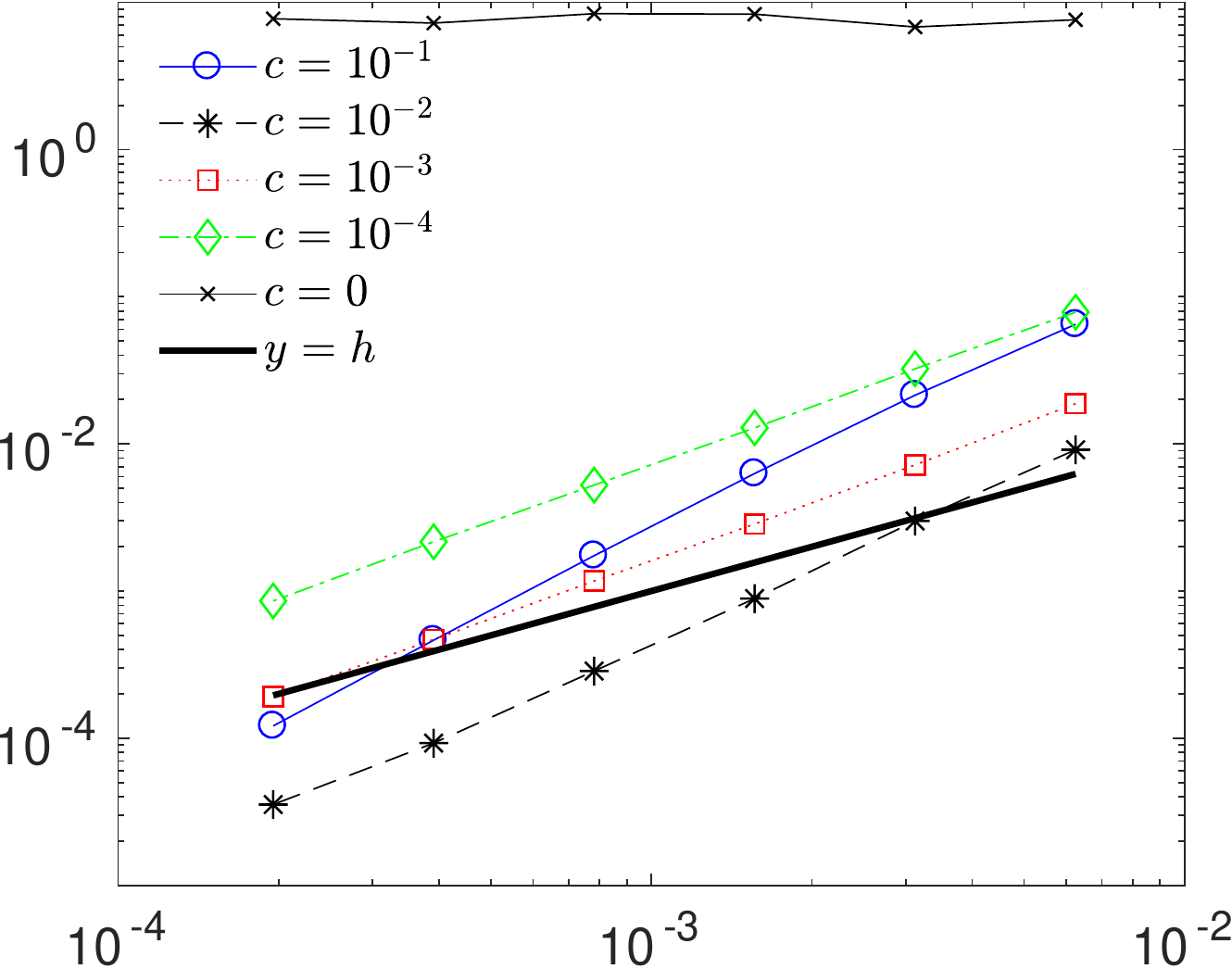} \\
\vspace{0.5cm}
\includegraphics[width=0.48\textwidth]{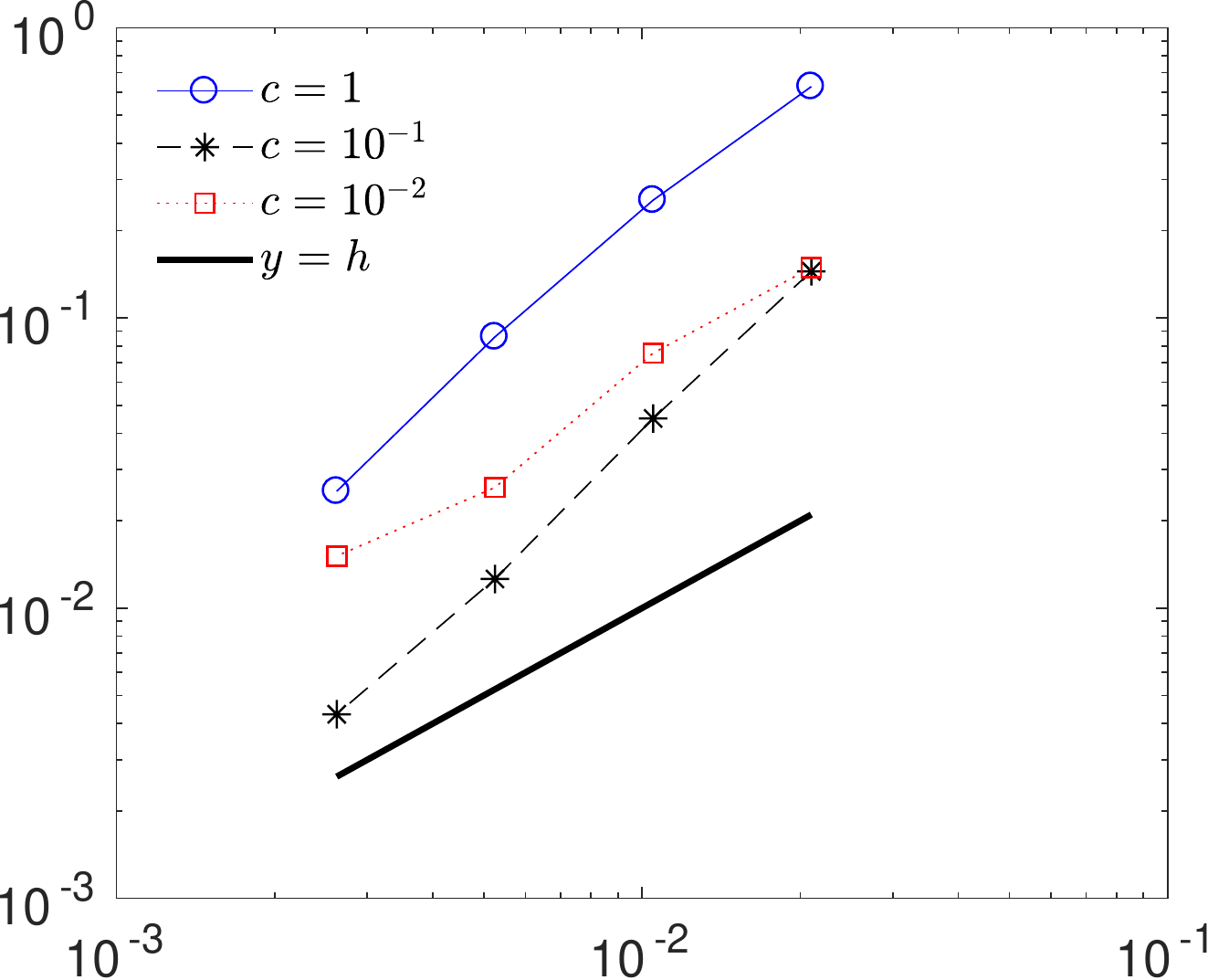} 
\includegraphics[width=0.48\textwidth]{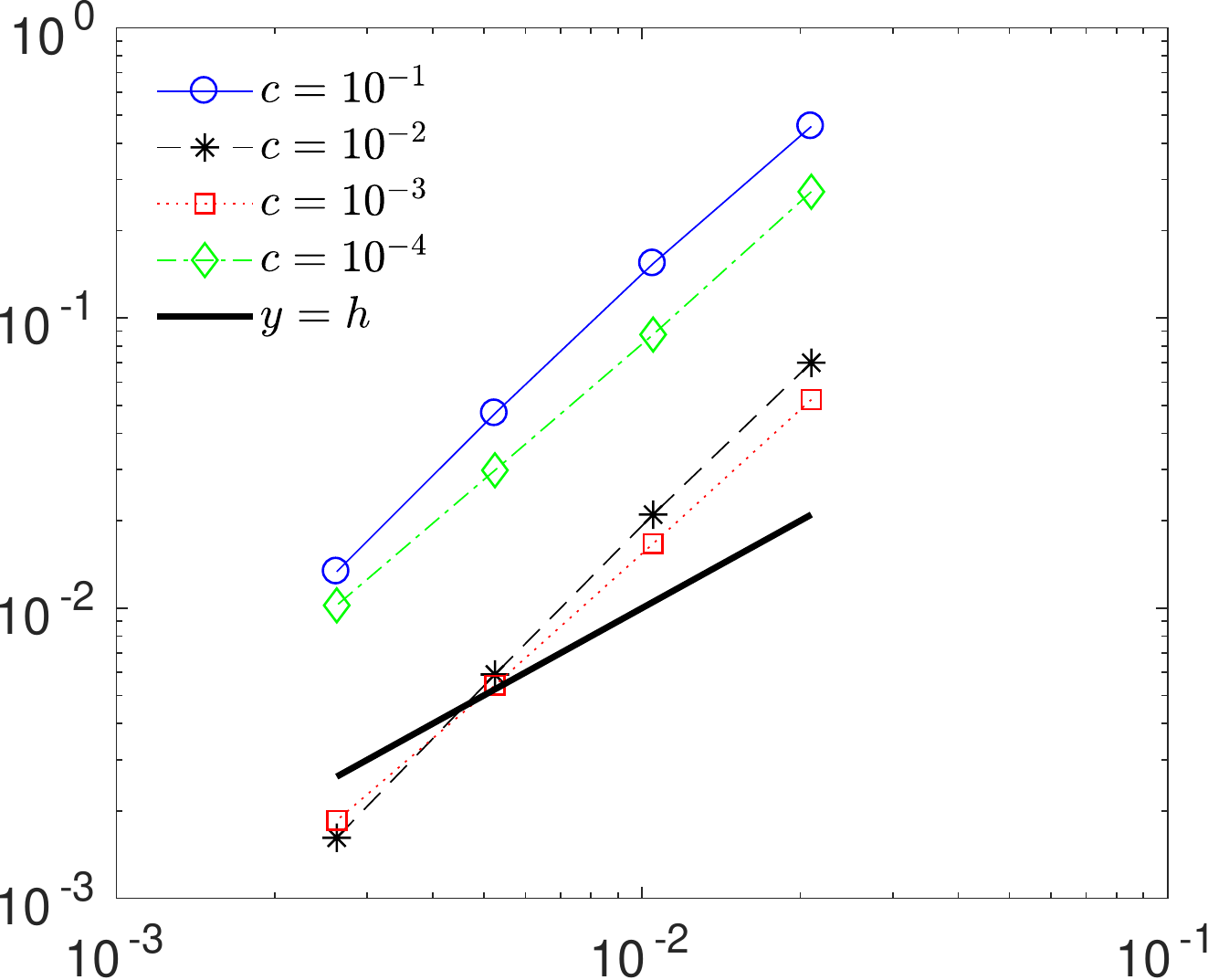}\\
\vspace{0.5cm}
\includegraphics[width=0.48\textwidth]{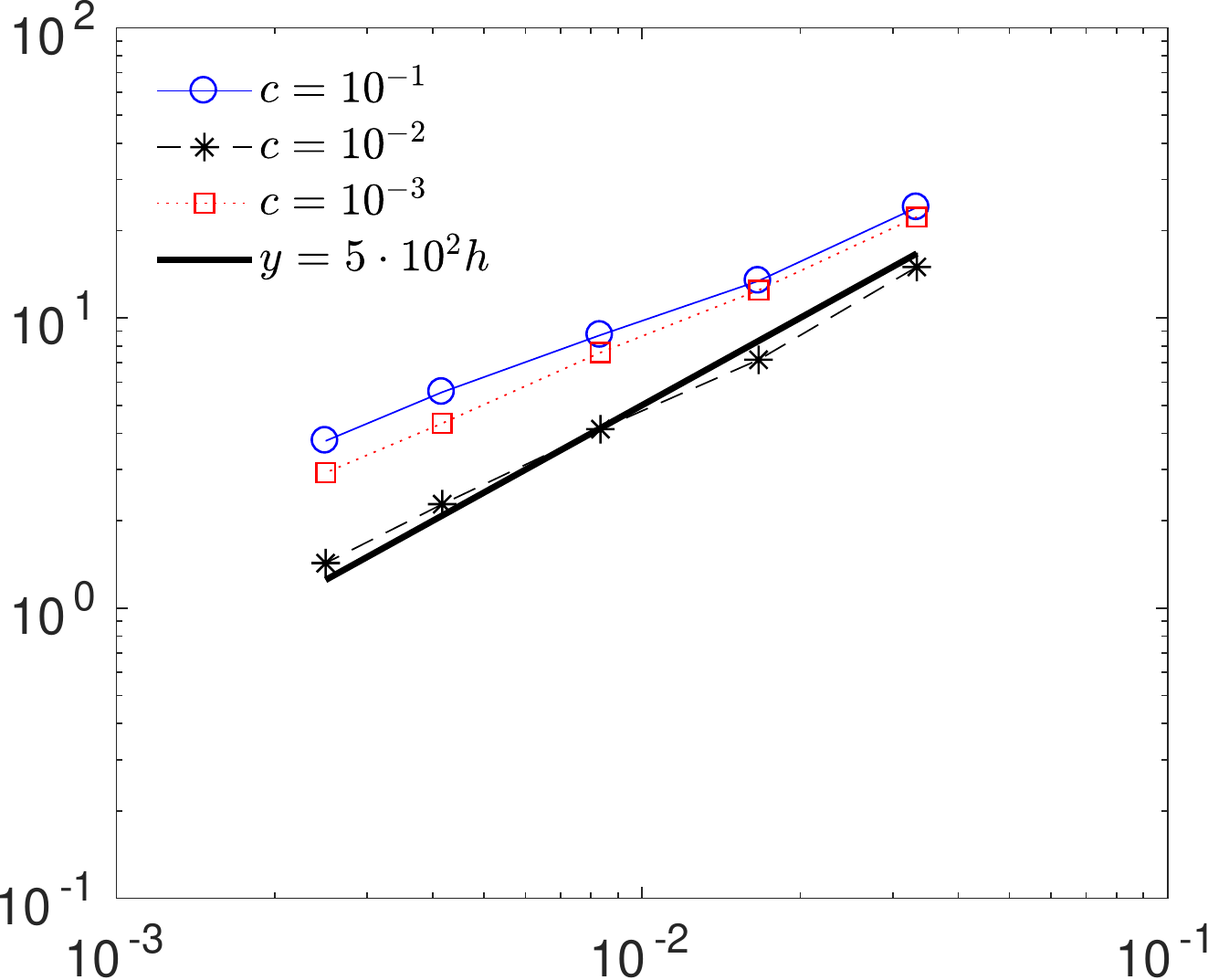} 
\includegraphics[width=0.48\textwidth]{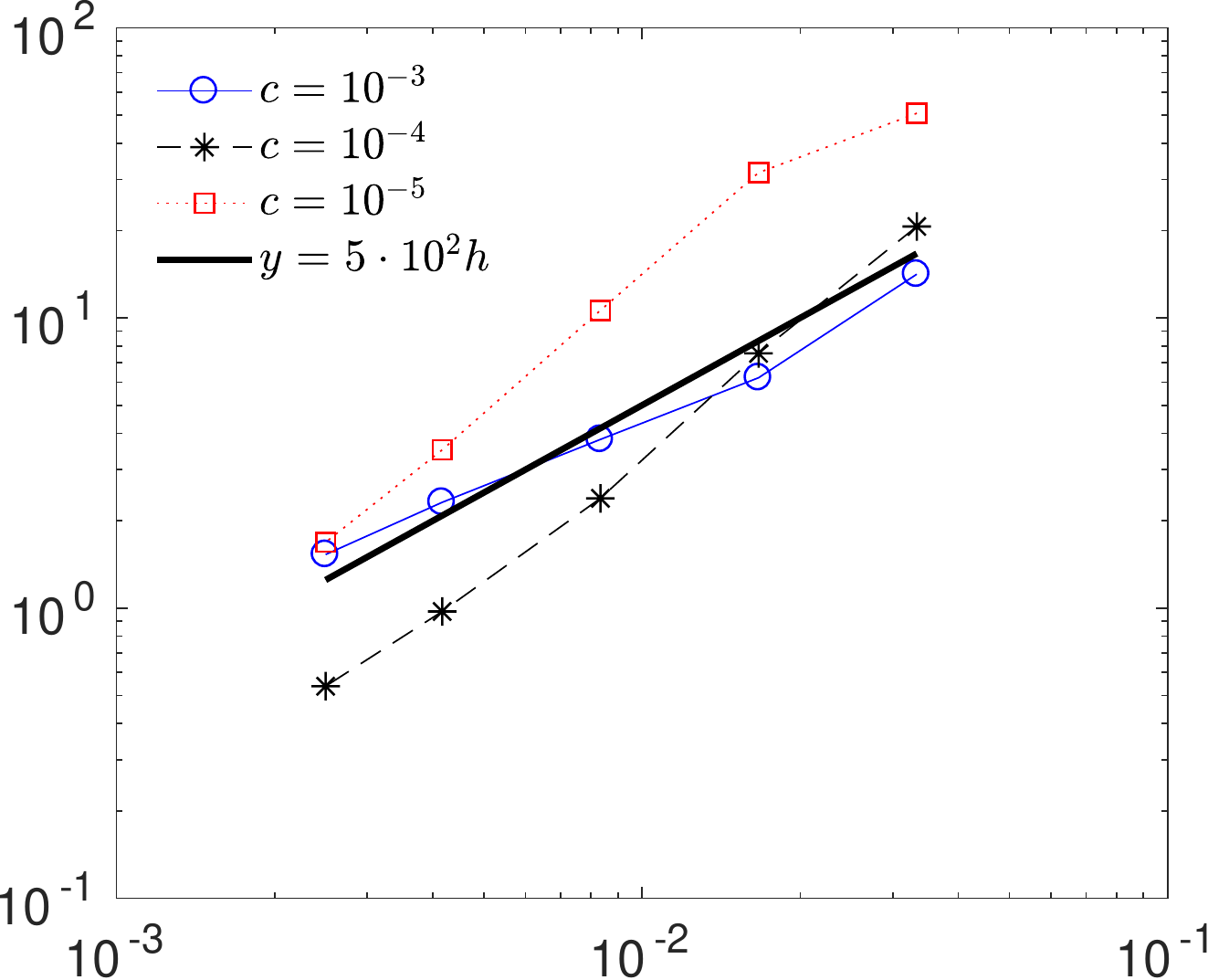}
\end{center}
\caption{The mean curvature vector: The error measured in the
  $L^2$-norm versus mesh size $h$. Linear elements are used, i.e. p=1. Left: The stabilization term~\eqref{eq:normalgradel} developed in~\cite{BurHanLarMas16} and \cite{GraLehReu16} with $\alpha=-1$ and $c_{\mcT}=c$. Right: The proposed stabilization with $\gamma=0$ and $c_F=c_\Gamma=c$. Top: Ellipse. Middle: Torus. Bottom: Deco-cube. \label{fig:errormeancurv}}
\end{figure}

\section{Discussion}
A strategy to control the condition number of the system matrix resulting from cut finite element discretizations, independently of how elements in the background mesh are cut by the geometry, is to add a stabilization term to the weak formulation.   For higher order elements than linear the linear systems resulting from cut finite element discretizations without stabilization are extremely ill-conditioned and it is not obvious how to precondition such linear systems. Therefore, the strategy of adding a stabilization term to the weak form has become a popular alternative to cure this ill-conditioning. 
The stabilization term most frequently used together with CutFEM is the ghost penalty stabilization~\cite{Bu10} which acts only on the element faces.  We have seen in this work that when higher order elements than linear are used in the cut finite element discretizations of surface PDEs the condition number can not be controlled by only adding such a face stabilization term.  We have proposed a remedy where we keep the face stabilization but also add a stabilization term acting on the interface/surface and prove that this new stabilization term controls the condition number for both linear as well as higher order elements.

As discussed in the introduction and in Section~\ref{sec:numexpS} a different stabilization term, the normal gradient stabilization~\eqref{eq:normalgradel} which is evaluated on each cut element, has recently been proposed in \cite{BurHanLarMas16} and \cite{GraLehReu16}. For surface PDEs this stabilization also controls the condition number for linear as well as higher order elements.  This term results in less non-zero entries in the stiffness matrix than the stabilization term we propose since it is local while the face stabilization term involves neighboring elements.  The implementation of~\eqref{eq:normalgradel} is also easier than the stabilization we propose. However, there are a couple of reasons why we choose to keep the ghost penalty term and propose a stabilization term based on this term. 
\begin{itemize}
\item The element stabilization~\eqref{eq:normalgradel} cannot be used in the discretization of bulk problems since it would destroy the convergence order while the ghost penalty stabilization could be used without destroying the convergence order and with control of the condition number of the linear systems.  Many applications include coupled bulk-surface problems and for those problems the ghost penalty stabilization is needed. In such applications, the effect of the proposed stabilization on the number of non-zero elements is very small and since the face stabilization is needed for the bulk problem the implementation of the proposed stabilization is also easy as only the stabilization term acting on the surface has to be added.
\item The ghost penalty stabilization has shown to be of importance also for other reasons than controlling the condition number. In~\cite{BuHaLaZa15}, for example, we show that by adding such a face stabilization we get a stable discretization for convection dominated problems  on surfaces and no other stabilization term such as for example a SUPG term is needed. In~\cite{HanLarZah15} we show that with the same face stabilization a method for computing the mean curvature vector for piecewise linear surfaces based on the Laplace--Beltrami operator enjoys first order convergence in the $L^2$-norm while, in general, without the stabilization convergence cannot be expected. In Section~\ref{sec:meancurv} we compute the mean curvature vector of three surfaces  and show that we obtain better accuracy with the proposed stabilization than with the stabilization~\eqref{eq:normalgradel}. Our method for computing the mean curvature vector can also be extended to give high order approximations~\cite{FrZa}. 
For time dependent problems on surfaces and coupled-bulk surface problems adding the ghost penalty stabilization enables to directly approximate space-time integrals in space-time cut finite element formulations using quadrature rules for the integrals over time \cite{HanLarZah16, Zah17}. This makes the implementation of space-time CutFEM convenient when higher order elements are used, in particular for coupled bulk-surface problems. The proposed stabilization thus also has all these advantages
 and in addition gives control of the condition number for linear as well as higher order elements. 
\item The proof of the condition number estimate is simple when the proposed stabilization term is used.  The main step in the proof is to show that one covers a non empty set inside an element by going in the normal direction of the surface from the part of the surface that intersects the element.  For the proposed stabilization one only needs to show that such a non-empty set exist for elements that have a large intersection with the surface.   However, to obtain the desired Poincar\'e inequality with the stabilization term~\eqref{eq:normalgradel} one needs to show that such a set exist in each element and in particular that this set is large enough~\cite{BurHanLarMas16, GraLehReu16}. We use the simplified structure of the proof to extend our results to the case of general codimension embeddings, see Section~\ref{sec:codim} where we present the minor modifications of the proof to Lemma \ref{lem:Poincare-Th} to verify the generalized version (\ref{eq:property5-cd}) of the Poincar\'e inequality P5.
\item We also note that the stabilization term~\eqref{eq:normalgradel} requires the computation of the normal vector inside an element. When the interface is implicitly represented by a higher dimensional function, for example a distance function, this normal vector can be taken to be the normal vector of the level sets of the function defining the interface implicitly. However, when an explicit representation is used for the surface a normal vector inside an element is not immediately available but normal vectors on the surface, used in the proposed stabilization, are available. 
\end{itemize}

\end{document}